%% file: main_structuredPrior.tex
\documentclass[a4paper,11pt]{article}
\usepackage{authblk}
\usepackage[british]{babel}
\usepackage[mono=false]{libertine}
\usepackage{setspace}

\usepackage[utf8]{inputenc} % allow utf-8 input
\usepackage[T1]{fontenc}    % use 8-bit T1 fonts
\usepackage{url}            % simple URL typesetting
\usepackage{booktabs}       % professional-quality tables
\usepackage{amsfonts}       % blackboard math symbols
\usepackage{nicefrac}       % compact symbols for 1/2, etc.
\usepackage{microtype}      % microtypography
\usepackage{mathtools}
\usepackage{amsmath,amsfonts,amssymb,amsthm,mathrsfs,bbm}
\usepackage{latexsym,amscd,amsbsy,dsfont}
\usepackage{mathtools}
\usepackage{dsfont}
\usepackage{algorithmic}
\usepackage{algorithm2e}
\usepackage{color}
\usepackage[usenames, dvipsnames]{xcolor}
\usepackage{graphicx}
\usepackage{epstopdf}
\usepackage{float}
\usepackage{cases}
\usepackage[top=2.5cm, bottom=2.5cm, left=2cm, right=2cm]{geometry}
\usepackage{caption}
\usepackage{physics}
\usepackage{tensor}
\usepackage{tikz}
\usepackage{subcaption}
\usetikzlibrary{arrows}
\usetikzlibrary{decorations.markings}
\usetikzlibrary{fit}
\usepackage[most]{tcolorbox}
\tcbuselibrary{theorems}
\usepackage{empheq}
\usetikzlibrary{bayesnet}
\tikzset{middlearrow/.style={decoration={markings,mark= at position 0.5 with {\arrow{#1}} ,},postaction={decorate}}}
\usetikzlibrary{arrows}

\usepackage{xr-hyper} % or use \usepackage{xr}
\usepackage[pdfpagemode=UseNone,bookmarksopen=false,colorlinks=true,urlcolor=RoyalBlue,citecolor=YellowOrange,citebordercolor=RoyalBlue,linkcolor=RoyalBlue]{hyperref}
\usepackage{bookmark}
\input{settings_custom}
\begin{document}
%\newpage

\title{The spiked matrix model with generative priors}
\author{Benjamin Aubin$^\dagger$, Bruno Loureiro$^\dagger$, Antoine Maillard$^\star$ \\ 
Florent Krzakala$^\star$, Lenka Zdeborov{\'a}$^\dagger$}
\date{
$\dagger$ \textit{Institut de Physique Th\'eorique \\
CNRS \& CEA \& Universit\'e Paris-Saclay, Saclay, France}\\
$\star$ \textit{Laboratoire de Physique Statistique\\
CNRS \& Sorbonnes Universit\'es \& \\
\'Ecole Normale Sup\'erieure, PSL University, Paris, France}\\
\vspace{1cm}
}

\maketitle

\begin{abstract}
Using a low-dimensional parametrization of signals is a generic and powerful way to enhance performance in signal processing and statistical inference. A very popular and widely explored type of dimensionality reduction is sparsity; another
type is generative modelling of signal distributions. Generative models based on neural networks, such as GANs or variational auto-encoders,
are particularly performant and are gaining on applicability. In this paper we study spiked matrix models, where a low-rank matrix is observed through a noisy
channel. This problem with sparse structure of the spikes has attracted
broad attention in the past literature. Here, we replace the sparsity assumption by
generative modelling, and investigate the consequences on statistical and
algorithmic properties. We analyze the Bayes-optimal
performance under specific generative models for the spike. In contrast with
the sparsity assumption, we do not observe regions of parameters where
statistical performance is superior to the best known algorithmic
performance. We show that in the analyzed cases the approximate
message passing algorithm is able to reach optimal performance. We also design
enhanced spectral algorithms and analyze their performance and
thresholds using random matrix theory, showing their superiority to the
classical principal component analysis. We complement our theoretical
results by illustrating the performance of the spectral algorithms when the spikes come from real datasets.
\end{abstract}

\newpage
\makeatletter
\def\l@subsubsection#1#2{}
\makeatother

{
  \hypersetup{linkcolor=black}
\begin{spacing}{0.9}
  \tableofcontents
  \end{spacing}
}

\newpage
\section{Introduction}
\input{files/introduction.tex}

\section{Analysis of information-theoretically optimal estimation}
\input{files/information_theory.tex}

\section{Approximate message passing with generative priors}
\input{files/amp.tex}

\section{Spectral methods for generative priors}
\label{sec:spectralmethods}
\input{files/spectral.tex}

\section{Acknowledgements}
\input{files/acknowledgements.tex}

\newpage
\appendix
\section*{Appendix}
\input{files/supplementary.tex}

\newpage

\bibliographystyle{unsrt}
\bibliography{refs}

\end{document}

%% file: settings_custom.tex
\def \({\left(}
\def \){\right)}
\def \[{\left[}
\def \]{\right]}

\newcommand{\bg}{{\textbf {g}}}

\newcommand{\bu}{{\textbf {u}}}

\newcommand{\bh}{{\textbf {h}}}
\newcommand{\rI}{{\mathrm{I}}}

\newcommand{\bv}{{\textbf {v}}}

\newcommand{\bB}{{\textbf {B}}}

\newcommand{\bx}{{\textbf {x}}}

\newcommand{\bomega}{{\boldsymbol{\omega}}}
\newcommand{\bgamma}{{\boldsymbol{\gamma}}}

\newcommand{\bz}{{\textbf {z}}}

\newcommand{\bc}{{\textbf {c}}}

\newcommand{\bxi}{{\boldsymbol{\xi}}}
\newcommand{\bzero}{{\textbf{0}}}

\newcommand{\be}{\begin{equation}}
\newcommand{\ee}{\end{equation}}
\newcommand{\beqa}{\begin{eqnarray}}
\newcommand{\eeqa}{\end{eqnarray}}
\newcommand{\mE}{\mathbb{E}}

\newcommand{\bea}{\begin{align}}
\newcommand{\eea}{\end{align}}

\newcommand{\rmd}{\mathrm{d}}

\newcommand{\bigo}{\ensuremath{\mathcal{O}}}

\newcommand{\indi}{\mathds{1}}
\theoremstyle{plain}
\newtheorem{thm}{\textbf{Theorem}} %Different numbering
\newtheorem{corollary}{\textbf{Corollary}}[section]
\newtheorem{lemma}{\textbf{Lemma}}[section]
\newtheorem{theorem}{\textbf{Theorem}}[section]
\newtheorem{proposition}{\textbf{Proposition}}[section]
\newtheorem{conjecture}{\textbf{Conjecture}}[section]

\theoremstyle{remark}

\newtheorem*{remark*}{\textbf{Remark}}

 \newenvironment{proofof}[1]{{\bf {\em Proof of #1.}}}{\hfill \rule{2mm}{2mm} %\qed 
 }

\DeclareMathAlphabet{\varmathbb}{U}{bbold}{m}{n}
\newcommand{\id}{\mathds{1}}
\newcommand{\EE}{\mathbb{E}}
\newcommand{\bbR}{\mathbb{R}}

\newcommand{\bbN}{\mathbb{N}}

\newcommand{\bbC}{\mathbb{C}}

\def\du#1{#1}
\newcommand{\ud}[1]{#1}
\newcommand{\Z}{\mathcal{Z}}
\newcommand{\mZ}{\mathcal{Z}}

\newcommand{\mO}{\mathcal{O}}

\newcommand{\mN}{\mathcal{N}}

\newcommand{\mC}{\mathcal{C}}

\newcommand{\extr}{\textrm{\textbf{extr}}}

\newcommand{\td}[1]{{\tilde{#1}}}

\newcommand{\R}{\mathbb{R}}

\newcommand{\spacecase}[0]{\vspace{0.3cm} \\}
\newcommand{\hhspace}[0]{\hspace{0.3cm} }
\newcommand{\andcase}[0]{\hspace{ 0.2cm }\textrm{ and }\hspace{ 0.2cm }}

\def\E{\mathbb{E}}

\def\out{\text{out}}

\def\dd{\text{d}}

\def\iid{\text{i.i.d.}}

\def\sgn{\text{sgn}}

\newcommand{\underlim}[2]{\underset{#1 \to #2}{\longrightarrow}}

%% file: files/introduction.tex
A key idea of modern signal processing is to exploit the structure of the
signals under investigation. A traditional and powerful way of doing so is via sparse
representations of the signals.
Images are typically sparse in the wavelet domain, sound in the
Fourier domain, and sparse coding \cite{olshausen1997sparse} is designed to search
automatically for dictionaries in which the signal is sparse. This
compressed representation of the signal can be used to enable
efficient signal processing under larger noise or with fewer samples
leading to the ideas behind compressed sensing
\cite{donoho2006compressed} or sparsity enhancing regularizations.
Recent years brought a surge of interest in another powerful and
generic way of representing signals -- generative modeling. In particular the generative adversarial
networks (GANs) \cite{goodfellow2014generative} provide an impressively
powerful way to represent classes of signals. A recent series of works on compressed
sensing and other regression-related problems successfully explored the idea
of replacing the traditionally used sparsity by generative
models \cite{tramel2016inferring,bora2017compressed, manoel2017multi,hand2017global,
  fletcher2018inference,hand2018phase,mixon2018sunlayer}.
These results and performances conceivably suggest that
\cite{BlogSoledad}:
\begin{equation}
{\rm Generative \, \, models \, \, are \, \, the \, \, new \, \, sparsity.} \nonumber
\end{equation}

Next to compressed sensing and regression, another technique in statistical analysis that uses
sparsity in a fruitful way is sparse principal component
analysis (PCA) \cite{zou2006sparse}. Compared to the standard PCA, in sparse-PCA the principal components are linear combinations of a few of
the input variables, specifically $k$ of them. This means (for rank-one) that we aim to decompose the observed data
matrix $Y \in {\mathbb R}^{n\times p}$ as $Y = {\bu} {\bv}^\intercal
+ \xi$ where the spike $\bv \in {\mathbb R}^p$ is a vector with only
$k \ll p$ non-zero components, and $\bu, \xi$ are commonly modelled as
independent and identically distributed (i.i.d.) Gaussian variables. 

The main goal of this paper is to explore the idea of replacing
sparsity of the spike $\bv$ by the assumption that the spike belongs to the
range of a generative model.  Sparse-PCA with structured sparsity
inducing priors is well studied, e.g.~\cite{jenatton2010structured},
in this paper we remove the sparsity entirely and in a sense replace it by lower
dimensionality of the latent space of the generative model. 
For the purpose of comparing generative model priors and sparsity we
focus on the rich range of properties in the noisy high-dimensional
regime (denoted below, borrowing statistical physics jargon, as the \emph{thermodynamic limit}) where the spike
$\bv$ cannot be estimated consistently, but can be estimated better than by
random guessing. In particular we analyze two
spiked-matrix models as considered in a series of existing works on sparse-PCA,
e.g.~\cite{rangan2012iterative,deshpande2014information,lesieur2015phase,perry2016optimality,lelarge2019fundamental,barbier2016mutual,miolane2017fundamental},
defined as follows:
\paragraph{Spiked Wigner model ($\bv \bv^\intercal$):}
 Consider an unknown vector (the spike) $\bv^\star \in \bbR^{p}$ drawn
 from a distribution $P_v$;  we observe a matrix $Y \in \bbR^{ p
   \times p}$ with a symmetric noise term $\xi \in \bbR^{p \times p}$ and $\Delta > 0$:
\begin{align}
	Y = \frac{1}{\sqrt{p}} {\bv^\star} {\bv^\star}^\intercal + \sqrt{\Delta} \xi \, ,\label{Wigner}
\end{align}
where
$\xi_{ij} {\sim} \mN\(0,1\)$ i.i.d. The aim is to find back the hidden spike ${\bv^\star}$ from $Y$ (up to a global sign).
\paragraph{Spiked Wishart (or spiked covariance) model ($\bu \bv^\intercal$):}
Consider two unknown vectors ${\bu}^{\star}\in \bbR^{n}$  and ${\bf
  v}^{\star} \in \bbR^{p}$ drawn from distributions $P_u$ and $P_v$
and let $\xi \in \bbR^{n \times p}$ with $\xi_{\mu i}
{\sim} \mN\(0,1\)$ i.i.d. and $\Delta > 0$, we observe
\begin{align}
	Y = \frac{1}{\sqrt{p}} {\bu^\star} {\bv^\star}^\intercal + \sqrt{\Delta} \xi \,;\label{Wishart}
\end{align}
the goal is to find back the hidden spikes ${\bu}^{\star}$ and ${\bv}^\star$ from $Y \in \bbR^{n \times p}$.

The noisy high-dimensional limit that we consider in this paper (the
{\it  thermodynamic limit}) is $p,n \! \to\!  \infty$ while $\beta\!
\equiv\! n/p \!=\! \Theta(1)$, and
the noise $\xi$ has a variance $\Delta\!=\!\Theta(1)$.
The prior $P_v$ is representing the spike $\bv$ via a $k$-dimensional
parametrization with $\alpha\!\equiv\! p/k \!=\! \Theta(1)$. In the sparse case,
$k$ is the number of non-zeros components of $\bv^{\star}$, while in
generative models $k$ is the number of latent variables.

\subsection{Considered generative models}
\label{sec:generative_models}
The simplest non-separable prior $P_v$ that we consider is the
Gaussian model with a covariance matrix $\Sigma$, that is
$P_{v}(\bv)={\cal N}(\bv;\bzero,\Sigma)$. This prior is not compressive,
yet it captures some structure and can be simply estimated
from data via the empirical covariance. We use this prior later to produce
Fig.~\ref{main:experiement_mnist}.

To exploit the
practically observed power of generative models, it would be desirable to consider
models (e.g. GANs, variational auto-encoders, restricted Boltzmann
machines, or others) trained on datasets of examples of possible spikes. Such
training, however, leads to correlations between the weights of the
underlying neural networks for which
the theoretical part of the present paper does not apply readily. To
keep tractability in a closed form, and subsequent theoretical insights, we focus on multi-layer generative models where all the
weight matrices $W^{(l)}$, $l=1,\dots,L$, are fixed, layer-wise independent,
i.i.d. Gaussian with zero mean and unit variance. Let $\bv \in \bbR^{p}$ be the output of such a generative model
\begin{equation}
    \bv = \varphi^{(L)} \( \frac{1}{\sqrt{k}} W^{(L)} \dots \varphi^{(1)} \(\frac{1}{\sqrt{k}} W^{(1)} \bz \)
    \dots \)\, .
    \label{main:eqMLmodel}
  \end{equation}
with $\bz\in \bbR^{k}$ a latent variable drawn from separable distribution
$P_z$, with $\rho_z = \EE_{P_z} \[z^2\]$ and $\varphi^{(l)}$ element-wise activation functions that can be either deterministic or
  stochastic. In the setting considered in this paper the ground-truth spike $\bv^*$ is
generated using a ground-truth value of the latent variable
$\bz^*$. The spike is then estimated from the knowledge of the
data matrix $Y$, and the known form of the spiked-matrix and of the
generative model. In particular the matrices $W^{(l)}$ are known, as are the
parameters $\beta$, $\Delta$, $P_z$, $P_u$, $P_v$,
$\varphi^{(l)}$. Only the spikes $\bv^*$, $\bu^*$ and the latent vector
$\bz^*$ are unknown, and are to be inferred.

For concreteness and simplicity, the generative model that will be
  analyzed in most examples given in the present paper is
  the single-layer case of (\ref{main:eqMLmodel}) with $L=1$: 
\begin{align}
	\bv = \varphi \( \frac{1}{\sqrt{k}} W \bz \) \hhspace \Leftrightarrow \hhspace
  \bv \sim P_{\rm out}\( \cdot \Big| \frac{1}{\sqrt{k}} W \bz \) \, . \label{gen_single}
\end{align}
We define the compression ratio $\alpha \equiv p/k$. In what follows we will illustrate our results for
  $\varphi$ being linear, sign and ReLU functions.

\subsection{Summary of main contributions}
We analyze how the availability of generative priors, defined in
section \ref{sec:generative_models}, influences the statistical and
algorithmic properties of the spiked-matrix models (\ref{Wigner}) and (\ref{Wishart}). Both sparse-PCA
and generative priors provide statistical advantages when the
effective dimensionality $k$ is small, $k \ll p$. However, we show that from
the algorithmic perspective the two cases are quite different.
This is why our main findings are best presented in a context of the results known
for sparse-PCA. We draw two main conclusions from the present work:

{\bf (i) No algorithmic gap with generative-model priors:}
Sharp and detailed results are known in the thermodynamic limit (as
defined above) when
the spike $\bv^\star$ is sampled from a separable distribution $P_v$. A detailed account of several examples can be found in
\cite{lesieur2017constrained}. The main finding for sparse priors $P_v$ is that
when the sparsity $\rho = k/p = 1/\alpha$ is large enough then there exist optimal
algorithms \cite{deshpande2014information}, while for
$\rho$ small enough there is a striking gap between statistically optimal performance and the one of best
known algorithms \cite{lesieur2015phase}. The small-$\rho$ expansion
studied in 
\cite{lesieur2017constrained} is consistent with the well-known results for exact
recovery of the support of $\bv^\star$ \cite{amini2009high,berthet2013computational}, which is one of
the best-known cases in which gaps between statistical and best-known algorithmic performance were described.

Our analysis of the spiked-matrix models with generative priors reveals that in
this case known algorithms are able to obtain (asymptotically) optimal
performance even when the dimension is greatly reduced, i.e. $\alpha \gg 1$. Analogous
conclusion about the lack of algorithmic gaps was reached for the
problem of phase retrieval under a generative prior in
\cite{hand2018phase}. This result suggests that plausibly generative
priors are better than sparsity as they lead to algorithmically easier
problems.

{\bf (ii) Spectral algorithms reaching statistical threshold:}
Arguably the most basic algorithm used to solve the spiked-matrix
model is based on the leading singular vectors of the matrix $Y$. We
will refer to this as PCA. 
Previous work on spiked-matrix models \cite{perry2016optimality,lesieur2017constrained} established that in the thermodynamic limit and for separable priors of zero mean PCA
reaches the best performance of all known efficient algorithms in terms of the value of noise
$\Delta$ below which it is able to provide positive correlation between
its estimator and the ground-truth spike. 
While for sparse priors
positive correlation is statistically reachable even for larger values of
$\Delta$ \cite{perry2016optimality,lesieur2017constrained}, no efficient algorithm beating the PCA threshold is
known\footnote{This result holds only for sparsity $\rho=\Theta(1)$. A
line of works shows that when sparsity $k$ scales slower than linearly
with $p$, algorithms more performant than PCA exist
\cite{amini2009high,deshpande2014sparse}}.

In the case of generative priors we find in this paper that other spectral
methods improve on the canonical PCA. We design a spectral method,
called LAMP, that
(under certain assumptions, e.g. zero mean of the spikes) reach the statistically optimal threshold,
meaning that for larger values of noise variance no other (even exponential) algorithm is
able to reach positive correlation with the spike. Again this is a
striking difference with the sparse separable prior, making the generative
priors algorithmically more attractive. We demonstrate the performance
of LAMP on the spiked-matrix model when
the spike is taken to be one of the fashion-MNIST images showing
considerable improvement over canonical PCA.

%% file: files/information_theory.tex
We first discuss the information theoretic results on the estimation
of the spike, regardless of the computational cost. A considerable
amount of results have been obtained for the spiked-matrix models with
separable priors~\cite{rangan2012iterative,deshpande2014information,deshpande2016asymptotic,krzakala_mutual_2016,barbier2016mutual,lelarge2019fundamental,AlaouiKrzakala,alaoui2017finite,barbier2018adaptive,mourrat2019hamilton}. Here,
we extend these results to the case where the spike
$\bv^\star \in \bbR^{p}$ is generated from a {\it generic
  non-separable prior} $P_v$ on $\bbR^p$.
  
\subsection{Mutual Information and Minimal Mean Squared Error}
We consider the mutual information between the ground-truth spike
$\bv^\star$ and the observation $Y$, defined as
$ I(Y;\bv^\star)=D_{\mathrm {KL}
}(P_{(v^\star,Y)}\|P_{v^\star} P_{Y})$. 
Next, we consider the best possible value of
the mean-squared-error on recovering the spike, commonly called the
minimum mean-squared-error (MMSE). The MMSE estimator is computed from
marginal-means of the posterior distribution $P(\bv|Y)$.
\begin{thm}\label{theorem_uu}[Mutual information for the
  spiked Wigner model with structured spike] Informally (see SM
  section \ref{appendix:proof_uu} for details and proof), assume the
  spikes $\bv^\star$ come from a sequence (of growing dimension $p$) of
  generic structured priors $P_v$ on $\bbR^p$, then
  \begin{align} \lim_{p \to \infty} i_p &\equiv \lim_{p\to \infty}
    \frac {I(Y;\bv^\star)}p = \inf_{\rho_v \ge q_v \ge 0} {i}_{\rm
    RS}(\Delta,q_v),\\
{\text with}~~~
    i_{\rm
      RS}(\Delta,q_v) &~\equiv
    \frac{(\rho_v-q_v)^2}{4\Delta} + \lim_{p \to \infty}
    \frac{I\left(\bv;\bv+\sqrt{\frac{\Delta}{q_v}} \bxi \right)}p
    \, \label{eq:information_theory:limip}\end{align}
and $\bxi$ being a
  Gaussian vector with zero mean, unit diagonal variance and
  $\rho_v=\lim\limits_{p \to \infty} \E_{P_v}[\bv^\intercal\bv]/p$.
\end{thm}

This theorem connects the asymptotic mutual information of the spiked
model with generative prior $P_v$ to the mutual information between $\bv$
taken from $P_{v}$ and its noisy version,
$I(\bv;\bv+\sqrt{{\Delta}/{q_v}}\bxi)$. Computing this later mutual information is itself a
high-dimensional task, hard in full generality, but it can be done for
a range of models. The simplest tractable case is when the prior $P_{v}$ is
separable, then it yields back exactly the formula known from
\cite{krzakala_mutual_2016,barbier2016mutual,lelarge2019fundamental}. It
can be computed also for the Gaussian generative model,
$P_{v}(\bv)={\cal N}(\bv;\bzero,\Sigma)$, leading to 
$I(\bv;\bv+\sqrt{{\Delta}/{q_v}}\bxi) = {\rm Tr}\left( \log{(\rI_p + q_v
    \Sigma/\Delta)}\right)/2$. 
    
More interestingly, the mutual information associated to the
generative prior in eq.~(\ref{eq:information_theory:limip}) can also be asymptotically computed for the multi-layer
generative model with random weights, defined in eq.~(\ref{main:eqMLmodel}). 
Indeed, for the single-layer prior (\ref{gen_single}) the corresponding formula for mutual information
has been derived and proven in \cite{Barbier2017c}. 
For the multi-layer case the mutual information formula has been derived in
\cite{manoel2017multi,reeves2017additivity} and proven for the case of two layers in
\cite{gabrie2018entropy}. Theorem \ref{theorem_uu} together with the
results from \cite{Barbier2017c, manoel2017multi,
  reeves2017additivity, gabrie2018entropy}
yields the following formula (see SM sec.~\ref{appendix:proof_uu} for
details) for the spiked Wigner model (\ref{Wigner}) with single-layer
generative prior (\ref{gen_single}):
\begin{align}
	i_{\rm RS} (\Delta,q_v) =  \frac{\rho_v^2}{4\Delta} + \frac{q_v^2}{4\Delta}
	+\frac{1}{\alpha} \min_{q_z} \max_{\hat{q}_z}
	\[\frac{1}{2} q_z \hat{q}_z  - \Psi_z(\hat{q}_z) - \alpha
  \Psi_{\rm out}\(\frac{q_v}{\Delta}, q_z\)   \]\, , 
	\label{main:free_entropy_uu}
\end{align}
where the functions $\Psi_z, \Psi_{\rm out}$ are defined by
\begin{align}
	\Psi_z (x) & \equiv \EE_{\xi} \[ \mZ_z\( x^{1/2} \xi ,x  \)
          \log \( \mZ_z\( x^{1/2} \xi ,x  \) \) \] \,
                     ,  \label{main:definition_Psi_z} 
        \\ 
	\Psi_{\rm out} (x,y) & \equiv \EE_{\xi, \eta} \[\mZ_{\rm
            out}\( x^{1/2} \xi , x , y^{1/2} \eta  , \rho_z - y \)
          \log\( \mZ_{\rm out}\( x^{1/2} \xi , x , y^{1/2} \eta  ,
          \rho_z - y \) \) \]\, ,
                \label{main:definition_Psi_out}
\end{align}
with $\xi, \eta {\sim} \mN\(0,1\)$ i.i.d., and $\mZ_z$ and $\mZ_{\rm out}$ are the normalizations of the following denoising scalar distributions:
\begin{align}
	Q_z^{\gamma, \Lambda} (z) \equiv \displaystyle \frac{P_z(z)}{\mZ_z(\gamma, \Lambda)}  e^{ - \frac{\Lambda }{2} z^2  + \gamma z  } \,; Q_{\rm out}^{B,A,\omega,V} (v, x) \equiv \displaystyle \frac{P_{\rm out}(v |x)}{\Z_{\rm out}(B, A,\omega, V)} e^{ -\frac{A}{2} v^2 + B v }   e^{ -\frac{\(x - \omega\)^2}{2V}  } \,.
\label{main:definition_Z}
\end{align}

Result (\ref{main:free_entropy_uu}) is remarkable in that it
connects the asymptotic mutual
information of a high-dimensional model with a simple
scalar formula that can be easily evaluated. In the SM
sec.~\ref{sec:appendix:replicafreeen} we show how this formula is obtained using the
heuristic replica method from statistical physics and, once we have
the formula in hand, we prove it using the interpolation method in SM
sec.~\ref{appendix:proof_uu}. In SM
sec.~\ref{sec:app:replicas:wishart} we also give the corresponding formula for the spiked Wishart model, and in sec.~\ref{sec:app:replicas:application}
for the multi-layer case.

Beyond its theoretical interest, the main point of the mutual
information formula is that it yields the optimal value of the
mean-squared error (MMSE). It is well-known \cite{cover2012elements}
that the mean-squared error is minimized by an estimator evaluating the
conditional expectation of the signal given the observations. Following generic theorems on the connection
between the mutual information and the MMSE \cite{GuoShamaiVerdu_IMMSE},
one can prove in particular that for the spiked-matrix model
\cite{AlaouiKrzakala} the MMSE on the spike $\bv^{\star}$ is
asymptotically given by:
\begin{align}
	{\rm MMSE}_v = \rho_v-q_v^\star  \,, \label{eq:MMSE}
\end{align}
where $q_v^\star$ is the optimizer of the function
$i_{\rm RS}\(\Delta , q_v\)$.
\subsection{Examples of phase diagrams}
\label{sec:examples}

Taking the extremization over $q_{v}, \hat{q}_z, q_{z}$ in eq.~(\ref{main:free_entropy_uu}), we obtain the following fixed point equations:
\begin{align}
		q_v = 2 \partial_{q_v} \Psi_{\rm out}
  \(\frac{q_v}{\Delta}, q_z \),  \hhspace	q_z =
  2 \partial_{\hat{q}_z} \Psi_{z} \(\hat{q}_z \),  \hhspace
		\hat{q}_z = 2 \alpha \partial_{q_z} \Psi_{\rm out} \(\frac{q_v}{\Delta}, q_z \).
	\label{main:SE_uu}
\end{align}
Using (\ref{eq:MMSE}), analyzing the fixed
points of eqs.~(\ref{main:SE_uu}) provides all the informations about
the performance of the Bayes-optimal estimator in the models under
consideration.

\paragraph{Phase transition:} A first question is whether better estimation
than random guessing from the prior is possible.
In terms of fixed points of eqs.~\eqref{main:SE_uu}, this corresponds to
the existence of the \emph{non-informative} fixed point $q^{\star}_{v} =
0$ (i.e. zero overlap with the spike, or maximum ${\rm MSE}_v =
\rho_v$).
Evaluating the right-hand side of eqs.~\eqref{main:SE_uu} at $q_{v} =
0$, we can see that $q_{v}^{\star} = 0$ is a fixed point if
\begin{align}
  \mathbb{E}_{P_{z}}\[z\] = 0 \andcase \mathbb{E}_{Q_{\out}^{0}}\[v\]
  = 0\, , \label{trivial_fixed}
\end{align}
where $Q_{\out}^{0}(v,x) \equiv Q_{\out}^{0,0,0,\rho_{z}}(v,x)$
from eq.~(\ref{main:definition_Z}). Note that for a deterministic
channel the second condition is equivalent to $\varphi$ being an odd
function.

When the condition (\ref{trivial_fixed}) holds,
$(q_{v},\hat{q}_{z},q_{z}) = (0,0,0)$ is a fixed point of
eq.~\eqref{main:SE_uu}. The numerical stability of this fixed point
determines a phase transition point $\Delta_{c}$, defined as the noise below which the fixed point $(0,0,0)$ becomes unstable. This corresponds to the value of $\Delta$ for which the largest eigenvalue of the Jacobian of the eqs.~\eqref{main:SE_uu} at $(0,0,0)$, given by

\begin{align}
    2\dd(\partial_{q_{v}}\Psi_{\rm out},\alpha\partial_{q_{z}}\Psi_{\out},\partial_{\hat{q}_{z}}\Psi_{z})|_{(0,0,0)} =
        \begin{pmatrix}
        \frac{1}{\Delta}\left(\mathbb{E}_{Q_{\out}^{0}}v^2\right)^2 & 0 & \frac{1}{\rho_{z}^{2}}\left(\mathbb{E}_{Q_{\out}^{0}}vx\right)^2 \\
        \frac{\alpha}{\Delta}\left(\mathbb{E}_{Q_{\out}^{0}}vx\right)^2 &  0 &\frac{\alpha}{\rho_{z}^2}\left(\mathbb{E}_{Q_{\out}^{0}}x^2-\rho_{z}\right)^2 \\ 
        0 &  \left(\mathbb{E}_{P_{z}}z^2\right)^2 & 0
    \end{pmatrix},
        \label{eq:phasetransition:jacobian}
\end{align}
becomes greater than one. The details of this calculation can be found in sec.~\ref{sec:app:stability} of the SM.

\begin{figure}[tb!]
	\centering
		\includegraphics[width=1.0\linewidth]{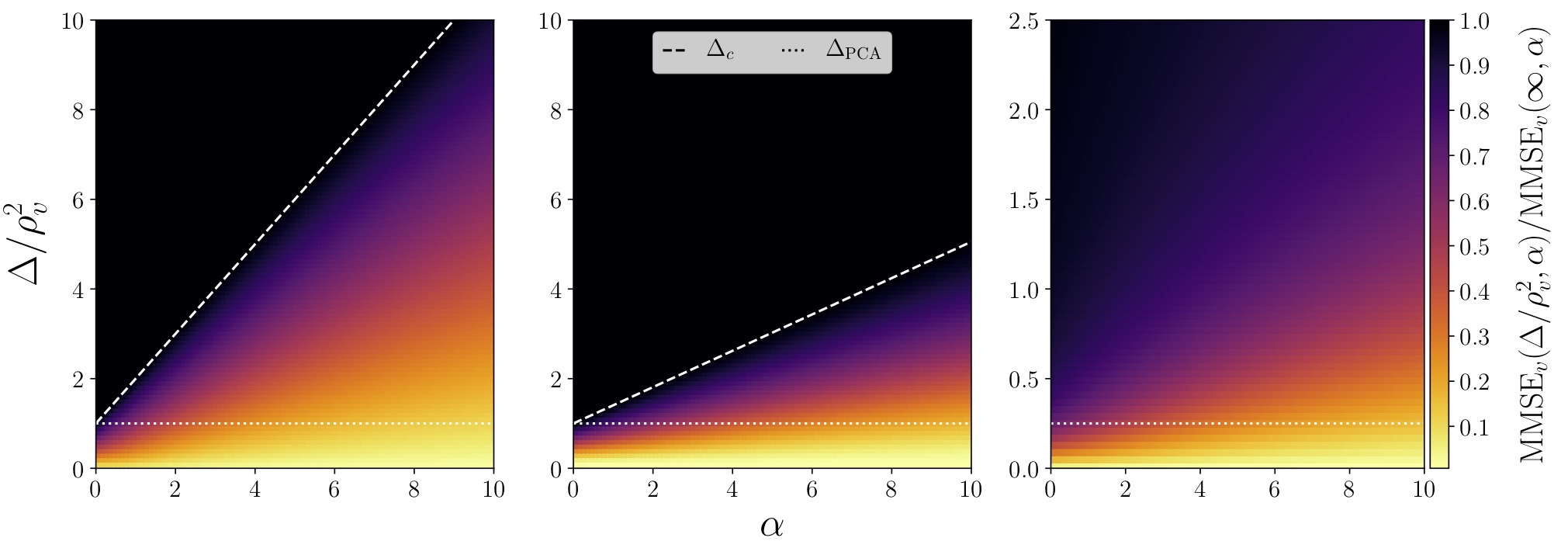}
	\caption{Spiked Wigner model: ${\rm MMSE}_v$ on the spike as a function of noise to signal ratio $\Delta/\rho_v^2$, and generative prior (\ref{gen_single}) with
          compression ratio $\alpha$ for linear (left, $\rho_v=1$), sign
          (center, $\rho_v=1$), and relu (right, $\rho_v=1/2$) activations. Dashed white lines
          mark the phase transitions $\Delta_c$, matched by both the
          AMP and LAMP algorithms. Dotted white line
          marks the phase transition of canonical PCA.}
	\label{main:fig_map_mse_delta_alpha}
\end{figure}

\begin{figure}[tb!]
	\centering
		\includegraphics[width=1.0\linewidth]{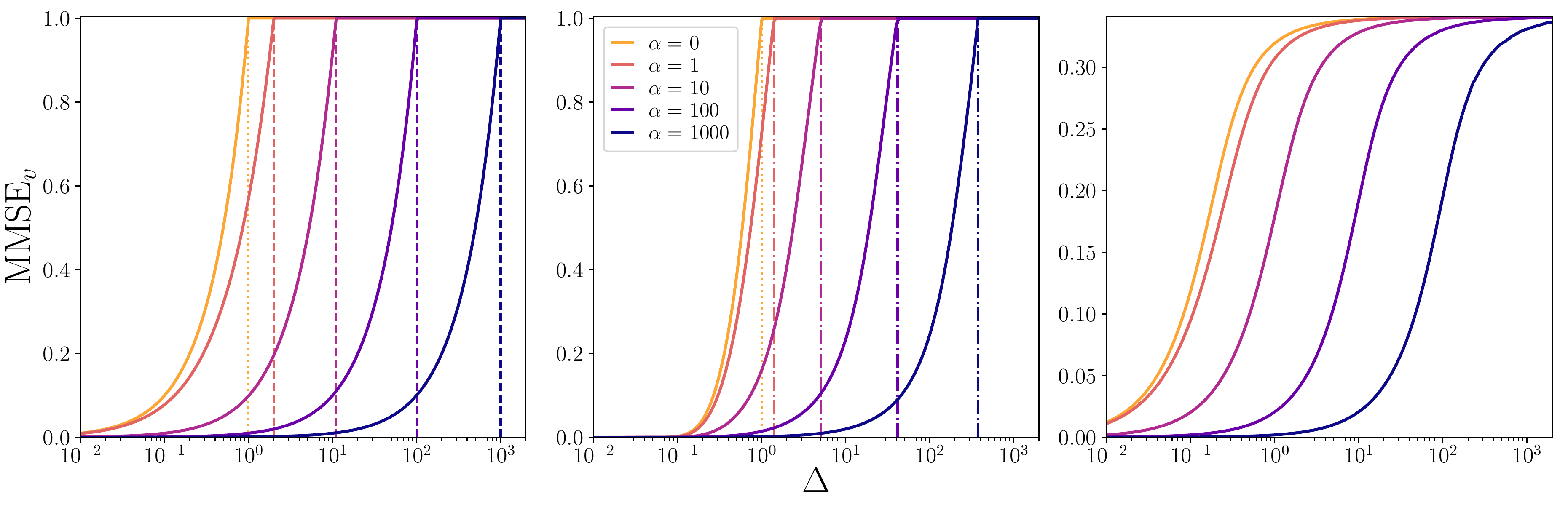}
	\caption{Spiked Wigner model: ${\rm MMSE}_v$ as a function of
          noise $\Delta$ for a wide range of compression ratios
          $\alpha=0,1,10,100,1000$, for linear (left), sign
          (center), and relu (right) activations. Unique stable fixed
          point of (\ref{main:SE_uu}) is found for all these cases.}
	\label{main:fig_mse_u_SE}
\end{figure}

It is instructive to compute $\Delta_{c}$ in specific cases. We
therefore fix $P_{z} = \mathcal{N}(0,1)$ and $P_{\out}(v|x)=
\delta(v-\varphi(x))$ and discuss two different choices of (odd)
activation function $\varphi$.
\begin{description}
    \item[Linear activation:] For $\varphi(x)=x$ the leading
      eigenvalue of the Jacobian becomes one at $\Delta_{c} =
      \alpha+1$.
Note that in the limit $\alpha = 0$ we recover the phase
    transition
$\Delta_{c}=1$ known from the case with separable prior \cite{lesieur2017constrained}. For
$\alpha>0$, we have $\Delta_{c}>1$ meaning the spike can be estimated
more efficiently when its structure is accounted for.

    \item[Sign activation:] For $\varphi(x)=\sgn(x)$ the
      leading eigenvalue of the Jacobian becomes one at $\Delta_{c} =
      1+\frac{4\alpha}{\pi^2}$. For $\alpha=0$, $P_{v} =
      \text{Bern(1/2)}$, and the transition $\Delta_{c}=1$ agrees with
      the one found for a separable prior distribution \cite{lesieur2017constrained}. As in the linear case, for
      $\alpha>0$, we can estimate the spike for
      larger values of noise than in
      the separable case.
\end{description}

In Fig.~\ref{main:fig_map_mse_delta_alpha} we solve the fixed point
equations (\ref{main:SE_uu}) and plot the MMSE obtained from the fixed
point in a heat map, for the linear, sign and relu
activations. The white dashed line marks the above stated threshold
$\Delta_c$. 
The property that we find the most
striking is that in these three evaluated cases, for all values of
$\Delta$ and $\alpha$ that we analyzed, we always found that
eq.~(\ref{main:SE_uu}) has a unique stable fixed point. Thus we have
not identified any first order phase transition (in the physics
terminology). This is illustrated in Fig.~\ref{main:fig_mse_u_SE} for larger values of
$\alpha$, where we solved the eq.~(\ref{main:SE_uu}) iteratively from
uncorrelated initial condition, and from initial condition
corresponding to the ground truth signal, and found that both lead to the same fixed
point.

%% file: files/amp.tex
\label{main:sec_amp}
A straightforward algorithmic evaluation of the Bayes-optimal estimator is exponentially costly. This section is devoted to the analysis of an approximate message passing (AMP) algorithm that for the analyzed cases is able to reach the optimal performance (in the thermodynamic limit). For the purpose of presentation,
we focus again on the spiked Wigner model (see SM for the spiked Wishart model). For separable priors, the AMP for the spiked Wigner model is well known~\cite{rangan2012iterative,deshpande2014information,lesieur2015phase}. It can, however, be extended to non-separable priors~\cite{metzler2016denoising,manoel2017multi,berthier2017state}. We show in SM sec.~\ref{appendix:amp_derivation} how AMP can be generalized to handle the generative model (\ref{gen_single}). It reads:
\begin{algorithm}[!htb]
  \caption{AMP algorithm for the spiked Wigner model with single-layer generative prior.}
\begin{algorithmic}
    \STATE {\bfseries Input:} $Y \in \bbR^{p \times  p}$ and $W\in \bbR^{p \times k}$:
    \STATE \emph{Initialize to zero:} $( \bg, \hat{\bv}, \bB_{v} , A_{v})^{t=0}$. 
	\STATE \emph{Initialize with:} $\hat{\bv}^{t=1}=\mN(0,\sigma^2) $, $\hat{\bz}^{t=1}=\mN(0,\sigma^2)$, and $\hat{\bc}^{t=1}_v = \id_p$, $\hat{\bc}^{t=1}_z=\id_k$, $t=1$.
    \REPEAT
    \STATE \emph{Spiked layer:}
    \STATE $\ud{{\bB}}_v^{t} = \frac{1}{\Delta} \frac{Y}{\sqrt{p}}   \ud{\hat{\bv}}^{t}
 -  \frac{1}{\Delta } \frac{ \( \id_p^\intercal \hat{\bc}^{t}_v\) }{p}  \hat{\bv}^{t-1}$ \andcase $A^{t}_v = \frac{1}{\Delta p} \|\hat {\bv}^t \|_2^2 \rI_p  $.
   \STATE \emph{Generative layer:}
   \STATE $V^t =\frac{1}{k} \( \id_k^\intercal \hat{\bc}^{t}_z\) \rI_p $, \hhspace ${\boldsymbol \omega }^t =  \frac{1}{\sqrt{k}} W \hat{ {\bf z} }^{t} - V^t {\bg}^{t-1}$  \andcase ${\bg}^{t} = f_{\rm out}\({\bB}^{t}_v , A^{t}_v , {\boldsymbol \omega}^{t} , V^{t} \) $,
  \STATE $\du{\Lambda}^t = \frac{1}{k} \| {\bg}^t \|_2^2 \rI_k$
	\andcase ${ {\boldsymbol  \gamma} }^t = \frac{1}{\sqrt{k}}  W^\intercal  {\bg}^t + \Lambda^t \hat{\bz}^t  $.
    \STATE \emph{Update of the estimated marginals:}
    \STATE $\hat{\bv}^{t+1} = \ud{f}_v(\bB^{t}_v , A^{t}_v , {\boldsymbol \omega}^t , \du{V}^t)   \hspace{0.5cm}\andcase \hspace{0.5cm}  \hat{\bc}^{t+1}_v = \partial_B f_v( \bB^{t}_v , A^{t}_v , {\boldsymbol \omega}^t , V^t) $,
    \STATE $\hat{\bz}^{t+1} = f_z({\boldsymbol \gamma}^t , \Lambda^t)   \hspace{0.5cm} \andcase \hspace{0.5cm}  \hat{\bc}^{t+1}_z = \partial_\gamma f_z ({\boldsymbol \gamma}^t , \Lambda^t)  $,
    \STATE ${t} = {t} + 1$.
    \UNTIL{Convergence.}
    \STATE {\bfseries Output: $\hat{\bv}, \hat{\bz}$.}
\end{algorithmic}
\label{main:AMP_vv}
\end{algorithm}
where $\rI_s$ and $\id_s$ denote respectively the identity matrix and vector of ones of size $s$. The update functions $f_{\rm out}$ and $f_v$ are the means of $V^{-1}\(x-\omega\)$ and $v$ with respect to $Q_{\rm out}$, eq.~\eqref{main:definition_Z}, while the update function $f_z$ is the mean of $z$ with respect to $Q_z$, eq. (\ref{main:definition_Z}).

The algorithm for the spiked Wishart model is very similar and both derivations are given in SM sec.~\ref{appendix:AMP_derivation}.  We define the overlap of the AMP estimator with the ground truth spike as $(\hat{\bv}^t)^\intercal \bv^\star/p {\longrightarrow} q_v^t$ as ${p \to \infty}$.
Perhaps the most important virtue of AMP-type algorithms is that their asymptotic performance can be tracked exactly via a set of scalar equations called {\it state evolution}. This fact has been proven for a range of models including the spiked matrix models with separable priors in \cite{javanmard2013state}, and with non-separable priors in \cite{berthier2017state}. To help the reader understand the state evolution equations we provide a heuristic derivation in the SM, section \ref{appendix:derivation_SE_from_AMP}. The state evolution states that the overlap $q_v^t$ evolves under iterations of the AMP algorithm as:
\begin{align}
		q_{v}^{t+1} = 2 \partial_{q_v} \Psi_{\rm out}
  \(\frac{q_{v}^{t}}{\Delta}, q_{z}^{t} \), \quad	q_{z}^{t+1} = 2 \partial_{\hat{q}_z} \Psi_z \(\hat{q}_{z}^t \),  \quad
		\hat{q}_{z}^{t} = 2 \alpha \partial_{q_z} \Psi_{\rm out} \(\frac{q_{v}^{t}}{\Delta}, q_{z}^{t} \),
	\label{main:SE_AMP_uu}
\end{align}
with initialization $q_v^{t=0} = \varepsilon $, $q_z^{t=0}=\varepsilon$ and a small $\varepsilon>0$.
We notice immediately that (\ref{main:SE_AMP_uu}) are the same equations as the fixed point equations related to the Bayes-optimal estimation (\ref{main:SE_uu}) with specific time-indices and initialization, but crucially the same fixed points.
Thus the analysis of fixed points in section \ref{sec:examples} applies also to the behaviour of AMP. In particular, since in all cases analyzed we found the stable fixed point of (\ref{main:SE_uu}) to be unique, it means the AMP algorithm is able to reach asymptotically optimal performance in all these cases. This is further illustrated in Fig.~\ref{main:bbp_lamp_amp_se} where we explicitly compare runs of AMP on finite size instances with the results of the asymptotic state evolution, thus also giving an idea of the amplitude of the finite size effects. Note that we provide a demonstration notebook in the \href{https://github.com/sphinxteam/StructuredPrior_demo}{GitHub repository} \cite{StructuredPrior_demo_repo} that compares AMP, LAMP and PCA numerical performances.

%% file: files/spectral.tex
\begin{figure}[tb!]
\centering
	\includegraphics[width=1.0\linewidth]{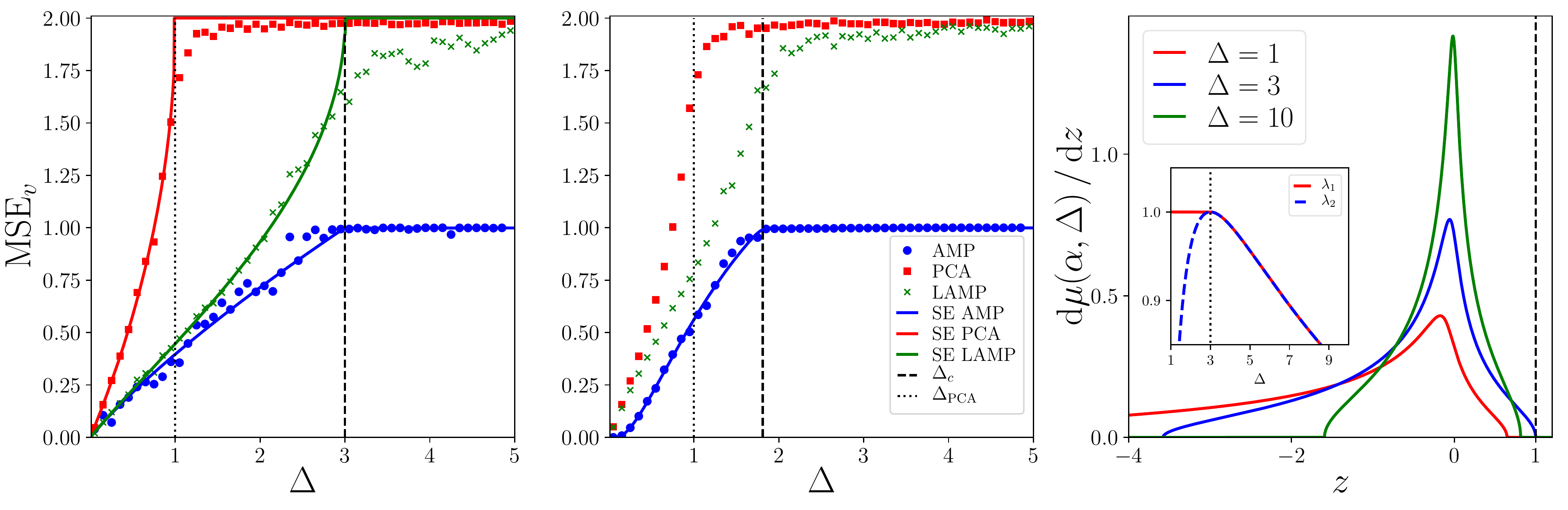}
	\caption{Comparison between PCA, LAMP and AMP for {(left)}
          the linear, {(center)} and sign activations, at compression ratio
          $\alpha=2$. Lines correspond to the theoretical asymptotic
          performance of PCA (red line), LAMP (green line) and AMP
          (blue line). Dots correspond to simulations
          of PCA (red squares), LAMP (green crosses) for $k=10^4$ and AMP (blue
          points) for $k=5.10^3$, $\sigma^2=1$. {(Right)} Illustration of the spectral phase transition in the
          matrix $\Gamma_p^{vv}$ eq.~(\ref{eq:matrix_symmetric}) at $\alpha=2$ with
          an informative leading eigenvector with eigenvalue equal to
          $1$ out of the bulk  for   $\Delta\le 1+\alpha$.  We show
          the bulk spectral density $\mu(\alpha, \Delta)$. The inset
          shows the two leading eigenvalues.}
	\label{main:bbp_lamp_amp_se}
\end{figure}

Spectral algorithms are the most commonly used ones for the spiked
matrix models. For instance, canonical PCA estimates the spike from the leading
eigenvector of the matrix~$Y$. A classical result from Baik, Ben Arous
and P\'ech\'e (BBP)~\cite{baik2005phase} shows that this eigenvector
is correlated with the signal if and only if the signal-to-noise ratio
$\rho_v^2/\Delta >1$. For sparse separable priors (with
$\rho_v^2=\Theta(1)$), $\Delta_{\rm PCA} = \rho_v^2$ is also the threshold
for AMP and it is conjectured that no polynomial algorithm can improve
upon it \cite{lesieur2017constrained}. In the previous section we show
that for the analyzed generative priors AMP has a better threshold
than PCA. Here we design a
spectral method, called LAMP, that matches the AMP threshold and is hence superior
over the canonical PCA. In order to do so, we follow the powerful
strategy pioneered in \cite{krzakala_spectral_2013}
and linearize the AMP around its
non-informative fixed point. In the spiked Wigner model with a single-layer
prior the linearized AMP leads to the following operator:
\begin{align}
 \du{\Gamma}_p^{vv} = \frac{1}{\Delta}  \( (a-b) \rI_p  +  b  \frac{\du{W} \du{W}^\intercal}{k} + c  \frac{\ud{\id}_p \ud{\id}_k^\intercal}{k} \frac{\du{W}^\intercal}{\sqrt{k} }  \) \times \(\frac{\du{Y}}{\sqrt{p}}  -  a \rI_p  \) \,,\label{eq:spectral:gammas}
\end{align}
where parameters are moments of distributions $P_z$ and $Q_{\rm out}^0$ according to
\begin{align}
	a \equiv \rho_v\,, \hhspace  b \equiv  \rho_z^{-1} \EE_{Q_{\rm out}^0} [v x]^2\,,  \hhspace c \equiv \frac{1}{2} \rho_z^{-3} \EE_{P_{\rm z}} \[z^3\]  \EE_{Q_{\rm out}^0}[v x^2] \EE_{Q_{\rm out}^0}[v x]\,.
\end{align}
We denote the spectral algorithm that takes the leading eigenvectors
of (\ref{eq:spectral:gammas}) as LAMP (for linearized-AMP). Its
derivation is presented in SM sec.~\ref{appendix:lamp_derivation} together with the one for the
spiked Wishart model. 

For the specific case of Gaussian $z$ and prior
(\ref{gen_single}) with the sign activation function we
obtain $(a,b,c)=(1,2 / \pi,0)$. For linear activation we get $(a,b,c)=(1,1,0)$, leading to 
\begin{align}\label{eq:matrix_symmetric}
  \Gamma_p^{vv} &= \frac{1}{\Delta} K_p \, \left[\frac{Y}{\sqrt{p}}
                  - \rI_p\right]~\text{with}~~K_p=\frac{\left[ W
                  W^\intercal \right]}k = \Sigma  \approx \frac{1}{n}\sum_\alpha
                  \bv^\alpha (\bv^\alpha)^\intercal \, ,
\end{align}
where the last two equalities come from the fact that for the
model (\ref{gen_single}) with linear activation and Gaussian
separable $P_z$, $K_{p}$ is asymptotically equal
to the covariance matrix between samples of spikes, $\Sigma$.
Interestingly, $\Sigma$ can be estimated empirically from samples of spikes,
without the knowledge of the matrix $W$ itself. 
Analogously to the state evolution for AMP, the asymptotic performance of both
PCA and LAMP can be evaluated in a closed-form for the spiked Wigner model
with single-layer generative prior with linear activation (\ref{gen_single}). The
corresponding expressions are derived in SM sec.~\ref{appendix:derivation_lAMP}
 and plotted in
Fig.~\ref{main:bbp_lamp_amp_se} for the three considered algorithms. 

In fact, the spectral method based on the matrix in eq.~(\ref{eq:matrix_symmetric})
can also be derived linearizing AMP with a
Gaussian prior with covariance $\Sigma$. This means that we can
use the above spectral method without extensive training by simply
computing the empirical covariance of $n$ samples of spikes, $\bv^\alpha$, $\alpha=1,\dots,n$. For illustration purposes, we display
the behaviour of this spectral method on the spiked Wigner model with
spikes coming from the Fashion-MNIST dataset in
Fig.~\ref{main:experiement_mnist}. A demonstration notebook is provided in the  \href{https://github.com/sphinxteam/StructuredPrior_demo}{GitHub repository}, illustrating PCA and LAMP performances on Fashion-MNIST dataset.

\begin{figure}[tb!]
\centering
	\includegraphics[width=1.0\linewidth]{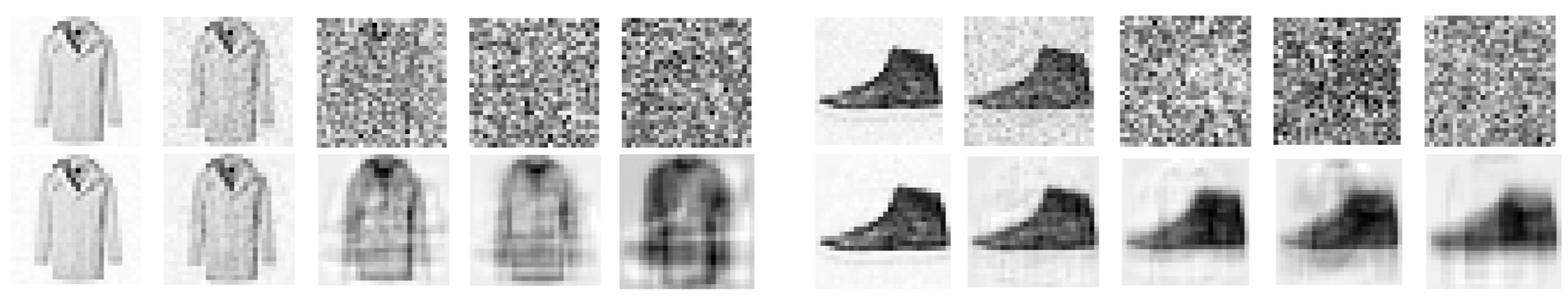}
	\caption{Illustration of canonical PCA (top line) and the LAMP
          (bottom line)
          spectral methods (\ref{eq:matrix_symmetric}) on the spiked
          Wigner model. The covariance $\Sigma$ is estimated
          empirically from the FashionMNIST database \cite{xiao2017/online}. The
          estimation of the spike is shown for two images from
          FashionMNIST, with           (from left to right), noise variance $\Delta=0.01,0.1,1,2,10$.}
	\label{main:experiement_mnist}
\end{figure}
      
Remarkably, the performance of the spectral method based on matrix (\ref{eq:matrix_symmetric}) can be
investigated independently of AMP using random matrix theory.
An analysis of the random matrix (\ref{eq:matrix_symmetric}) shows that a spectral phase transition for
generative prior with linear activations appears at 
$\Delta_c=1+\alpha$ (as for AMP).
This transition is analogous to the well-known BBP
transition \cite{baik2005phase}, but a non-GOE random matrix
(\ref{eq:matrix_symmetric}) needs to be analyzed. 
For the spiked Wigner models with linear generative
prior we prove two theorems
describing the behavior of the supremum of the bulk spectral density, the transition of the largest eigenvalue 
and the correlation of the corresponding eigenvector:
\begin{thm}[Bulk of the spectral density, spiked Wigner, linear activation]
	\label{thm:main_lambdamax_uu_case}
	Let $\alpha, \Delta > 0$, then:

        $(i)$ The spectral measure of $\Gamma_p^{vv}$ converges almost surely and in the weak sense to a compactly supported 
   	probability measure $\mu(\alpha, \Delta)$. We denote
        $\lambda_{\rm max}$ the supremum of the support of
        $\mu(\alpha, \Delta)$.
        
$(ii)$ For any $\alpha > 0$, as a function of $\Delta$, $\lambda_{\rm max}$ has a unique global maximum, reached exactly at the point $\Delta = \Delta_c(\alpha) = 1 + \alpha$.
								Moreover, $\lambda_{\rm max}(\alpha, \Delta_c(\alpha)) = 1$.
\end{thm}

\begin{thm}[Transition of the largest eigenvalue and eigenvector,
  spiked Wigner, linear activation]
	\label{thm:main_transition_uu_case}
	Let $\alpha > 0$.
	We denote $\lambda_1 \geq \lambda_2$ the first and second eigenvalues of $\Gamma^{vv}_p$.
If $\Delta \geq \Delta_c(\alpha)$, then as $p \to \infty$ we have a.s.
$\lambda_1 {\to} \lambda_{\rm max}$ and $\lambda_2
{\to} \lambda_{\rm max}$.       If $\Delta \leq
\Delta_c(\alpha)$, then as $p \to \infty$ we have a.s. $\lambda_1
{\to} 1$ and $\lambda_2 {\to}
\lambda_{\rm max}$. Further, denoting $\tilde{\bv}$ a normalized
($\norm{\tilde{\bv}}^2 = p$ ) eigenvector of $\Gamma^{vv}_p$ with
eigenvalue $\lambda_1$, then $|\tilde{\bv}^\intercal \bv^\star |^2
/p^2 {\to} \epsilon(\Delta)$ a.s., where 
$\epsilon(\Delta) = 0$ for all $\Delta \geq \Delta_c(\alpha)$,
		$\epsilon(\Delta) > 0$ for all $\Delta < \Delta_c(\alpha)$ and $\lim_{\Delta \to 0}\epsilon(\Delta) =1$.
\end{thm}

Thm.~\ref{thm:main_lambdamax_uu_case} and
Thm.~\ref{thm:main_transition_uu_case} are illustrated in
Fig.~\ref{main:bbp_lamp_amp_se}. 
The proof gives the value of $\epsilon(\Delta)$, which turns out to
lead to the same MSE as in Fig.~\ref{main:bbp_lamp_amp_se} in the linear case. 
We state the theorems counterparts for the $\bu \bv^\intercal$ linear case in SM sec.~\ref{appendix:rmt}.
The proofs of the theorems and the precise arguments used to derive the 
eigenvalue density, the transition of $\lambda_1$ and the computation of $\epsilon(\Delta)$ are given in SM sec.~\ref{appendix:rmt}, and a Mathematica demonstration notebook is provided in the \href{https://github.com/sphinxteam/StructuredPrior_demo}{GitHub repository} is also provided.
We also describe in SM the difficulties to circumvent to generalize the analysis to a non-linear
activation function with random matrix theory.

%% file: files/acknowledgements.tex
This work is supported by the ERC under the European Union’s Horizon
2020 Research and Innovation Program 714608-SMiLe, as well as by the
French Agence Nationale de la Recherche under grant
ANR-17-CE23-0023-01 PAIL. 

We gratefully acknowledge the support of
NVIDIA Corporation with the donation of the Titan Xp GPU used for this
research. We thank Google Cloud for providing us access to their platform through the Research Credits Application program. 

We would also like to thank the Kavli Institute for Theoretical Physics (KITP)  for welcoming us during part of this research, with the support of the National Science Foundation under Grant No. NSF PHY-1748958.

We thank Ahmed El Alaoui for insightful discussions about
the proof of the Bayes optimal performance, and Remi Monasson for his
insightful lecture series that inspired partly this work.

%% file: files/supplementary.tex
\section{Definitions and notations}
\label{appendix:definitions_notations}

\input{files/supplementary/defs.tex}

 %\newpage

\section{Mutual information from the replica trick}
\label{sec:appendix:replicafreeen}
\input{files/supplementary/replica_fe.tex}
% \newpage

\section{Proof of the mutual information for the $\bv  \bv^\intercal$ case}
\label{appendix:proof_uu}
\input{files/supplementary/proof_uu.tex}

% \newpage
\section{Heuristic derivation of AMP from the two simples AMP algorithms}\label{appendix:AMP_derivation}
\label{appendix:amp_derivation}
\input{files/supplementary/amp_derivation.tex}
% \newpage

\section{Heuristic derivation of LAMP }
\label{appendix:derivation_lAMP}
\input{files/supplementary/lamp_derivation.tex}

%\newpage

\section{Transition from state evolution - stability}
\label{sec:app:stability}
\input{files/supplementary/transition_se.tex}
% \newpage

\section{Random matrix analysis of the transition}\label{appendix:rmt}
\input{files/supplementary/transition_rmt.tex}
% \newpage

\section{Phase diagrams of the Wishart model}\label{appendix:plots_vu}
\input{files/supplementary/plots_vu.tex}

% \newpage

%% file: files/supplementary/defs.tex
In this section we recall the models introduced in the main body of the article, and introduce the notations used throughout the Supplementary Material.

%%%%%%%%%%%%%%%%%%%%%%%
\subsection{Models}
%%%%%%%%%%%%%%%%%%%%%%%
\paragraph{Spiked Wigner model ($\bv \bv^\intercal$):}
 Consider an unknown vector (the spike) $\bv^\star \in \bbR^{p}$ drawn from a distribution $P_v$,  we observe a matrix $Y \in \bbR^{ p \times p}$ such that:
\begin{align}
	Y = \frac{1}{\sqrt{p}} {\bv^\star} {\bv^\star}^\intercal + \sqrt{\Delta} \xi \, ,\label{app:Wigner}
\end{align}
\noindent with symmetric noise $\xi \in \bbR^{p \times p}$ drawn from $\xi_{ij} \underset{\iid}{\sim} \mN\(0,1\)$ and $\Delta > 0$. The aim is to find back the hidden spike ${\bv^\star}$ from the observation of $Y$.

\paragraph{Spiked Wishart (or spiked covariance) model ($\bu \bv^\intercal$):}
Consider two unknown vectors ${\bu}^{\star}\in \bbR^{n}$  and ${\bf v}^{\star} \in \bbR^{p}$ drawn from distributions $P_u$ and $P_v$, we observe $Y \in \bbR^{n \times p}$ such that
\begin{align}
	Y = \frac{1}{\sqrt{p}} {\bu^\star} {\bv^\star}^\intercal + \sqrt{\Delta} \xi \,,\label{app:Wishart}
\end{align}
\noindent with noise $\xi \in \bbR^{n \times p}$ drawn $\xi_{i\mu} \underset{\iid}{\sim} \mN\(0,1\)$, $\Delta > 0$, and the goal is to find back the hidden spikes ${\bu}^{\star}$ and ${\bv}^\star$ from the observation of $Y$. We define the ratio between the spike dimensions $\beta = n/p$.

In either models, we are interested in the case where $\bv^{\star}$ is given by a generative model. In the setting studied here the generative model is a fully-connected single-layer neural network (a.k.a. generalised linear model) with Gaussian random weights $W\in\mathbb{R}^{p\times k}$, $W_{il}\underset{\iid}{\sim}\mathcal{N}(0,1)$ and latent variable $\bz^{\star}\in\mathbb{R}^{k}$ drawn from a given factorised distribution $P_{z}$,
\begin{align}
  \bv^{\star} = \varphi\left(\frac{1}{\sqrt{k}}W\bz^{\star}\right) \hhspace \text{ with } \hhspace z^{\star}_{l}\underset{\iid}{\sim} P_{z},
  \label{sec:model:glm}
\end{align}
\noindent where $\varphi:\mathbb{R}\to\mathbb{R}$ is the activation function, a real-valued function acting component-wise on $\mathbb{R}^{p}$ that can be deterministic or stochastic. An equivalent formulation of eq.~\eqref{sec:model:glm} is
\begin{align}
    \bv^{\star} {\sim} P_{\rm \out}\left(\cdot\Big|\frac{1}{\sqrt{k}}W\bz^{\star}\right).
    \label{sec:model:glm_2}
\end{align}
For instance, a deterministic layer with activation $\varphi$ is written in this formulation as $P_{\rm \out}(v|x) = \delta(v-\varphi(x))$. We define the compression rate of the signal as $\alpha = p/k$.

Although we will mainly focus on the single-layer model, some of our results apply more broadly to any generative prior with a well-defined free energy density in the thermodynamic limit. In particular, we will mention the example of a fully-connected multi-layer generative prior, given by
\begin{align}
    \bv^{\star} = \varphi^{(L)}\left(\frac{1}{\sqrt{k}}W^{(L)}\cdots\varphi^{(1)}\left(W^{(1)}\bz\right)\right) \hhspace \text{ with } \hhspace z^{\star}_{l}\underset{\iid}{\sim} P_{z}\label{sec:model:mlglm}
\end{align}
\noindent where now $\{\varphi^{(l)}\}_{1\leq l \leq L}$ are a family of real-valued component-wise activation functions and $W^{(l)}_{\nu_{l}\nu_{l-1}}\underset{\iid}{\sim}\mathcal{N}(0,1)$ are independently drawn random weights. The equivalent probabilistic formulation of the multi-layer case is 
\begin{align}
    \bv&\sim P^{(L)}_{\out}\left(\cdot\Big|\frac{1}{\sqrt{k_{L}}}W^{(L)}\textbf{h}^{(L)}\right), & \bv\in\bbR^p \notag\\
    \bh^{(L)}&\sim P^{(L-1)}_{\out}\left(\cdot\Big|\frac{1}{\sqrt{k_{L-1}}}W^{(L-1)}\textbf{h}^{(L-1)}\right),& \bh^{(L)}\in\bbR^{k_{L}} \notag\\
    &\hspace{3cm}\vdots &\notag\\
    \bh^{(2)} &\sim P_{\rm \out}^{(1)}\left(\cdot\Big|\frac{1}{\sqrt{k_{1}}}W^{(1)}\bz\right),& \bh^{(2)} \in \bbR^{k_2}\notag\\
    \bz &\underset{\iid}{\sim} P_{z},& \bz\in\bbR^{k_{1}}
\end{align}
\noindent where we introduced the hidden variables $\textbf{h}^{(l)}\in\mathbb{R}^{k_{l}}$ for $2\leq l\leq L$ and the family of densities $\left\{P_{\rm \out}^{(l)} \right\}_{1\leq l\leq L}$. In this case, we define the compression rate as the ratio between the dimensions of the latent variable in the first layer $\bz\in\mathbb{R}^{k_{1}}$ and the signal $\bv\in\mathbb{R}^{p}$, $\alpha = p/k_1$. It is also useful to define the compression at each layer, $\alpha_{l}=k_{l}/k_{1}$. The thermodynamic limit for this generative model is defined by taking $p\to\infty$ while keeping all $\alpha, \alpha_{l} \sim O(1)$, $1\leq l \leq L$. As one might expect, the single-layer generative prior is a particular case with $L=1$.

%%%%%%%%%%%%%%%%%%%%%%%%%%%%%%%%%%%%%%%%%%%%%%%
\subsection{Bayesian inference and posterior distribution}
%%%%%%%%%%%%%%%%%%%%%%%%%%%%%%%%%%%%%%%%%%%%%%%
Since the information about the generative model $P_{v}$ of the spike is given, the optimal estimator for $\bv^{\star}$ is the mean of its posterior distribution, $\hat{\bv}^{\text{opt}} = \mathbb{E}_{P(\bv^{\star}|Y)}\bv$, which in general reads
\begin{align}
    P(\bv^{\star}|Y) =\frac{1}{P(Y)} P_{v}(\bv^{\star})\prod\limits_{1\leq i<j\leq p}\frac{1}{\sqrt{2\pi\Delta}}e^{-\frac{1}{2\Delta}\left(Y_{ij}-\frac{v_{i}^{\star}v^{\star}_{j}}{\sqrt{p}}\right)^2}\label{eq:app:defs:posterioruu}\,,
\end{align}
\noindent for the $\bv\bv^{\intercal}$ model and by
\begin{align}
    P(\bv^{\star}|Y) =\frac{1}{P(Y)} P_{v}(\bv^{\star})\int_{\mathbb{R}^{n}}\dd\bu~P_{u}(\bu)\prod\limits_{1\leq i \leq p , 1 \leq \mu \leq n}\frac{1}{\sqrt{2\pi\Delta}}e^{-\frac{1}{2\Delta}\left(Y_{\mu i}-\frac{u_{\mu}^\star v_{i}^{\star}}{\sqrt{p}}\,,\right)^2}\label{eq:app:defs:posterioruv}
\end{align}
\noindent for the $\bu\bv^{\intercal}$ model. In both cases the evidence $P(Y)$ is fixed as the normalisation of the posterior. In the specific case of a single-layer generative model from eq.~\eqref{sec:model:glm}, we can be more explicit and write the prior for $\bv^{\star}$ explicitly
\begin{align}
    P_v(\bv^{\star}) = \int_{\mathbb{R}^{k}}\dd\bz^{\star}P_{z}\left(\bz^{\star}\right)\prod\limits_{i=1}^{p}P_{\rm \out}\left(v^{\star}_{i}\Big|\frac{1}{\sqrt{k}}\sum\limits_{l=1}^{k}W_{il}z_{l}^{\star}\right).\label{eq:app:intro:generative_prior}
\end{align}
The multi-layer case is written similarly by integrating over the intermediate hidden variables and their respective distributions. It is important to stress that we assume the structure of the generative model is known, i.e. $(P_{z},P_{\rm \out}, W)$ (and $P_u$ in the $\bu\bv^{\intercal}$  case) are given and the only unknowns of the problem are the spike $\bv^{\star}$ and the corresponding latent variable $\bz^{\star}$. This setting, in which the Bayesian estimator is optimal, is commonly refereed as the \emph{Bayes-optimal inference}.

In principle eqs.~\eqref{eq:app:defs:posterioruu} and \eqref{eq:app:defs:posterioruv} are of little use, since sampling from these high-dimensional distributions is a hard problem. Luckily, physicists have been dealing with high-dimensional distributions - such as the Gibbs measure in statistical physics - for a long time. The \emph{replica trick} and the \emph{approximate message passing} (AMP) algorithm presented in the main body of the paper are two of the statistical physics inspired techniques we borrow  to circumvent the hindrance of dimensionality.

\paragraph{Summary of the Supplementary Material:} A detailed account of the derivation of eq.~\eqref{main:free_entropy_uu} from the replica method is given in Section \ref{sec:appendix:replicafreeen}. Although the replica calculation is not mathematically rigorous, it gives a constructive method to compute the mutual information. The final expression can be made rigorous using an interpolation method, which we detail in Section \ref{appendix:proof_uu}. The sketch for the derivation of the AMP algorithm \ref{main:AMP_vv} and its associated spectral algorithm in eq.~\eqref{eq:spectral:gammas} are discussed respectively in Section \ref{appendix:amp_derivation} and \ref{appendix:lamp_derivation}. We detail the stability analysis of the state evolution equations leading to the transition point for generic activation function in Section \ref{sec:app:stability}, and finally we present a rigorous proof for the transition in the case of linear activation in Section \ref{appendix:rmt}.

%%%%%%%%%%%%%%%%%%%%%%%%%%%%%%%%%%%%%%
\subsection{Notation and conventions}
%%%%%%%%%%%%%%%%%%%%%%%%%%%%%%%%%%%%%%
\paragraph{Index convention:} In the whole paper, we use the convention that indices $\mu$, $i$ and $l$ correspond respectively to variables $\bu$, $\bv$ and $\bz$ such that $\mu\in [1:n]$, $i\in[1:p]$ and $l\in[1:k]$.

Unless otherwise stated, $\xi,\eta \in \bbR$ denote independent random variables variables distributed according to $\mN(0,1)$.

\paragraph*{Normalised second moments}
We define $\rho_{v}$ as the normalised second moments of the priors $P_{v}, P_{u}$ and $P_{z}$ respectively,
\begin{align}
  \rho_{v} =\lim\limits_{p\to\infty}\mathbb{E}_{P_{v}}\left[\frac{\bv^\intercal\bv}{p}\right], && \rho_{u} =\lim\limits_{n\to\infty}\mathbb{E}_{P_{u}}\left[\frac{\bu^\intercal\bu}{n}\right],
  &&\rho_{z} =\lim\limits_{z\to\infty}\mathbb{E}_{P_{z}}\left[\frac{\bz^\intercal\bz}{k}\right].
\end{align}
In the case we consider $P_z(\bz) = \prod\limits_{l=1}^{k}P_{z}(z_{l})$, $\rho_{z}$ is simply the one-dimensional second moment of $P_{z}$
\begin{align}
  \rho_{z} = \mathbb{E}_{P_{z}}z^2.
\end{align}
In the case $P_v$ is the single-layer generative model in eq.~\eqref{eq:app:intro:generative_prior} with $W_{il}\underset{\iid}{\sim}\mathcal{N}(0,1)$ and $z_{l}\underset{\iid}{\sim}P_{z}$, $\rho_{v}$ is self-averaging in the thermodynamic limit and is given by
\begin{align}
  \rho_{v} = \mathbb{E}_{Q_{\out}^{0}}v^2\,,
\end{align}
\noindent where $Q_{\out}^{0}$ is defined below in eq.~\eqref{eq:app:intro:defQ}.

\paragraph*{Denoising distributions}
The upshot of the replica calculation is that the high dimensional mutual information between the spike and the data $I(Y,\bv^{\star})$ is given by a simple one-dimensional expression, c.f. the right-hand side of the main part eq.~\eqref{main:free_entropy_uu}. This expression can be interpreted as the mutual information of a one-dimensional denoising problem.

Below we introduce the one-dimensional probability densities appearing in the factorised mutual information, from which the free energy and the AMP update equations are derived from:
\begin{align}
		Q_u(u; B,A) &\equiv \displaystyle \frac{1}{\Z_{u} (B,A)} P_u(u) e^{ -\frac{1}{2} A  u^2  + Bu} \, , \spacecase
		Q_z(\ud{z}; \gamma,\Lambda) &\equiv \displaystyle \frac{1}{\mZ_z (\gamma,\Lambda)} P_z(\ud{z}) e^{ - \frac{1}{2} \Lambda\ud{z}^2  + \gamma\ud{z}  } \, , \spacecase
		Q_{\rm out} (v,x;B,A,\ud{\omega},V) &\equiv  \displaystyle\frac{1}{\Z_{\rm out}(B,A,\ud{\omega},V)} e^{ -\frac{1}{2} Av^2 + Bv} P_{\rm out}\(v|x\)  e^{ -\frac{1}{2}V^{-1}  \(x - \ud{\omega}\)^2  } \, , \spacecase
		Q_{\rm out}^0 (v,x; \rho_z) &\equiv  Q_{\rm out} (v,x;0,0,0, \rho_z)  =  \displaystyle\frac{1}{\Z_{\rm out}^0} P_{\rm out}\(v|x\)  e^{ -\frac{1}{2\rho_z}x^2  } \, .
		\label{eq:app:intro:defQ}
\end{align}

\paragraph*{Free entropy terms}
\label{eq:app:defs:psis}
The mutual information density can be written in terms of the partition functions of the denoising distributions above as:
\begin{align}
	\Psi_{u} (x) & \equiv \EE_{\xi} \[ \mZ_{u}\( x^{1/2} \xi , x  \) \log \( \mZ_u\( x^{1/2} \xi ,x  \) \) \] \,, \spacecase
	\Psi_z (x) & \equiv \EE_{\xi} \[ \mZ_z\( x^{1/2} \xi ,x  \) \log \( \mZ_z\( x^{1/2} \xi ,x  \) \) \] \,, \spacecase
	\Psi_{\rm out} (x,y) & \equiv \EE_{\xi, \eta} \[\mZ_{\rm out}\( x^{1/2} \xi , x , y^{1/2} \eta  , \rho_z - y \) \log\( \mZ_{\rm out}\( x^{1/2} \xi , x , y^{1/2} \eta  , \rho_z - y \) \) \] \,.
\end{align}

\paragraph*{AMP update functions}
\label{appendix:definition:updates_amp}
Similarly, the update functions appearing in AMP are also given in terms of the moments of the above denoising distributions:
\begin{align}
	f_u(B,A) &\equiv \partial_B \log\(\mZ_u\) = \EE_{Q_u} \[ u \]\,,  \hhspace
	\partial_B f_u (B,A) \equiv \EE_{Q_u} \[ u^2  \] - (\ud{f}_{u})^2 \spacecase
	f_z(\gamma ,\Lambda) &\equiv  \partial_\gamma \log\(\mZ_z\) = \EE_{Q_z} \[ \ud{z} \]\,, \hhspace 
	\partial_\gamma f_z (\gamma ,\Lambda) \equiv  \EE_{Q_z} \[\ud{z}^2 \] - (f_z)^2 \spacecase	
	f_{v}(B,A,\ud{\omega},v) &\equiv \partial_B \log\(\mZ_{\rm out}\) = \mE_{Q_{\rm out}} \[ v \]\,, \hhspace
	\partial_B f_v(B,A,\ud{\omega},v) \equiv \mE_{Q_{\rm out}} \[ v^2  \] - (f_v)^2 \spacecase
	f_{\rm out} (B,A,\ud{\omega},v) &\equiv \partial_\omega \log \( \mZ_{\rm out} \) =V^{-1}\EE_{Q_{\rm out}} \[ x- \omega\] \,, \hhspace
	\partial_{\omega} f_{\rm out} (B,A,\ud{\omega},v) \equiv \displaystyle \frac{\partial f_{\rm out}}{\partial \omega}
	\label{eq:app:defs:fvfout}
\end{align}

%% file: files/supplementary/replica_fe.tex
In this section we give a derivation for the mutual information formula in main part~eq.~\eqref{main:free_entropy_uu} from the \emph{replica trick}. The derivation is detailed for the symmetric $\bv\bv^{\intercal}$ model, since the derivation for the asymmetric $\bu\bv^{\intercal}$ model follows exactly the same steps. In both cases, it closely follows the calculation of the replica free energy of the spiked matrix model with factorized prior in \cite{lesieur2017constrained}.

Before diving into the derivation, we note that the formula in main part~eq.~\eqref{main:free_entropy_uu} actually holds for any channel of the form
\begin{align}
	P(Y|\omega) = \prod\limits_{1\leq i<j\leq p}e^{g\left(Y_{ij},\omega_{ij}\right)}\,,
	\label{appendix:replicas:P_Y}
\end{align}
\noindent where $\omega\in\mathbb{R}^{p\times p}$ is a matrix with components $\omega_{ij}\equiv \frac{v_{i}v_{j}}{\sqrt{p}}$ and $g:\mathbb{R}^{2}\to\mathbb{R}$ is any two-dimensional real function such that $P(Y|\omega)$ is properly normalised. The gaussian noise in eq.~\eqref{Wigner} is a particular case given by $g(Y,\omega) = -\frac{1}{2\Delta}(Y-\omega)^2-\frac{1}{2}\log{2\pi\Delta}$.

The first step in the derivation is to note that the mutual information $I(Y,\bv^{\star})$ between the observed data $Y$ and the spike $\bv^{\star}$ can be writen as
\begin{align}
	I(Y,\bv^{\star}) = \frac{1}{4\Delta}\mathbb{E}_{P_{v}}\left[{\bv^{\star}}^\intercal\bv^{\star}\right]^2 - \mathbb{E}_{Y}\log{\mathcal{Z}(Y)}\label{eq:app:replicas:mutualinfo}\,,
\end{align}
\noindent where
\begin{align}
	\mathcal{Z}(Y) = \int_{\mathbb{R}^{p}}\dd \bv~P_{v}(\bv)\prod\limits_{1\leq i<j\leq p}e^{g(Y_{ij},\omega_{ij})-g(Y_{ij},0)}.
\end{align}
Note that since the data is generated from a planted spike $\bv^{\star}$, we have $Y=Y(\bv^{\star})$, and therefore the partition function $\mathcal{Z}$ depends on $\bv^{\star}$ implicitly through $Y$.

\subsection{Derivation of the replica free energy for the $\bv\bv^{\intercal}$ model}
The partition function $\mathcal{Z}$ is a $p$-dimensional integral, and computing the average over $Y$ (a $p\times p$ integral) of $\log{\mathcal{Z}}$ seems hopeless. The replica trick is a way to surmount this hindrance. It consists of writing
\begin{align}
	\mathbb{E}_{Y}\log{Z} = \lim\limits_{r\to 0^{+}}\frac{1}{r}\left(\mathbb{E}_{Y}\mathcal{Z}^{r}-1\right).
\end{align}
Note that $\mathcal{Z}^{r}$ is the partition function of $r$ non-interacting copies (named in the physics literature and hereafter \emph{replicas}) of the initial system. The average over the \emph{replicated partition function} $\mathcal{Z}^{r}$ can be conveniently written as
\begin{align}
	\mathbb{E}_{Y}\mathcal{Z}^{r}
	&=\int\prod\limits_{1\leq i<j\leq p}\dd Y_{ij}~e^{g(Y_{ij},0)}\int_{\mathbb{R}^{p\times (r+1)}}\prod\limits_{a=0}^{r}\dd\bv^a P_{v}\left(\bv^{a}\right)\prod\limits_{a=0}^{r}\prod\limits_{1\leq i<j\leq p} e^{g\left(Y_{ij},\omega_{ij}^{a}\right)-g\left(Y_{ij},0\right)}\,,\label{eq:app:replicas:avgZ}
\end{align}
\noindent where in the second line we have defined
\begin{align}
	\bv^{a} =
		\begin{cases}
			\bv^{\star} & \text{ for } a=0\\
			\bv^{a} & \text{ for } 1\leq a\leq r\,.
		\end{cases}
\end{align}

%%%%%%%%%%%%%%%%%%%%%%%%%%%%%%%%%%%
\paragraph*{Averaging over $Y$}
%%%%%%%%%%%%%%%%%%%%%%%%%%%%%%%%%%%
The key observation to simplify the integrals in eq.~\eqref{eq:app:replicas:avgZ} is to note that $\omega_{ij}$ is of order $1/\sqrt{p}$, and therefore in the large-$p$ limit of interest, we can keep only terms of order $1/p$,
\begin{align}
   \exp \( \displaystyle \sum\limits_{a=0}^{r}\left[g(Y_{ij},\omega^{a}_{ij}) - g(Y_{ij},0)\right] \) &= 1+\sum\limits_{a=0}^{r}\left(\partial_{\omega}g\right)_{\omega=0}\omega^{a}_{ij}+\frac{1}{2}\sum\limits_{a=0}^{r}\left(\partial^2_{\omega}g\right)_{\omega=0}\left(\omega^{a}_{ij}\right)^2\notag\\
    &\hspace{1cm}+\frac{1}{2}\sum\limits_{a,b=0}^{r}\left(\partial_{\omega}g\right)^2_{\omega=0}\omega^{a}_{ij}\omega^{b}_{ij}+\mO\left(p^{-3/2}\right)\label{eq:app:replica:expansion}
\end{align}
From the normalisation condition of $P(Y|\omega)$, we can derive the following relations
\begin{align}
\int\prod\limits_{1\leq i<j\leq p}\dd Y_{ij}~e^{g(Y_{ij},0)} &= 1,\notag\\
\int\prod\limits_{1\leq i<j\leq p}\dd Y_{ij}~e^{g(Y_{ij},0)}\left(\partial_{\omega}g\right)_{\omega=0} &= 0,\notag\\
\int\prod\limits_{1\leq i<j\leq p}\dd Y_{ij}~e^{g(Y_{ij},0)}\left[\partial^2_{\omega}g+\left(\partial_{\omega}g\right)^2\right]_{\omega=0} &= 0.
\label{appendix:replica:bayes_relations}
\end{align}
Further defining
\begin{align}
    \Delta^{-1} = \int\prod\limits_{1\leq i<j\leq p}\dd Y_{ij}~e^{g(Y_{ij},0)}\left(\partial_{\omega}g\right)^2_{\omega=0},
\end{align}
\noindent allows us to evaluate the integral over $Y$ term by term in the expansion in eq.~\eqref{eq:app:replica:expansion},
\begin{align}
\mathbb{E}_{Y}\mathcal{Z}^{r} &= \int_{\mathbb{R}^{p\times (r+1)}}\prod\limits_{a=0}^{r}\dd\bv^{a}~P_{v}\left(\bv^{a}\right)\prod\limits_{1\leq i<j\leq p}\left[1+\frac{1}{2\Delta}\sum\limits_{0\leq a<b\leq r}\omega^{a}_{ij}\omega^{b}_{ij}+\mO\left(p^{-3/2}\right)\right]
\notag\\
&=\int_{\mathbb{R}^{p\times (r+1)}}\prod\limits_{a=0}^{r}\dd\bv^{a}~P_{v}\left(\bv^{a}\right)\prod\limits_{1\leq i<j\leq p}e^{\frac{1}{2\Delta}\sum\limits_{0\leq a<b\leq r}\omega^{a}_{ij}\omega^{b}_{ij}}+\mO\left(p^{-3/2}\right).
\end{align}
The upshot of this expansion is that on the large-$p$ limit $\Delta$ is the only relevant parameter we need from the channel. Therefore, from the perspective of the mutual information density, a channel with parameter $\Delta$ is completely equivalent to a Gaussian channel with variance $\Delta$. This property is known as \emph{channel universality} \cite{lesieur2017constrained}.
%%%%%%%%%%%%%%%%%%%%%%%%%%%%%%%%%%%%%%%%%%%%%%%%%%%
\paragraph*{Rewritting as a saddle-point problem}
Note that we can rewrite
%%%%%%%%%%%%%%%%%%%%%%%%%%%%%%%%%%%%%%%%%%%%%%%%%%%
\begin{align}
    \sum\limits_{1\leq i<j\leq p}\omega^{a}_{ij}\omega_{ij}^{b} = \frac{1}{p}\sum\limits_{1\leq i<j\leq n}v_{i}^{a}v_{j}^{a}v_{i}^{b}v_{j}^{b}=\frac{p}{2}~\left(q_{v}^{ab}\right)^2\,,
\end{align}
\noindent where we defined the overlap between two replicas as $q_{v}^{ab}=p^{-1}\sum\limits_{i=1}^{p}v_{i}^{a}v_{i}^{b}$. This allows us to write the average over the replicated partition function as a function of a set of order parameters $q_v^{ab}$, and therefore to factorise all the index $i$ dependence of the exponential,
\begin{align}
\mathbb{E}_{Y}\mathcal{Z}^{r} &= \int_{\mathbb{R}^{p\times (r+1)}}\prod\limits_{a=0}^{r}\dd\bv^{a}~P_{v}\left(\bv^{a}\right)e^{\frac{p}{4\Delta}\sum\limits_{0\leq a<b\leq r}\left(q_{v}^{ab}\right)^2}.\label{eq:app:partialZ}
\end{align}
Since the expression above only depends on $q_{v}^{ab}$ now, we exchange the integral over the spike for an integral over this order parameter by introducing
\begin{align}
    1 &\propto \int_{\mathbb{R}^{(r+1) \times (r+1)}}\prod\limits_{0\leq a< b\leq r}\dd q_{v}^{ab}\prod\limits_{0\leq a< b\leq r}\delta\left(\sum\limits_{i=1}^{p}q_{v}^{ab}-p q^{ab}_{v}\right)\notag\\
    &\propto \int_{\mathbb{R}^{(r+1) \times (r+1)}}\prod\limits_{0\leq a< b\leq r}\dd q_{v}^{ab}\int_{\left(i\mathbb{R}\right)^{(r+1) \times (r+1)}}\prod\limits_{0\leq a< b\leq r}\hat{q}_{v}^{ab} e^{-p\sum\limits_{0\leq a<b\leq r}\hat{q}^{ab}_{v}q^{ab}_{v}+\sum\limits_{0\leq a<b\leq r}\hat{q}^{ab}_{v}\sum\limits_{i=1}^{p}v^{a}_{i}v_{i}^{b}}
\end{align}
Note that we neglected some constants and made a rotation to the complex axis over the Fourier integral. These will not be important for the argument that follows.

Inserting this identity in eq.~\eqref{eq:app:partialZ} yields
\begin{align}
    \mathbb{E}_{Y}\mathcal{Z}^{r} &\propto \int_{\mathbb{R}^{(r+1) \times (r+1)}}\prod\limits_{0\leq a< b\leq r}\dd q_{v}^{ab}\int_{\left(i\mathbb{R}\right)^{(r+1) \times (r+1)}}\prod\limits_{0\leq a< b\leq r}\hat{q}_{v}^{ab}~e^{p\Phi^{(r)}\left(q^{ab},\hat{q}^{ab}\right)}\notag\,,\\
    \Phi^{(r)}(q_{v}^{ab},\hat{q}_{v}^{ab}) &= \frac{1}{4\Delta}\sum\limits_{0\leq a<b\leq r}\left(q_{v}^{ab}\right)^2-\sum\limits_{0\leq a< b\leq r}\hat{q}_{v}^{ab}q_{v}^{ab}+\Psi_{v}^{(r)}(\hat{q}_{v}^{ab})\,,
\end{align}
\noindent where $\Psi_{v}^{(r)}(\hat{q}_{v}^{ab})$ contains all the information about the prior $P_{v}$:
\begin{align}
   \Psi_{v}^{(r)}(\hat{q}_{v}^{ab}) = \frac{1}{p}\log{\int_{\mathbb{R}^{p\times (r+1)}}\prod\limits_{a=0}^{r}\dd\bv^{a}~P_{v}\left(\bv^{a}\right)}\prod\limits_{i=1}^{p}e^{\sum\limits_{0\leq a<b\leq p}v_{i}^{a}\hat{q}_{v}^{ab}v_{i}^{b}}\,.
\end{align}
Note that when the prior factorises, $P_{v}(\bv) = \prod\limits_{i=1}^{p}P_{v}(v_{i})$, $\Psi_{v}^{(r)}$ is given by a simple one-dimensional integral. However in the case of a generative model for $\bv$, $P_{v}$ is kept general.

We are interested in the mutual information density in the thermodynamic limit. According to eq.~\eqref{eq:app:replicas:mutualinfo}, this is given by
\begin{align}
    \lim\limits_{p\to\infty}i_{p}(Y,\bv^{\star}) &= \lim\limits_{p\to\infty} \frac{1}{p}I(Y,\bv^{\star}) =\frac{1}{4\Delta} \lim\limits_{p\to\infty}\mathbb{E}_{P_{v}}\left[\frac{{\bv^\star}^\intercal\bv^{\star}}{p}\right]-\lim\limits_{p\to\infty}\frac{1}{p}\mathbb{E}_{Y}\log{\mZ}\\
    &=\frac{\rho_{v}^2}{4\Delta} -\lim\limits_{r\to 0^{+}}\frac{1}{r}\left(\lim\limits_{p\to\infty}\frac{1}{p}\mathbb{E}_{Y}\mathcal{Z}^{r}\right).
\end{align}
\noindent where we assumed that $\rho_{v}$, the re-scaled second moment of $P_{v}$, remains finite and that we can commute the $r\to 0^{+}$ and the $p\to\infty$ limit. Since $\mathbb{E}_{Y}\mathcal{Z}^{r}$ is given in terms of an integral weighted by $e^{p\Psi^{(r)}}$, in the limit $p\to\infty$ the integral will be dominated by the configurations of $(q_{v}^{ab},\hat{q}^{ab})$ that extremise the potential $\Psi^{(r)}$. This extremality condition, known as the Laplace method, yields the following \emph{saddle-point equations},
\begin{align}
    \hat{q}^{ab}_{v} = \frac{1}{2\Delta}q_{v}^{ab}\,, && q_{v}^{ab} = \lim\limits_{p\to\infty}\partial_{\hat{q}_{v}}\Psi_{v}^{(r)}(\hat{q}^{ab}_{v}).
\end{align}
\noindent where we also assume that $P_{v}$ is such that $\Psi_{v}^{(r)}$ remains well defined in the limit $p\to\infty$.

%%%%%%%%%%%%%%%%%%%%%%%%%%%%%%%%%%%%%%%%%%%
\paragraph*{Replica symmetric solution}
%%%%%%%%%%%%%%%%%%%%%%%%%%%%%%%%%%%%%%%%%%%
Enforcing the first saddle-point equation allow us to write
\begin{align}
    \lim\limits_{p\to\infty}\frac{1}{p}\mathbb{E}_{Y}\mathcal{Z}^{r} = \underset{q_{v}^{ab}}{\extr}\left[-\frac{1}{2\Delta}\sum\limits_{0\leq a<b\leq r}\left(q^{ab}_{v}\right)^2+\lim\limits_{p\to\infty}\Psi_{v}^{(r)}\left(\frac{q_{v}^{ab}}{\Delta}\right)\right]\label{eq:app:replicas:partialresult}
\end{align}
Solving this extremisation problem for general matrices is cumbersome. We therefore restrict ourselves to solutions that are \emph{replica symmetric}
\begin{align}
    q^{ab}_{v} = q_{v} \qquad\text{ for }\qquad 0\leq a\leq r.
\end{align}
The replica symmetry assumption might seen restrictive, but it is justified in the Bayes-optimal case under consideration - see \cite{nishimori2001statistical}. Replica symmetry allow us to factor the $r$ dependence
explicitly for each term,
\begin{align}
    \sum\limits_{0\leq a<b\leq r}\left(q_{v}^{ab}\right)^2 = \frac{r(r+1)}{2}q_{v}^2, && \sum\limits_{0\leq a<b\leq r}v_{i}^{a}q_{v}^{ab}v_{i}^{b}=q_{v} v^{\star}\sum\limits_{a=1}^{r}v^{a}_{i}+q_{v}\sum\limits_{a,b=1}^{r}v^{a}_{i}v^{b}_{i}
\end{align}
\noindent the last sum that couples $a,b$ can be decoupled using
\begin{align}
    e^{\frac{q_{v}}{2}\sum\limits_{a,b=1}^{r}v_{i}^{a}v_{i}^{b}} =\mathbb{E}_{\xi}\left[e^{-\sqrt{q_{v}}\xi\sum\limits_{a=1}^{r}\left(v_i^{a}\right)^2}\right]
\end{align}
\noindent where $\xi\sim\mathcal{N}(0,1)$. This transformation factorise $\Psi^{(r)}_{v}$ in replica space,
\begin{align}
    \Psi^{(r)}_{v}(q_{v}) &= \frac{1}{p}\log\int_{\mathbb{R}^{p}}\dd\bv^{\star}~P_{v}\left(\bv^{\star}\right)\int_{\mathbb{R}}\frac{\dd\xi}{\sqrt{2\pi}}e^{-\frac{1}{2}\xi^2}\left[\int_{\mathbb{R}^{p}}\dd\bv~P_{v}(\bv)~\prod\limits_{i=1}^{p}e^{-\frac{q_{v}}{2\Delta}v_{i}^2+\left(\frac{q_v}{\Delta}v_{i}^{\star}+\sqrt{\frac{q}{\Delta}}\xi\right) v_{i}}\right]^{r}\notag\\
    &\underset{r\to 0^{+}}{=}\frac{r}{p}\mathbb{E}_{\xi,P_{v}(v^{\star})}\log{\int_{\mathbb{R}^{p}}\dd\bv~P_{v}(\bv)~\prod\limits_{i=1}^{p}e^{-\frac{q_{v}}{2\Delta}v_{i}^2+\left(\frac{q_v}{\Delta}v_{i}^{\star}+\sqrt{\frac{q}{\Delta}}\xi\right) v_{i}}}+\mO\left(r^{2}\right).
\end{align}
\noindent allowing us to take the $r\to 0^{+}$ limit explicitly, and giving the following partial result
\begin{align}
    \lim\limits_{p\to\infty}i_{p}(Y,\bv^{\star}) = \frac{\rho_{v}^2}{4\Delta}+\underset{q_{v}}{\extr}\left[\frac{1}{4\Delta}q_{v}^2-\lim\limits_{p\to\infty}\Psi_{v}\left(\frac{q_{v}}{\Delta}\right)\right]\,,\label{eq:app:finalvvlayer}
\end{align}
\noindent where
\begin{align}
    \Psi_{v}\left(\frac{q_{v}}{\Delta}\right) =\lim\limits_{r\to 0^{+}}\Psi_{v}^{(r)} = \frac{1}{p} \mathbb{E}_{\xi, P_{v}(\bv^{\star})}\log{\mathbb{E}_{P_{v}(\bv)}\left[\prod\limits_{i=1}^{p}e^{-\frac{q_{v}}{2\Delta}v_{i}^2+\left(\frac{q_v}{\Delta}v_{i}^{\star}+\sqrt{\frac{q_{v}}{\Delta}}\xi\right) v_{i}}\right]}.\label{eq:app:replicas:Psiv}
\end{align}

%%%%%%%%%%%%%%%%%%%%%%%%%%%%%%%%%%%%%%%%%%%%%%%%%%%%%%
\paragraph*{Interpretations of $\Psi_{v}$ as a mutual information:}
%%%%%%%%%%%%%%%%%%%%%%%%%%%%%%%%%%%%%%%%%%%%%%%%%%%%%%
The prior term $\Psi_{v}$ in the free energy has an interesting interpretation as the mutual information of an effective denoising problem over $\bv$. To see this, we complete the square in the exponential of eq.~\eqref{eq:app:replicas:Psiv},
\begin{align}
    \Psi_{v}\left(x\right) &= \frac{1}{p}\mathbb{E}_{\xi,P_{v}(\bv^{\star})}\log{\int_{\mathbb{R}^{p}}\dd\bv~P_{v}(\bv)\prod\limits_{i=1}^{p}e^{-\frac{x}{2}\left[v_{i}-\left(v_{i}^{\star}+x^{-1/2}\xi\right)\right]^2+\frac{x}{2}\left(v_{i}^{\star}+x^{-1/2}\xi\right)^2}},\notag\\
    &=\frac{x}{2p}\mathbb{E}_{\xi,P_{v}(\bv^\star)}\sum\limits_{i=1}^{p}\left(v_{i}^{\star}+x^{-1/2}\xi\right)^2+ \frac{1}{p}\mathbb{E}_{\xi,P_{v}(v^{\star})}\log{\int_{\mathbb{R}^{p}}\dd\bv~P_{v}(\bv)\prod\limits_{i=1}^{p}e^{-\frac{x}{2}\left[v_{i}-\left(v_{i}^{\star}+x^{-1/2}\xi\right)\right]^2}},\notag\\
    &=\frac{x}{2}\mathbb{E}_{P_{v}}\left[\frac{\bv^\intercal\bv}{p}\right]+\frac{1}{2}+\frac{1}{p}\mathbb{E}_{\xi,P_{v}(\bv^{\star})}\log{\int_{\mathbb{R}^{p}}\dd\bv~P_{v}(\bv)\prod\limits_{i=1}^{p}e^{-\frac{x}{2}\left[v_{i}-\left(v_{i}^{\star}+x^{-1/2}\xi\right)\right]^2}}.\label{eq:app:replicas:Psivcompletesquare}
\end{align}
The last integral is a convolution between the prior $P_{v}$ and a un-normalised Gaussian. Up to an aditive constant it admits a natural representation as the mutual information of a denoising problem,
\begin{align}
    \frac{1}{p}\mathbb{E}_{\xi,P_{v}(\bv^{\star})}\log{\int_{\mathbb{R}^{p}}\dd\bv~P_{v}(\bv)\prod\limits_{i=1}^{p}e^{-\frac{x}{2}\left[v_{i}-\left(v_{i}^{\star}+x^{-1/2}\xi\right)\right]^2}} = -\frac{1}{p}I(\bv^{\star};\bv^{\star}+x^{-1/2}\xi)-\frac{1}{2}\,.
\end{align}
Putting together with eq.~\eqref{eq:app:replicas:Psivcompletesquare} and taking the limit,
\begin{align}
    \lim\limits_{p\to\infty}\Psi_{v}\left(\frac{q_{v}}{\Delta}\right) = \frac{q_{v}\rho_{v}}{2\Delta}-\lim\limits_{p\to\infty}\frac{1}{p}I\left(\bv^{\star};\bv^{\star}+\sqrt{\frac{\Delta}{q_{v}}}\xi\right).
\end{align}
Together with eq.~\eqref{eq:app:finalvvlayer}, this representation lead to eq.~\eqref{eq:information_theory:limip} in the main article.

Interestingly, the signal to noise ratio in the effective denoising problem is proportional to $\Delta$ and inversely proportional to the overlap $q_{v}$. This is quite intuitive: when $\Delta\gg 1$ (or the overlap with the ground truth is small), denoising is hard. On the other hand, when $\Delta=0$ the mutual information reaches its upper bound, given by the entropy of $P_{v}$.
%%%%%%%%%%%%%%%%%%%%%%%%%%%%%%%%%%%%%%%%%%%%%%%%%%%%%
\subsection{Free energy for the $\bu\bv^{\intercal}$ model}
\label{sec:app:replicas:wishart}
%%%%%%%%%%%%%%%%%%%%%%%%%%%%%%%%%%%%%%%%%%%%%%%%%%%%%
The exact same steps outlined above can be followed for the spiked Wishart model with spikes $\bu^{\star}\in\mathbb{R}^{n}$ and $\bv^{\star}\in\mathbb{R}^{p}$ drawn from non-factorisable priors $P_{u}$ and $P_{v}$ respectively. In this case, the free energy density associated with the following partition function
\begin{align}
    \mathcal{Z}^{uv}(Y) = \int_{\mathbb{R}^{p}}\dd\bv~P_{v}\left(\bv\right)\int_{\mathbb{R}^{n}}\dd\bu~P_{u}\left(\bu\right)\prod\limits_{\mu=1}^{n}\prod\limits_{i=1}^{p}e^{g\left(Y_{\mu i},\frac{u_{\mu}v_{i}}{\sqrt{p}}\right)-g\left(Y_{\mu i},0\right)}
\end{align}
\noindent is given by
\begin{align}
    \lim\limits_{p\to\infty}\frac{1}{p}\mathbb{E}_{Y}\log{\mathcal{Z}^{uv}} = \underset{q_{u},q_{v}}{\extr}\left[\frac{\beta}{2\Delta}q_{u}q_{v}-\lim\limits_{p\to\infty}\Psi_{v}\left(\beta\frac{ q_{u}}{\Delta}\right)-\beta\lim\limits_{n\to\infty}\Psi_{u}\left(\frac{q_{v}}{\Delta}\right)\right]
\end{align}
\noindent with $\beta = n/p$ fixed. The functions $\Psi_{v},\Psi_{u}$ are given by
\begin{align}
    \Psi_{u}\left(\beta\frac{q_{v}}{\Delta}\right) &= \frac{1}{n}\mathbb{E}_{\xi, P_{u}(\bu^{\star})}\log\int_{\mathbb{R}^{n}}\dd\bu~P_{u}(\bu)\prod\limits_{\mu=1}^{n}e^{-\beta\frac{q_{v}}{2\Delta}u_{\mu}^2+\left(\beta\frac{q_{u}}{\Delta}u^{\star}_{\mu}+\sqrt{\beta\frac{q_{v}}{\Delta}}\xi\right)u_{\mu}}\notag\\
    \Psi_{v}\left(\frac{q_{u}}{\Delta}\right) &= \frac{1}{p}\mathbb{E}_{\xi, P_{v}(\bv^{\star})}\log\int_{\mathbb{R}^{p}}\dd\bv~P_{v}(\bv)\prod\limits_{i=1}^{p}e^{-\frac{q_{u}}{2\Delta}v_{i}^2+\left(\frac{q_{v}}{\Delta}v^{\star}_{i}+\sqrt{\frac{q_{u}}{\Delta}}\xi\right)v_{\mu}}
\end{align}
%%%%%%%%%%%%%%%%%%%%%%%%%%%%%%%%%%%%%%%%%%%%%%%%%
\subsection{Application to generative priors}
\label{sec:app:replicas:application}
%%%%%%%%%%%%%%%%%%%%%%%%%%%%%%%%%%%%%%%%%%%%%%%%%

%%%%%%%%%%%%%%%%%%%%%%%%%%%%%%%%%%%%%%%%%%%%%%%%%
\paragraph*{Generalised linear model prior}
%%%%%%%%%%%%%%%%%%%%%%%%%%%%%%%%%%%%%%%%%%%%%%%%%
The expression we derived for the mutual information density in the $\bv\bv^{\intercal}$ model is valid for any prior $P_v$ as long as $\Psi_{v}$ is well defined in the thermodynamic limit. For the specific case when
\begin{align}
    P_{v}(\bv) = \int_{\mathbb{R}^{k}}\left(\prod\limits_{l=1}^{k}\dd z_{l}~P_{z}(z_{l})\right)\prod\limits_{i=1}^{p}P_{\out}\left(v_{i}\Big|\frac{1}{\sqrt{k}}\sum\limits_{l=1}^{k}W_{il}z_{l}\right)\,,
\end{align}
\noindent with $W_{il}\underset{\iid}{\sim}\mathcal{N}(0,1)$, $\Psi_{v}$ is, up to a global $1/\alpha$ scaling, the Bayes-optimal free energy of a generalised linear model with channel given by
\begin{align}
    \tilde{P}_{\out}\left(v|x;\xi,q_{v}\right) = P_{\out}(v|x)e^{-\frac{q_{v}}{2\Delta}v^2+\sqrt{\frac{q_{v}}{\Delta}}\xi v}\,,
\end{align}
\noindent and factorised prior $P_z$. The expression for this free energy is well known - see for example \cite{Barbier2017c} for a derivation and \cite{Barbier2017c} for a proof - and reads
\begin{align}
    \lim\limits_{p\to\infty}\Psi_{v} = \frac{1}{\alpha}\underset{q_{z},\hat{q}_{z}}{\extr}\left[-\frac{1}{2}q_{z}\hat{q}_{z}+\alpha \Psi_{\out}\left(\frac{q_{v}}{\Delta},q_{z}\right)+\Psi_{z}\left(\hat{q}_{z}\right)\right]
\end{align}
\noindent where the functions $\Psi_{\out}$ and $\Psi_{z}$ are defined in eq.~\eqref{eq:app:defs:psis}. Inserting this expression in our general formula for the mutual information density eq.~\eqref{eq:app:finalvvlayer} give us
\begin{align}
    \lim\limits_{p\to\infty} i_{p} =\frac{\rho_{v}}{4\Delta}+ \underset{q_{v},q_z,\hat{q}_{z}}{\extr}\left[\frac{1}{4\Delta}q_{v}^2+\frac{1}{2\alpha}\hat{q}_{z}q_{z}-\Psi_{\out}\left(\frac{q_{v}}{\Delta},q_{z}\right)-\frac{1}{\alpha}\Psi_{z}(\hat{q}_{z})\right]
    \label{eq:app:replicas:mutualinfo}
\end{align}
\noindent which is precisely the result from eq.~\eqref{main:free_entropy_uu}. The extremisation problem in eq.~\eqref{eq:app:replicas:mutualinfo} is solved by looking for the directions $(q_{v},\hat{q}_{z},q_{z})$ of zero gradient of the potential $\Psi_{v}$. These saddle-point equations are known in this context as \emph{state evolution equations}, and they can be conveniently written in terms of the auxiliary function we defined in Section \ref{eq:app:defs:psis}, equations  (\ref{eq:app:intro:defQ}-\ref{eq:app:defs:fvfout}) as
\begin{align}
    q_{v} &= 2\partial_{q_{v}}\Psi_{\out}\left(\frac{q_{v}}{\Delta},q_{z}\right) = \mathbb{E}_{\xi,\eta}\left[\mathcal{Z}_{\out}\left(\sqrt{\frac{q_{v}}{\Delta}}\xi,\frac{q_{v}}{\Delta},\sqrt{q_{z}}\eta,\rho_{z}-q_{z}\right)f_{v}\left(\sqrt{\frac{q_{v}}{\Delta}}\xi,\frac{q_{v}}{\Delta},\sqrt{q_{z}}\eta,\rho_{z}-q_{z}\right)^2\right]\notag\\
    \hat{q}_{z} &= 2\alpha\partial_{q_{z}}\Psi_{\out}\left(\frac{q_{v}}{\Delta},q_{z}\right) = \mathbb{E}_{\xi,\eta}\left[\mathcal{Z}_{\out}\left(\sqrt{\frac{q_{v}}{\Delta}}\xi,\frac{q_{v}}{\Delta},\sqrt{q_{z}}\eta,\rho_{z}-q_{z}\right)f_{\out}\left(\sqrt{\frac{q_{v}}{\Delta}}\xi,\frac{q_{v}}{\Delta},\sqrt{q_{z}}\eta,\rho_{z}-q_{z}\right)^2\right]\notag\\
    q_{z} &= 2\partial_{\hat{q}_{z}}\Psi_{z}\left(\hat{q}_{z}\right) = \mathbb{E}_{\xi}\left[\mathcal{Z}_{z}\left(\sqrt{\hat{q}_{z}}\xi,\hat{q}_{z}\right)f_{z}\left(\sqrt{\hat{q}_{z}}\xi,\hat{q}_{z}\right)^2\right]
    \label{eq:app:replicas:SE}
\end{align}
%%%%%%%%%%%%%%%%%%%%%%%%%%%%%%%%%%%%%%%%%%%%%%%%%
\paragraph*{Multi-layer prior}
%%%%%%%%%%%%%%%%%%%%%%%%%%%%%%%%%%%%%%%%%%%%%%%%%
The multi-layer prior can be conveniently written as
\begin{align}
    P_{v}(\bv) = \int \prod\limits_{l=1}^{L}\prod\limits_{\nu_{l}=1}^{k_{l}}\dd h^{(l)}_{\nu_{l}}P_{\out}^{(l-1)}\left(h^{(l)}_{\nu_{l}}\Big|\frac{1}{\sqrt{k_{l-1}}}\sum\limits_{\nu_{l-1}=1}^{k_{l-1}}W^{(l-1)}_{\nu_{l}\nu_{l-1}}h_{\nu_{l-1}}\right)\prod\limits_{i=1}^{p}P^{(L)}_{\out}\left(v_{i}\Big|\frac{1}{\sqrt{k_{L}}}\sum\limits_{\nu_{L}=1}^{k_{L}}W_{i \nu_{L}}h_{L}\right)\,,
\end{align}
\noindent where we define $\textbf{h}^{(1)} \equiv \bz\in\mathbb{R}^{k_{1}}$ and $P_{\out}^{(0)}\equiv P_{z}$. As in the single-layer case, the Bayes-optimal free energy of $P_{v}$ has been computed in \cite{manoel2017multi}, and in our notation it is written as
\begin{align}
    \lim\limits_{p\to\infty}\Psi_{v} = \frac{1}{\alpha}\underset{\{\hat{q}_{l},q_{l}\}_{1\leq l\leq L}}{\extr}\left[-\frac{1}{2}\sum\limits_{l=1}^{L}\alpha_{l}\hat{q}_{l}q_{l}+\alpha\Psi_{\out}\left(\frac{q_{v}}{\Delta},q_{L}\right)+\sum\limits_{l=2}^{L}\alpha_{l}\Psi_{\out}\left(\hat{q}_{l},q_{l-1}\right)+\Psi_{z}\left(\hat{q}_{1}\right)\right]\,,
\end{align}
\noindent where in this case $\alpha = p/k_1$ and we defined $\alpha_{l} = k_{l}/k_{1}$ for $1\leq l\leq L$ (note in particular that $\alpha_1 = 1$). The $(\hat{q}_{l},q_{l})$ are the overlaps of the hidden variables $\textbf{h}^{(l)}$ at each layer, and to be consistent with the shorthand notation introduced we have $(\hat{q}_{1},q_{1})=(\hat{q}_{z},q_{z})$. Inserting this expression in our general formula for the mutual information density eq.~\eqref{eq:app:finalvvlayer}:
\begin{align}
    \lim\limits_{p\to\infty}i_{p} = \frac{\rho_{v}}{4\Delta}+\underset{q_{v},\{\hat{q}_{l},q_{l}\}_l}{\extr}\left[\frac{1}{4\Delta}q_{v}^2+\frac{1}{2\alpha}\sum\limits_{l=1}^{L}\alpha_{l}\hat{q}_{l}q_{l}-\frac{1}{\alpha}\sum\limits_{l=2}^{L}\alpha_{l}\Psi_{\out}\left(\hat{q}_{l},q_{l-1}\right)-\Psi_{\out}\left(\frac{q_{v}}{\Delta},q_{L}\right)-\frac{1}{\alpha}\Psi_{z}\left(\hat{q}_{1}\right)\right].
\end{align}

%% file: files/supplementary/proof_uu.tex
In this section, we present a proof of the theorem~\ref{theorem_uu} in the main part,
for the mutual information of Wigner model eq.~\eqref{app:Wigner} with structured prior 
\begin{align}
  Y = \frac{1}{\sqrt{p}} {\bv^\star} {\bv^\star}^\intercal + \sqrt{\Delta} \xi \, ,
\end{align}
\noindent where the spike $\bv^{\star}\in\mathbb{R}^{p}$ is drawn from
$P_{v}$. The proof is based on Guerra Interpolation \cite{Guerra2003,korada2009exact}.

\subsection{Notations, free energies, and Gibbs average}
The mutual information being invariant to reparametrization, we shall
work instead inside this section with the following notations:
\begin{align}
  Y = {\sqrt{\frac{\lambda}p}} {\bv^\star} {\bv^\star}^\intercal + \xi \, ,
\end{align}
where $\lambda$ is the signal to noise ratio. Up to the
reparametrization, it corresponds to our model with
$\lambda=\Delta^{-1}$. Our aim is to compute $\frac{I(Y;\bv )}p$.

While the information theoretic notation is convenient in stating the
theorem, it is more convinient to use statistical physics notation and
"free energies" for the proof, that relies heavily on concepts from
mathematical physics. Let us first translate one into the other.  The
mutual information between the observation ${Y}$ and the unknown $\bv$
is defined using the entropy as $I(Y;\bv) = H(Y)-H(Y|\bv)$. Using
Bayes theorem one obtains $H(Y)=\E_{Y} \{{\log E_{P_v} P_Y(Y|\bv) \}}$
and a straightforward computation shows that the mutual information
per variable is then expressed as
\begin{equation}\frac{I(Y;\bv )}p = f_p +
\lambda    \frac{\E[\bv^\intercal \bv]}{4 p}\, ,
\label{def:fnrg}
\end{equation}
where, using again statistical physics terms, $f_p=-E_{{Y}} \left[ \log \mathcal{Z}_{p}(Y) \right]/p$ is the so called free energy density and $\mathcal{Z}_{p}(Y)$ the partition function defined by
\begin{equation}
  \mathcal{Z}_{p}(Y) \equiv \int_{\bbR^p} \dd \bv~P_v(\bv) \exp \left(
    \sum_{i < j} \left( -\lambda  \frac{v^2_i
        v^2_j}{2p} + \sqrt{\lambda}\frac{v_i v_j Y_{ij}} {\sqrt{p}}
    \right) \right)\,. \label{partsum}
\end{equation}
Notice that the sum does not includes the diagonal term in
(\ref{partsum}). Different conventions can be used dependning on
whether or not one suppose the diagonal terms to be measured, but
these yields only order $1/p$ differences in the free energies, and
thus does not affect the limit $p \to \infty$.
Correspondingly, we thus define the Hamiltonian:
\begin{align}
-H(\bv) &\equiv\sum_{i < j}  \sqrt{\frac{\lambda}{p}}Y_{ij}v_i v_j -\frac{\lambda}{2p} v_i^2v_j^2 =  \sum_{i < j}  \sqrt{\frac{\lambda}{p}}\xi_{ij}v_iv_j + \frac{\lambda}{p}v_iv_jv_i^\star v_j^\star  -\frac{\lambda}{2p} v_i^2v_j^2.\nonumber
\end{align}
so that the partition function (\ref{partsum}) is associated with the Gibbs-Boltzmann measure $e^{-H}/\mathcal{Z}_{p}(Y)$.

Consider now the term $I\left(\bv;\bv+\bz/{\sqrt{q_v \lambda}}\right)$ that enters the expression to be proven eq.~\eqref{eq:information_theory:limip}.  This is the mutual information for another denoising problem, in which we assume one
observes a noisy version of the vector ${\bv^\star }$, denoted $\tilde{\bf y}$ such that
\begin{equation}
\tilde{\bf y} = \frac 1{\sigma} {\bv^\ast} + \bz,
\label{eq:denoising}
\end{equation}
where $\bz \sim \mN(\bzero_p,\rI_p)$ and $\sigma=1/\sqrt{q_v \lambda}$, where we shall assume that the limit exists. Again, it is easier to work with free energies. We thus write the corresponding posterior distribution as
\be
P(\bv |\tilde{\bf y} )  = \frac{1}{\mathcal{Z}_{0}(\tilde{\bf y},\sigma) } P_v( \bv )
\exp \left( -\frac{\| \bv \|_2^2}{2 \sigma^2} + \frac{\bv^\intercal {\tilde{{\bf y}}}}{\sigma}\right)\, ,
\ee
where $\mathcal{Z}_{0}(\tilde{\bf y} )$ is the normalization factor. For this denoising problem, the averaged free energy per variables reads
\begin{align}
f^0_p(\sigma)  \equiv - \frac{1}{p} \E_{\tilde{\bf y}} [ \log \mathcal{Z}_{0}(\tilde{\bf y},\sigma)]
\label{eq:frng_den},
\end{align}
and a short computation shows that
\begin{align*}
	 I\(\bv;\bv+\frac{1}{\sqrt{q_v \lambda}} \bz\) = f_p^0 \(\frac{1}{\sqrt{\lambda q_v}}  \) + \frac{\rho_v \lambda q_v}{2}
\end{align*}

Putting all the pieces together, this means that we need to prove the following statement on the free energy $f_p$:
the free energy $f_p=-\mathbb{E}_{Y} \left[ \log{\mathcal{Z}_{p}(Y)} \right]/p$ is given, as $p \to \infty$ by
\begin{align}
 \lim\limits_{p \to \infty} f_p = \min \phi_{\rm RS}\left(\frac 1{\sqrt{q_v \lambda}} \right) \text{with}~~\phi_{\rm RS}\left(r\right) \equiv
 \lim\limits_{p \to \infty} f_p^0 \left(r\right) + \frac {\lambda q_v^2}{4}  \, .
\end{align}
This statement is  equivalent to theorem \ref{theorem_uu}, and we
shall present a proof for the case where the prior over $\bv$ has a
"good" limit: {\bf we shall assume that the limiting free energy
  exists and concentrates over the disorder, and that the distribution
  over each $v_i$ is bounded}. These hypothesis will be explicitly
given when needed.

Finally, it will be useful to consider Gibbs averages, and to work with $r$ copies of the same system. For any $g : (\mathbb{R}^{p})^{r+1} \mapsto \mathbb{R}$, we  define the Gibbs average as
 \begin{align}\label{gibbs_average_one}
&\left\langle g(\bv^{(1)},\cdots,\bv^{(r)},\bv^\star )\right\rangle \equiv\frac{\int g(\bv^{(1)},\cdots,\bv^{(r)},\bv^\star ) \prod_{l=1}^r e^{- H(\bv^{(l)})} \rmd P_v(\bv^{(l)})}{\left(\int e^{- H(\bv^{(l)})}  \rmd P_v(\bv^{(l)})\right)^r}. %\nonumber
 \end{align}
This is the average of $g$ with respect to the posterior distribution of $r$ copies $\bv^{(1)},\cdots,\bv^{(r)}$ of $\bv^\star $. The variables $\{\bv^{l}\}_{l=1...r}$ are called \emph{replicas}, and are interpreted as random variables independently drawn from the posterior. When $r=1$ we simply write $g(\bv,\bv^\star )$ instead of $g(\bv^{(1)},\bv^\star )$. Finally we shall denote the overlaps between two replicas as follows: for $l,l'=1...r$, we let
\begin{equation}
  R_{l,l'} \equiv\bv^{(l)} \cdot \bv^{(l')} = \frac{1}{p} \sum_{i=1}^p v_i^{(l)}v_i^{(l')}\, .
\end{equation}

A simple but useful consequence of Bayes rule is that the
$(r+1)$-tuples $(\bv^{(1)},\cdots,\bv^{(r+1)})$ and
$(\bv^{(1)},...,\bv^{(r)},\bv^{*})$ have the same under under the
expectation $\E \langle \cdot \rangle$ (see \cite{krzakala_mutual_2016} or proposition $16$ in~\cite{lelarge2019fundamental}). This bears the name of {\em the Nishimori property} in the spin glass literature~\cite{nishimori2001statistical}.

\subsection{Guerra Interpolation for the upper bound}
We start by using the Guerra interpolation to prove an exact formula for the free energy. 

Let $t \in [0,1]$ and let $q_v$ be a non-negative variable. We now consider an interpolating Hamiltonian
\begin{align}\label{interpolating_hamiltonian}
-H_t(\bv) &\equiv\sum_{i  <  j}  \sqrt{\frac{t\lambda}{p}} \xi_{ij}v_iv_j + \frac{t\lambda}{p}v_iv_i^\star v_jv_j^\star  -\frac{t\lambda}{2p}v_i^2v_j^2 \nonumber\\
&~~+\sum_{i=1}^p  \sqrt{(1-t)\lambda q_v} z_{i}v_i +  (1-t)\lambda q_v v_i v_i^\star  -\frac{(1-t)\lambda q_v}{2} v_i^2\,. \nonumber
\end{align}
The Gibbs states associated with this Hamiltonian $-H_t$ correspond to an estimation problem  given an augmented set of observations
\begin{align}%\label{augmented_observations}
\begin{cases}
Y_{ij} &= \sqrt{\frac{t\lambda}{p}} v^\star _iv^\star _j + \xi_{ij}, \quad 1 \le i \le j \le p, \nonumber
\\
\td{y}_i &= \sqrt{(1-t)\lambda q_v}v^\star _i + z_i, \quad 1 \le i \le p.
\nonumber
\end{cases}
\nonumber
\end{align}
Reproducing the argument of~\cite{krzakala_mutual_2016}, we prove
using Guerra's interpolation~\cite{Guerra2003} and the Nishimori
property the following:
\begin{proposition}[Upper bound on the Free energy]: Assume the
  elements of $\bv$ are bounded. Then there exists a constant $K>0$
  such that for all $q_v \in \R$ we have:
\begin{equation}
f_p \le  f_p^0(1/\sqrt{\lambda q_v}) + \frac {\lambda q_v^2}{4} + \frac{K}{p} \, .
\end{equation}
\label{prop_upper}
\end{proposition}
%
%\begin{equation}
 % f_p \le \phi_{\rm RS}(\lambda) + \frac{K}{p}.
 % \end{equation}
  %for some constant $K$.

The proof is a {\it verbatim} reproduction of the argument
of~\cite{krzakala_mutual_2016} for non-factorized prior. We define
\begin{equation}
  \varphi(t) \equiv- \frac{1}{p}\E\log \int e^{-H_t({\bv})}  \rmd P_v({\bv}).
    \end{equation}
 A simple calculation based on Gaussian integration by parts (in
 technical terms, Stein's lemma) applied on the gaussian variebles $\xi$
 and $z$ shows that (see \cite{krzakala_mutual_2016} for details
{\small
\begin{align*}
\varphi'(t) = & \frac{\lambda}{4} \E\left\langle (R_{1,2}-q_v)^2\right\rangle_{t} - \frac{\lambda}{4} q_v^2 - \frac{\lambda}{4p^2} \sum_{i=1}^p \E\left\langle {v_i^{(1)}}^2{v_i^{(2)}}^2\right\rangle_{t} \\
&-\frac{\lambda}{2} \E\left\langle (R_{1,*}- q_v)^2\right\rangle_{t} + \frac{\lambda}{2}q_v^2 + \frac{\lambda}{2p^2} \sum_{i=1}^p \E\left\langle {v_i}^2{v_i^{*}}^2\right\rangle_{t},
\end{align*}
}
We now use the Nishimori property, and the expressions involving the pairs $({\bv},{\bv}^\star )$ and $({\bv}^{(1)},{\bv}^{(2)})$ become equal. We thus obtain
\begin{equation}
  \varphi'(t) = - \frac{\lambda}{4} \E\left\langle (R_{1,*}-q_v)^2\right\rangle_{t} + \frac{\lambda}{4} q_v^2 +
  \frac{\lambda}{4p^2} \sum_{i=1}^p \E\left\langle {v_i}^2{v_i^{*}}^2\right\rangle_{t}.
  \end{equation}
Observe that the last term is $\bigo\left(1/p\right)$ since the variables $v_i$ are bounded. Moreover, the first term is always non-negative so we obtain
\begin{equation}
  \varphi'(t) \le \frac{\lambda}{4} q_v^2 + \frac{K}{p}.
    \end{equation}
Since $\varphi(1) =f_p$ and $\varphi(0) = f_p^0(1/\sqrt{\lambda q_v})$, integrating over $t$, we obtain for all $q_v \ge0$,
$f_p \le \phi_{\rm RS} (\lambda, q_v) + \frac{K}{p},$
and this concludes the proof of the upper bound of proposition.

\subsection{A bound of the Franz-Parisi Potential}
To attack the lower bound, we shall adapt the argument of \cite{AlaouiKrzakala}, that uses the Franz-Parisi potential \cite{franz1995recipes}, and this will require additional concentration properties on the prior model. For $\bv^\star  \in \mathbb{R}^p$ fixed, $m \in \mathbb{R}$ and $\epsilon>0$ we follow \cite{AlaouiKrzakala} and define
\begin{equation}\label{free_energy_fixed_overlap}
\Phi^p_{\epsilon}(m,\bv^\star ) \equiv- \frac{1}{p}\mathbb{E}\log \int_{\bbR^p} \indi\{ R_{1,*} \in [m,m+\epsilon)\} e^{-H(\bv)}  \rmd P_v(\bv)\, .
\end{equation}
This is simply the free energy with configurations forced to be at a distance $m$ (to precision $\epsilon$) from the ground truth. Note that since the measure is limited to a subset of configurations, it is clear that $\E_{\bv^\star }  \Phi^p_{\epsilon}(m,\bv^\star ) \ge  f_p$.

We are now going to prove an interpolating bound for the Franz-Parisi
Potential:
\begin{proposition}[Lower bound on the Franz-Parisi potential]: Assume
  the elements of $\bv$ are bounded. Then there exists $K>0$ such that
  for any $m=q_v$ and $\epsilon>0$ we have
\begin{equation}
\Phi^p_{\epsilon}(m=q_v,\bv^\star ) \ge
f_p^0\left(1/\sqrt{\lambda q_v} ,\bv^\star \right) + \frac {\lambda q_v^2}{4}  - \frac{\lambda}{2} \epsilon^2 + \frac{K}{p}.
\end{equation}
\end{proposition}

The proof proceeds very similarly. Let $t \in [0,1]$ and consider a
slightly different interpolating Hamiltonian
\begin{align}\label{interpolating_hamiltonian}
-H_t(\bv) &\equiv\sum_{i < j}  \sqrt{\frac{t\lambda}{p}} \xi_{ij}v_iv_j + \frac{t\lambda}{p}v_iv_i^\star v_jv_j^\star  -\frac{t\lambda}{2p}v_i^2v_j^2 \nonumber\\
&~~+\sum_{i=1}^p  \sqrt{(1-t)\lambda q_v} z_{i}v_i +  (1-t) \lambda m v_iv_i^\star  -\frac{(1-t)\lambda q_v}{2} v_i^2, \nonumber
\end{align}
Notice the subtle change: in front of the term $(1-t)v_i v_i^\star $ we replace the $q_v$ from the former section by $m$. We define now
\begin{equation}
  \varphi_{\epsilon,m}(t) \equiv- \frac{1}{p}\E\log \int_{\bbR^p} e^{-H_t({\bv})} \indi\{ R_{1,*} \in [m,m+\epsilon)\} \rmd P_v({\bv}).
    \end{equation}
Denoting now the Gibbs average with the additional constraint $\indi\{ R_{1,*} \in [m,m+\epsilon)\}$ as $\langle \rangle_t^{m,\epsilon}$, we find when we repeat the former computation:
{\small
\begin{align*}
 \varphi_{\epsilon,m}'(t) = & \frac{\lambda}{4} \E\left\langle (R_{1,2}-q_v)^2\right\rangle^{m,\epsilon}_{t} - \frac{\lambda}{4} q_v^2
+ \frac{\lambda}{2}m^2 - \frac{\lambda}{2}
\E\left\langle (R_{1,*}- m)^2\right\rangle^{m,\epsilon}_{t} + o\left(1\right)
\end{align*}
}
The trick is now to notice that, by construction, the  $\EE\left\langle (R_{1,*}- m)^2\right\rangle^{m,\epsilon}_{t} \le \epsilon^2$ given the overlap restriction, and therefore
{\small
\begin{align*}
 \varphi_{\epsilon,m}'(t) \ge  & \frac{\lambda}{4} \E\left\langle (R_{1,2}-q_v)^2\right\rangle^{m,\epsilon}_{t} - \frac{\lambda}{4} q_v^2
+\frac{\lambda}{2}m^2  - \frac{\lambda \epsilon^2}{2}  + o\left(1\right) \,,
\end{align*}
and
\begin{align*}
 \varphi_{\epsilon,m}'(t)
\ge   - \frac{\lambda}{4} q_v^2 +\frac{\lambda}{2}m^2 - \frac{\lambda \epsilon^2}{2}  + o\left(1\right)\,. \\
\end{align*}
}
We now denote
\begin{align}
f^0_p(\sigma,\bv^\star )  \equiv - \frac{1}{N} \E_{z} [ \log \mathcal{Z}_{0}(\tilde{\bf y},\sigma)]
\label{eq:frng_den_xstar},
\end{align}
with the previous $f_p^0$ being the expectation $f^0_p(\sigma)  \equiv
\E_{\bv^\star } [ f^0_p(\sigma,\bv^\star ) ]$. Then, since
$\varphi_{\epsilon,m}(1) =\Phi^p_{\epsilon}(m,\bv^\star )$ and
$\varphi_{\epsilon,m}(0) \ge f_p^0(1/\sqrt{\lambda q_v})$ (again, this
is an obvious consequence  of the restriction in the sum) integrating
over $t$, we obtain a bound for the Parisi-Franz potential for any
$q_v$ and $m$. Using, in particular, the value $m=q_v$, this yields
yields the final result.

\subsection{From the Potential to a Lower bound on the free energy}
It remains to connect the Franz-Parisi potential to the actual free
energy. This is done by proving a Laplace-like result between the free
energy and the Franz-Parisi free energy, again following the technics
used in the separable case in \cite{AlaouiKrzakala}:
\begin{proposition}\label{laplace_one}
There exists $K>0$ such that for all $\epsilon>0$, we have
\begin{equation}
f_p \ge \E_{\bv^\star}\Big[\min_{ l \in \Z , |l| \le K/\epsilon  }\Phi^p_{\epsilon}(l\epsilon,\bv^\star )\Big] -  \frac{\log(K/\epsilon)}{\sqrt{p}}.
\end{equation}
\end{proposition}
Combining this proposition with the bound on the Franz-Parisi
potential, we see that
\begin{equation}
f_p \ge \E_{\bv^\star } \hspace{-.1cm}\Big[\min_{q_v=l\epsilon \atop |l|\le
  K/\epsilon} f_p^0\left(1/\sqrt{\lambda q_v} ,\bv^\star \right) + \frac {\lambda
  q_v^2}{4} \Big] -  \frac{\lambda}{2} \epsilon^2 -  \frac{\log(K/\epsilon)}{\sqrt{p}}.
\end{equation}
At this point, we need to push the expectation with respect to the
spike {\it inside} the minimum. This is the {\it only assumption} that we
  are going to require over the generative model: {\it that its free energy
    concentrates over the distribution of spikes}. This
  finally leads to following result:
\begin{proposition}[Laplace principle]\label{laplace_two}
  Assume that the free energy $f_p^0(\bv^\star )$ concentrates such that
\begin{equation}
  \E\left[\left|f_p^0(\frac{1}{\sqrt{\lambda q_v}},\bv^\star
      )-\E\left[f_p^0(\frac{1}{\sqrt{\lambda q_v}},\bv^\star
        )\right]\right|\right]<C/\sqrt{p}
\end{equation}
  for some constant $C$ for all $q_v$ in $[0,\rho_v)$, then:
\begin{equation}
  f_p \ge \min_{q_v}\Big[ f_p^0\left(1/\sqrt{\lambda q_v} \right) + \frac {\lambda
    q_v^2}{4} \Big] +o\left(\frac{\log p}{\sqrt{p}}\right).
\end{equation}
\end{proposition}
which gives us the needed converse bound. To conclude this section, let us prove these propositions.

\begin{proofof}{Proposition~\ref{laplace_one}}
  This is prooven in \cite{AlaouiKrzakala}, and we breifly repeat the
  arguement here. Let $\epsilon>0$. Since the prior $P_v$ has bounded
  support, we can grid the set of the overlap values $R_{1,*}$ by
  $2K/\epsilon$ many intervals of size $\epsilon$ for some $K
  >0$. This allows the following discretisation, where $l$ runs over
  the finite range $\{-K/\epsilon,\cdots,K/\epsilon\}$:
\begin{align}
- f_{p} &= \frac{1}{p} \E\log \sum_{l} \int_{\bbR^p} \indi\{R_{1,*} \in [l\epsilon, (l+1)\epsilon)\} e^{-H(\bv)}  \rmd P_v(\bv)\nonumber\\
&\le \frac{1}{p} \E\log \frac{2K}{\epsilon}\max_{l} \int_{\bbR^p} \indi\{R_{1,*} \in [l\epsilon, (l+1)\epsilon)\} e^{-H(\bv)}  \rmd P_v(\bv)\nonumber\\
&= \frac{1}{p} \E \max_{l} \log\int_{\bbR^p} \indi\{R_{1,*} \in [l\epsilon, (l+1)\epsilon)\} e^{-H(\bv)}  \rmd P_v(\bv) + \frac{\log (2K/\epsilon)}{p}.\label{crude_upperBound}
\end{align}
Note that in the above, the expectation $\E$ is taken with respect to both
the noise matrix ${\xi}$ and the spike ${\bv}^\star $. We shall now
use concentration of measure to push the expectation over ${\xi}$ to the
other side of the maximum in order to recover the Franz-Parisi
potential as defined in the previous section.

 Let
\begin{equation}
 Z_l \equiv\int_{\bbR^p}   \indi\{R_{1,*} \in [l\epsilon, (l+1)\epsilon)\}
 e^{-H(\bv)}  \rmd P_v(\bv).
 \end{equation}
One can show that each term $X_l= \frac{1}{p}\log Z_l$ individually
concentrates around its expectation with respect to the random
variable $\xi$. This follows from the following lemma
\begin{lemma}\label{sub_gaussian_awkward}[from \cite{AlaouiKrzakala}]
There exists a constant $K >0$ such that for all $\gamma \ge 0$ and all $l$,
\begin{equation}
  \E_{\xi} e^{\gamma(X_l-\E_{\xi} [X_l])} \le \frac{K
    \gamma}{\sqrt{p}}e^{K\gamma^2/p}.
  \end{equation}
\end{lemma}
that is a direct consequence of the Tsirelson-Ibragimov-Sudakov
inequality~\cite{boucheron2013concentration}, see
\cite{AlaouiKrzakala}, Lemma 7.

Given that all $X_l$ concentrates, the expectation of the maximum concentrates as well:
\begin{align*}
\E_{\xi}\max_{l}(X_l-\E_{\xi}[X_l])  &\le \frac{1}{\gamma} \log \E_{\xi} \exp\left(\gamma\max_{l}(X_l-\E_{\xi}[X_l])\right)\\
&= \frac{1}{\gamma} \log \E_{\xi} \max_{l}e^{\gamma(X_l-\E'[X_l])}\\
&\le \frac{1}{\gamma} \log \E_{\xi} \sum_{l}e^{\gamma(X_l-\E_{\xi}[X_l])}\\
&\le \frac{1}{\gamma} \log \left(\frac{2K}{\epsilon}\frac{\gamma K}{\sqrt{p}}e^{\gamma^2K/p}\right)\\
&= \frac{\log (2K/\epsilon)}{\gamma} + \frac{1}{\gamma}\log\frac{\gamma K}{\sqrt{p}} + \frac{\gamma K}{p}.
\end{align*}
We set $\gamma = \sqrt{p}$ and obtain
\begin{equation}
\E_{\xi}\max_{l} (X_l-\E_{\xi}[X_l]) \le
\frac{\log(K/\epsilon)}{\sqrt{p}}.
\end{equation}
Therefore, inserting the above estimates into~\eqref{crude_upperBound}, we obtain
\begin{align*}
-f_p &\le \E_{\bv^\star }\max_{l}  \E_{\xi} X_l + \frac{\log
       (K/\epsilon)}{\sqrt{p}} +  \frac{\log (K/\epsilon)}{p} \le  \E_{\bv^\star }\max_l \Phi_{\epsilon}(l\epsilon,\bv^\star )  + 2\frac{\log (K/\epsilon)}{\sqrt{p}}
\end{align*}
so that finally
\begin{align*}
f_p &\ge \E_{\bv^\star }\min_l \Phi_{\epsilon}(l\epsilon,\bv^\star )  - \frac{\log (K/\epsilon)}{\sqrt{p}},
\end{align*}
for some constant $K$.
\end{proofof}

\begin{proofof}{Proposition~\ref{laplace_two}}
  Here we need to pay attention to the fact that the prior is not
  separable, and thus at this point the proof differs from form \cite{AlaouiKrzakala},
We wish to push the expectation with respect to $\bv^\star $ inside the
minimum. We start by using again $q_v=l\epsilon$ and defining the
following random (in $\bv^\star $ variable):
\begin{equation}
\tilde X_l = - \left(f_p^0\left(1/\sqrt{\lambda l \epsilon} ,\bv^\star \right) + \frac {\lambda
  q_v^2}{4}\right)
\end{equation}
and start from Proposition~\ref{laplace_one}:
\begin{equation}
-  f_p \le \E_{\bv^\star } \hspace{-.1cm}\Big[\max_{q_v=l\epsilon \atop |l|\le
  K/\epsilon} \tilde X_l \Big] +  \frac{\lambda}{2} \epsilon^2 +
\frac{\log(K/\epsilon)}{\sqrt{p}}. \label{eql21}
\end{equation}
We now wish to push the max inside. We proceed as follow:
\begin{align}
\E_{\bv^\star } \Big[| \max_l \left(\tilde X_l - \E [\tilde X_l] \right)|
\Big ]  &\le \E_{\bv^\star } \Big[ \sum_l | \left(\tilde X_l - \E [\tilde X_l] \right)|
              \Big ] \\
  &= \sum_l  \E_{\bv^\star } \Big[ | \left(\tilde X_l - \E [\tilde X_l] \right)|
      \Big ] \\
  &\le \sum_l  \frac{C}{\sqrt{p}} = \frac{K}{\epsilon \sqrt{p}}
\end{align}
Inserting this in eq.\eqref{eql21} we find that
\begin{equation}
-  f_p \le \hspace{-.1cm}   \max_{q_v=l\epsilon \atop |l|\le
  K/\epsilon} \Big[ \E_{\bv^\star } \tilde X_l \Big] +  \frac{\lambda}{2}
\epsilon^2 + \frac{K'}{\epsilon\sqrt{p}}
+ \frac{\log(K/\epsilon)}{\sqrt{p}},
\end{equation}
and therefore,
\begin{equation}
  f_p \ge  \min_{q_v=l\epsilon \atop |l|\le
  K/\epsilon} \Big[ -\E_{\bv^\star } \tilde X_l \Big] -  \frac{\lambda}{2}
\epsilon^2 - \frac{K'}{\epsilon\sqrt{p}}
- \frac{\log(K/\epsilon)}{\sqrt{p}},
\end{equation}
so that choosing finally $\epsilon=p^{-1/4}$ we reach
\begin{equation}
  f_p \ge  \min_{q_v=l\epsilon \atop |l|\le
    K/\epsilon} \Big[
f_p^0\left(1/\sqrt{\lambda q_v} \right) + \frac {\lambda
  q_v^2}{4}
  \Big]
+o\left(\frac{\log p}{\sqrt{p}}\right).
\end{equation}
\end{proofof}
\subsection{Main theorem}
We can now combine the upper and lower bound to reach the statement of
the main theorem, presented in the main as theorem \ref{theorem_uu}:
\begin{theorem}\label{theorem_uu_detail}[Mutual information and MMSE for the
  spiked Wigner model with structured spike]
  Assume the spikes $\bv^\star$ come from a sequence (of growing dimension) of
  generic structured prior $P_v$ on $\bbR^p$, such that
  \begin{enumerate}
    \item The elements of $\bv$ are bounded by a constant.
  \item The free energy $f^0_p(\lambda q_v) = - \frac{1}{N} \E_{\tilde{\bf
        y}} [ \log \mathcal{Z}_{0}(\tilde{\bf y},1/\sqrt{\lambda q_v})] $ has a limit
    $f_0(\lambda q_v)$ for all $q_v \in [0,\rho_v]$ as    $p \to \infty$.
   \item The free energy $f_p^0(\bv^\star )$ concentrates such that
  $\E\left[|f_p^0(1/\sqrt{\lambda  q_v},\bv^\star )-\E\left[f_p^0(1/\sqrt{\lambda  q_v},\bv^\star )\right]|\right]<C/\sqrt{p}$ for some constant $C$ for all $q_v \in [0,\rho_v]$ as   $p \to \infty$:
    \end{enumerate}
  then
  \begin{equation} \lim\limits_{p \to \infty} i_p \equiv \lim\limits_{p\to \infty}
    \frac {I(Y;\bv^\star)}p = \inf_{\rho_v \ge q_v \ge 0} {i}_{\rm
      RS}(\Delta,q_v), \end{equation} with \begin{equation} i_{\rm
      RS}(\Delta,q_v) =
    \frac{(\rho_v-q_v)^2}{4\Delta} + \lim\limits_{p \to \infty}
    \frac{I\left(\bv;\bv+\sqrt{\frac{\Delta}{q_v}} \bz \right)}p
    \, \end{equation}
with $z$ being a
  Gaussian vector with zero mean, unit diagonal variance and
  $\rho_v=\lim\limits_{p \to \infty} \E_{P_v}[\bv^\intercal\bv]/p$.
\end{theorem}

\subsection{Mean-squared errors}
It remains to deduce the optimal mean squared errors from the mutual
information. These are actually simple application of known results
which we reproduce here briefly for completeness. It is instructive to
distinguish between the reconstruction of the spike and the
reconstruction of the rank-one matrix.

Let us first focus on the denoising problem, where one aim to
reconstruct the rank-one matrix $X^\star =\bv^\star  {\bv^\star }^\intercal$. In this case the
mean squared error between an estimate $\hat X(Y)$ and the hidden one
$X^\star $ reads
\begin{equation}
 {\rm Matrix-mse}(\hat X,Y)=\frac 1{p^2} \|\bv^\star  {\bv^\star }^\intercal- \hat X(Y) \|_2^2
\end{equation}
It is well-known \cite{cover2012elements} that the mean
squared error is minimized by using the conditional expectation
of the signal given the observation, that is the posterior mean. The
minimal mean square error is thus given by
\begin{equation}
 {\rm Matrix-MMSE}(Y)=\frac 1{p^2} \|\bv^\star  {\bv^\star }^\intercal-\E[\bv \bv^\intercal |Y] \|_F^2
\end{equation}
We can now state the result:
\begin{theorem}\label{matrix-mmse}[Matrix MMSE, from \cite{deshpande2014information,deshpande2016asymptotic,barbier2016mutual}]
  The matrix-MMSE is asymptotically given by
  \begin{equation}
    \lim\limits_{p\to \infty} {\rm Matrix-MMSE}(Y) = \rho_v^2-(q_v^\star )^2
  \end{equation}
  where $q_v^\star$ is the optimizer of the function
  $i_{\rm RS}\(\Delta , q_v\)$.
\end{theorem}
\begin{proof}
  This is a simple application of the I-MMSE theorem
  \cite{GuoShamaiVerdu_IMMSE}, that has been used in this context
  multiple-times (see e.g
  \cite{deshpande2014information,deshpande2016asymptotic,barbier2016mutual}). Indeed, the
  I-MMSE theorem states that, denoting $\lambda=\Delta^{-1}$:
  \begin{equation}
    \frac{d}{d\lambda} \frac Ip =\frac 14 {\rm Matrix-MMSE}(Y) 
  \end{equation}
We thus need to compute the derivative of the mutual information:
  \begin{align}
    \frac{d}{d\lambda} i_{\rm RS}(q_v^\star ,\Delta=1/\lambda) &=
    {\partial}_{\lambda} i_{\rm RS}(q_v^\star ,\Delta=1/\lambda)  +
    {\partial}_{q_v}i_{\rm RS}(q_v,\Delta=1/\lambda)|_{q_v^\star }{\partial}_{\lambda}(q_v^\star )\\
    &={\partial}_{\lambda} i_{\rm RS}(q_v^\star ,\Delta=1/\lambda) 
  \end{align}
  where we used
  $ {\partial}_{q_v} i_{\rm RS}(q,\Delta=1/\lambda)|_{q_v^\star }=0$.  Denoting
  then ${\cal I}(\lambda_v,q_v)=\lim\limits_{p \to \infty}
  \frac{I\left(\bv;\bv+\sqrt{\frac{1}{\lambda q_v}} \bz \right)}p$ we find  
  \begin{align}
    {\partial}_{\lambda} i_{\rm RS}(q_v^\star ,\Delta=1/\lambda) =
    \frac{(\rho_v-q_v)^2}4 + {\partial}_{\lambda} {\cal I}(\lambda_v,q_v) |_{q_v^\star }
  \end{align}
  We now use the fact that the derivate of the replica mutual
  information is zero at $q^\star $. This implies
   \begin{align}
 \frac {\lambda} 2 (\rho_v-q^\star _v) = \partial_{q_v}  {\cal I}(\lambda_v,q_v)|_{q_v^\star }
     = \frac{\lambda} {q_v^\star } {\partial}_{\lambda} {\cal I}(\lambda_v,q_v) |_{q_v^\star }
   \end{align}
   so that
   \begin{align}
     {\partial}_{\lambda} i_{\rm RS}(q_v^\star ,\Delta=1/\lambda)
     =\frac{(\rho_v-q_v)^2}4  +  \frac 1 2 (\rho_v-q^\star _v) q_v^\star  =
     \frac 14 \left(  \rho_v^2 - (q_v^\star )^2  \right)
   \end{align}
   which proves the claim.
 \end{proof}
 
We now consider the problem of reconstruction the spike itself. In this
case the mean square error reads
\begin{align}
  {\rm Vector-mse}(\hat X,Y) &=\frac 1{p} \|\bv- \hat\bv(Y) \|_2^2 \\
   {\rm Vector-MMSE}(Y) &=\frac 1{p} \|\bv- \E[\bv|Y] \|_2^2 \\
\end{align}
Taking the square and averaging, we thus find that the asymptotic
vector MMSE reads
\begin{align}
     {\rm Vector-MMSE}(Y) &=\rho_v + \frac{\|\E[\bv|Y] \|_2^2}p - 2
                            \frac{\E[ \bv^\intercal \bv^\star |Y] }p = \rho_v - \frac{\E[\bv^\intercal \bv^\star|Y] }p
                           	\label{appendix:proof_mmse}
\end{align}
where we have use the Nishimori identity. In order to show that the
MMSE is given by $\rho_v-q^\star _v$, we thus needs to show that $q_v^\star $ is
indeed equal to $ \frac{\E[\bv^\intercal \bv^\star|Y] }p$.

Fortunately, this is easy done by using Theorem $7$ in
\cite{alaoui2017finite}, which apply in our case since it only depends
on the free energy and the Franz-Parisi bound, that we have reproduced
in the coupled cases in the present section. This proposition states
the convergence in probability of the overlaps:
\begin{theorem}[Convergence in probability of the overlap, from
  \cite{alaoui2017finite}]
  Informally, for the Wigner-Spikel model:
  \begin{equation}
    \lim\limits_{p \to \infty} \E \langle \indi(|R_{1,*}| - q_v^\star | \ge \epsilon) \rangle \to 0
  \end{equation}
Note that the absolute value is necessary here, because if the prior
is symmetric, is it impossible to distinguish between $\bv^\star $ and
$-\bv^\star $. If the prior is not symmetric, then the absolute value can
be removed.
  \end{theorem}
%

%% file: files/supplementary/amp_derivation.tex
In this section we present the derivation of the AMP algorithm described in sec.~\ref{main:sec_amp} of the main part. The idea is to simplify the Belief Propagation (BP) equations by expanding them in the large $n, p, k$ limits. Together with a Gaussian ansatz for the distribution of BP messages, this yields a set of $\mO\left(k^2\right)$ simplified equations known as \emph{relaxed BP} (rBP) equations. The last step to get the AMP algorithm is to remove the target dependency of the messages that further reduces the number of iterative equations to $\mO\left(k\right)$. 

Our derivation is closely related to the derivation of AMP for a series of statistical inference problems with factorised priors, see for example \cite{lesieur2017constrained} and references therein. In the interest of the reader, instead of repeating these steps in detail here we describe how two AMP algorithms derived for independent inference problems can be composed into a single AMP for a structured inference problem. In particular, this is illustrated for the case of interest in this manuscript, namely a spiked-matrix estimation with single-layer generative model prior. In this case, the underlying inference problems are the \emph{rank-one matrix  factorization} (MF) \cite{lesieur2017constrained} and the \emph{generalized linear model} (GLM)\cite{Barbier2017c}. We focus the derivation on the more general Wishart model ($\bu\bv^\intercal$) as the result for the Wigner model ($\bv\bv^\intercal$) flows directly from it.

\paragraph{Factor graph:}
In order to compose AMP algorithms, the idea is to replace the separable prior $P_v$ of the variable $\bv$ in the low-rank MF model by a non-separable prior coming from a GLM model with channel $P_{\rm out}$ (see definition in eq.~\eqref{sec:model:glm_2}), while keeping separable distributions $P_u$ and $P_z$ for the variables $\bu\in\mathbb{R}^{n}$ and $\bz\in\mathbb{R}^{k}$.\footnote{Note that differently from the replica calculation in sec.~\ref{sec:appendix:replicafreeen}, to write down the factor graph and derive the associated AMP algorithm we need to fix beforehand the structure of the prior distribution.} Hence to obtain the factor graph of the $\bu\bv^\intercal$ model, we connect the factor graphs of the MF (in green) and GLM (in red) models together by means of $P_{\rm out}$ (in black) (See Fig.~\ref{appendix:factor_graph_vu}).
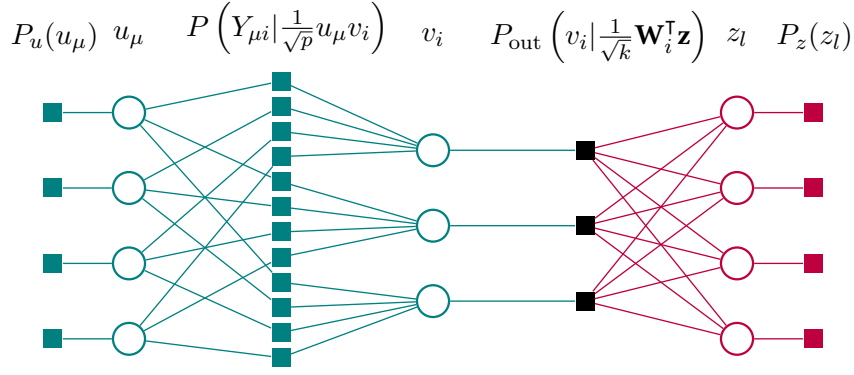
\begin{figure}[!h]
\centering
	\begin{tikzpicture}
	\tikzstyle{factor}=[rectangle,fill=black,minimum size=7pt,inner sep=1pt]
    \tikzstyle{factorMF}=[rectangle,fill=teal,minimum size=7pt,inner sep=1pt]
    \tikzstyle{factorGLM}=[rectangle,fill=purple,minimum size=7pt,inner sep=1pt]
    \tikzstyle{latentMF}=[circle,draw=teal,line width=0.3mm,minimum size=12pt]
    \tikzstyle{latentGLM}=[circle,draw=purple,line width=0.3mm,minimum size=12pt]
    \tikzstyle{latent}=[circle,draw=black,line width=0.3mm,minimum size=12pt]
    \tikzstyle{annot} = [text width=2.5cm, text centered]
    \tikzstyle{edgeMF} = [draw=teal,line width=0.5]
    \tikzstyle{edgeGLM} = [draw=purple,line width=0.5]
    \tikzstyle{edge} = [draw=black,line width=0.5]
	\def\Nv{4}
	\def\Nu{3}
	\def\Nf{12}
		\foreach \name / \y in {1,...,\Nv}
        	\node[latentMF] (V-\name) at (0,-\y) {};
        \foreach \name / \y in {1,...,\Nv}
        	\node[factorMF] (PV-\name) at (-1,-\y) {};
       	\foreach \name / \y in {1,...,\Nf}
        	\node[factorMF] (F-\name) at (2,-0.25-\y/3) {};
        \foreach \name / \y in {1,...,\Nu}
        	\node[latentMF] (U-\name) at (4,-0.5-\y) {};
         \foreach \name / \y in {1,...,\Nu}
        	\node[factor] (F2-\name) at (6,-0.5-\y) {};
        \foreach \name / \y in {1,...,\Nv}
        	\node[latentGLM] (W-\name) at (8,-\y) {};
        \foreach \name / \y in {1,...,\Nv}
        	\node[factorGLM] (PW-\name) at (9,-\y) {};

	\edge [-,edgeMF]{V-1}{F-1};
	\edge [-,edgeMF]{V-2}{F-2};
	\edge [-,edgeMF]{V-3}{F-3};
	\edge [-,edgeMF]{V-4}{F-4};
	\edge [-,edgeMF]{V-1}{F-5};
	\edge [-,edgeMF]{V-2}{F-6};
	\edge [-,edgeMF]{V-3}{F-7};
	\edge [-,edgeMF]{V-4}{F-8};
	\edge [-,edgeMF]{V-1}{F-9};
	\edge [-,edgeMF]{V-2}{F-10};
	\edge [-,edgeMF]{V-3}{F-11};
	\edge [-,edgeMF]{V-4}{F-12};

    \foreach \j in {1,...,4}
        	\edge[-,edgeMF]{F-\j}{U-1} ;
    \foreach \j in {5,...,8}
        	\edge[-,edgeMF]{F-\j}{U-2};
    \foreach \j in {9,...,12}
        	\edge[-,edgeMF]{F-\j}{U-3};
    \foreach \j in {1,...,\Nu}
        	\edge[-,edgeMF]{U-\j}{F2-\j};
    \foreach \i in {1,...,\Nu}
    	 \foreach \j in {1,...,\Nv}
        	\edge[-,edgeGLM]{F2-\i}{W-\j};
    \foreach \i in {1,...,\Nv}
        \edge[-,edgeMF]{PV-\i}{V-\i};

    \foreach \j in {1,...,\Nv}
    	\edge[-,edgeGLM]{PW-\j}{W-\j};

	\node[annot,above of=PV-1, node distance=1cm] {$P_u(u_\mu)$};
	\node[annot,above of=V-1, node distance=1cm] {$u_\mu $};
	\node[annot,above of=F-1, node distance=0.7cm] {$P\(Y_{\mu i }|\frac{1}{\sqrt{p}} u_\mu\ud{v}_{i}\)$};
	\node[annot,above of=U-1, node distance=1.5cm] {$\ud{v_{i}} $};
	\node[annot,above of=F2-1, node distance=1.5cm] {$P_{\rm out}\(\ud{v_{i}}|\frac{1}{\sqrt{k}}\textbf{W}_{i}^\intercal \bz\)$};
	\node[annot,above of=W-1, node distance=1cm] (hl) {$\ud{z}_l$};
	\node[annot,above of=PW-1, node distance=1cm] {$P_z(\ud{z}_l)$};
\end{tikzpicture}
\caption{Factor graph corresponding to a low-rank matrix factorization layer (green) with a prior coming from a GLM (red). We stress that in the classical low-rank layer, red part does not exist and black nodes $P_{\rm out}(v_i | . )$ are replaced by separable prior $P_v(v_i)$.}
\label{appendix:factor_graph_vu}
\end{figure}

%%%%%%%%%%%%%%%%%%%%%%%%%%%%%%%%%%%%
\subsection{Heuristic Derivation}
%%%%%%%%%%%%%%%%%%%%%%%%%%%%%%%%%%%%
We recall the AMP equations for the two modules and we will explain how to plug them together.
%%%%%%%%%%%%%%%%%%%%%%%%%%%%%%%%%%%%%%%%%%%%%%%%%%%%%%%%%%%%%%%%%%%%%%%%
\paragraph{AMP equations for the MF layer (variables $\bv$ and $\bu$):}
%%%%%%%%%%%%%%%%%%%%%%%%%%%%%%%%%%%%%%%%%%%%%%%%%%%%%%%%%%%%%%%%%%%%%%%%
Consider the low-rank matrix factorization model eq.~\eqref{app:Wishart} with separable priors $P_u$ and $P_v$ for the variables $\bu$ and $\bv$. The corresponding non-Bayes-optimal AMP equations, given in \cite{lesieur2017constrained}, read:
\begin{equation}
	\begin{cases}
	\hat{\bu}^{t+1} = \ud{f}_u(\bB_u^t , A_u^t)   \, , \spacecase
	\hat{\bc}_u^{t+1} = \partial_{B}\du{f}_u(\bB_u^t , A_u^t)  \, , \spacecase
	\hat{\bv}^{t+1} = \ud{f}_v(\bB_v^t , A_v^t)   \, , \spacecase
	\hat{\bc}_v^{t+1} = \partial_{B}\du{f}_v(\bB_v^t , A_v^t)  \, , \spacecase
	\end{cases}
\andcase
\begin{cases}
	\bB_v^{t} &= \frac{1}{\sqrt{p}} S^\intercal \hat{\bu}^t  - \frac{1}{p} (S^2)^\intercal  \hat{\bc}_u^t \rI_p \hat{\bv}^{t-1} \, , \spacecase
	A_v^t &= \[\frac{1}{p} (S^2)^\intercal (\hat{\bu}^t)^2 - \frac{1}{p} R^\intercal \(\hat{\bc}_u^t + (\hat{\bu}^t)^2\)\]\rI_p \, , \spacecase
	\bB_u^{t} &= \frac{1}{\sqrt{p}} S \hat{\bv}^t  - \frac{1}{p} S^2  \hat{\bc}_v^t \rI_n \hat{\bu}^{t-1} \, , \spacecase
	A_u^t &= \[ \frac{1}{p} S^2 (\hat{\bv}^t)^2 - \frac{1}{p} R \(\hat{\bc}_v^t + (\hat{\bv}^t)^2\)\] \rI_n \, , \spacecase
\end{cases}
\label{appendix:amp:mf}
\end{equation}
with matrices $S$ and $R$ defined as
\begin{align}
	S_{\mu i} = \frac{Y_{\mu i}}{\Delta} \andcase R_{\mu i } = - \frac{1}{\Delta} + S_{\mu i}^2 \, ,
\end{align}
and the operation $\left(\cdot\right)^2$ is taken component-wise. The update function $f_u$ is the mean of $Q_u$, defined in sec.~\ref{appendix:definition:updates_amp}, and $f_v$ is the mean of the distribution $Q_v(\ud{v}; \ud{B},\du{A}) \equiv \displaystyle \frac{1}{\Z_v (\ud{B},\du{A})} P_v(\ud{v}) e^{ -\frac{1}{2} \du{A}  \ud{v}^2  + \ud{B}\ud{v} }$.
%%%%%%%%%%%%%%%%%%%%%%%%%%%%%%%%%%%%%%%%%%%%%%%%%%%%%%%%%%%%%%
\paragraph{AMP equations for the GLM layer (variable $\bz$):}
%%%%%%%%%%%%%%%%%%%%%%%%%%%%%%%%%%%%%%%%%%%%%%%%%%%%%%%%%%%%%%
On the other hand, the non-Bayes-optimal AMP equations for the GLM model in eq.~\eqref{sec:model:glm}, given in \cite{Barbier2017c}, read
\begin{equation}
\begin{cases}
	\hat{\bz}^{t+1} = f_z({\boldsymbol \gamma}^t , \Lambda^t)   \spacecase
	\hat{\bc}_z^{t+1} = \partial_\gamma f_z({\boldsymbol \gamma}^t , \Lambda^t)\spacecase
	\bg^{t} = f_{\rm out}\(\bv^\star, {\boldsymbol \omega^{t}} ,V^{t} \)
	\end{cases}
	\andcase
	\begin{cases}
	\Lambda^t =  - \frac{1}{k} (W^2)^\intercal \partial_\omega\bg^t \rI_k  \andcase {\boldsymbol \gamma}^t =  \frac{1}{\sqrt{k}} W^\intercal \bg^{t} + \Lambda^t \hat{\bz}^t  \spacecase
	V^t =\frac{1}{k} (W^2)\hat{\bc}_z^{t} \rI_p  \andcase  {\boldsymbol \omega}^t =  \frac{1}{\sqrt{k}} W \hat{\bz}^{t} - V^t \bg^{t-1}
	\end{cases}
	\label{appendix:amp:glm}
\end{equation}
where $f_z$ is the mean of $Q_z$ defined in sec.~\ref{appendix:definition:updates_amp} and $f_{\rm out}$ is the mean of $V^{-1}(x-\omega)$ with respect to $Q_{\rm out}\(x;v^\star,\omega, V \) =\frac{P_{\rm out}\(v^\star | x \)}{\Z_{\rm out}(v^\star, \omega,V)}   e^{ -\frac{1}{2}V^{-1}  \(x - \omega\)^2  } $.

%%%%%%%%%%%%%%%%%%%%%%%%%%%%%
\paragraph{Plug and play:}
%%%%%%%%%%%%%%%%%%%%%%%%%%%%%
In principle composing the AMP equations for the inference problems above is complicated, and require analyzing the BP equations on the composed factor graph in Fig.~\ref{appendix:factor_graph_vu}. However, the upshot of this cumbersome computation is rather simple: the AMP equations for the composed model are equivalent to coupling the MF eqs.~\eqref{appendix:amp:mf} and the GLM eqs.~\eqref{appendix:amp:glm} by replacing $Q_v(v;B,A)$ and $Q_{\rm out}(x;\omega, V)$ with the following joint distribution:
\begin{align}
	Q_{\rm out}(v,x;B,A,\omega,\ud{V}) &\equiv  \displaystyle\frac{1}{\Z_{\rm out}(B,A,\omega,\ud{V})} e^{ -\frac{1}{2} A v^2 + B v } P_{\rm out}\(v | x \)  e^{ -\frac{1}{2}V^{-1}  \(x - \omega\)^2  } \,.
	\label{appendix:Qout_joint}
\end{align}
The associated update functions $f_v$, $f_{\rm out}$ are thus replaced by the mean of $v$ and $V^{-1}(x-\omega)$ with respect to this new joint distribution $Q_{\rm out}$. Replacing this distribution in both AMP algorithms eq.~\eqref{appendix:amp:mf}-\eqref{appendix:amp:glm}, we obtain the AMP algorithm of the structured model, summarized in the next section.

%%%%%%%%%%%%%%%%%%%%%%%%%%%%%%%%%%%%%%%%%%%%%%%%%%%%%%%%%%%%%%%%%%%%%%%
\subsection{Summary of the AMP algorithms - $\bv\bv^\intercal$ and $\bu\bv^\intercal$}
\label{appendix:amp:composed}
%%%%%%%%%%%%%%%%%%%%%%%%%%%%%%%%%%%%%%%%%%%%%%%%%%%%%%%%%%%%%%%%%%%%%%%
Replacing the separable distributions $Q_u$ and $Q_{\rm out}$ by the joint distribution eq.~\eqref{appendix:Qout_joint} and corresponding update functions as described above, we obtain the following AMP algorithm for the Wishart model:
%%%%%%%%%%%%%%%%%%%%%%%%%%%%%%%%%%%%%%%%%%%%%%%%%%%
\paragraph{Wishart model ($\bu\bv^\intercal$):}
%%%%%%%%%%%%%%%%%%%%%%%%%%%%%%%%%%%%%%%%%%%%%%%%%%%
\begin{algorithmic}
    \STATE {\bfseries Input:} vector $Y \in \bbR^{n \times  p}$ and matrix $W\in \bbR^{p \times k}$:
    \STATE \emph{Initialize to zero:} $( \bg, \hat{\bu}, \hat{\bv}, \bB_{u} , A_{u}, \bB_{v} , A_{v})^{t=0}$
	\STATE \emph{Initialize with:} $\hat{\bu}^{t=1}=\mN(0,\sigma^2)$, $\hat{\bv}^{t=1}=\mN(0,\sigma^2)$, $\hat{\bz}^{t=1}=\mN(0,\sigma^2)$,\\
	\hspace{2.2cm}$\hat{\bc}^{t=1}_u = \id_n$, $\hat{\bc}^{t=1}_v = \id_p$, $\hat{\bc}^{t=1}_z=\id_k$. $t=1$
    \REPEAT
    \STATE \emph{Spiked layer:}
    \STATE  $\bB_u^{t} = \frac{1}{\sqrt{p}} S \hat{\bv}^t  - \frac{1}{p} S^2  \hat{\bc}_v^t \rI_n \hat{\bu}^{t-1}$ \andcase $A_u^t = \[ \frac{1}{p} S^2 (\hat{\bv}^t)^2 - \frac{1}{p} R \(\hat{\bc}_v^t + (\hat{\bv}^t)^2\)\] \rI_n$ \spacecase
    \STATE  $\bB_v^{t} = \frac{1}{\sqrt{p}} S^\intercal \hat{\bu}^t  - \frac{1}{p} (S^2)^\intercal  \hat{\bc}_u^t \rI_p \hat{\bv}^{t-1}$ \andcase $A_v^t = \[\frac{1}{p} (S^2)^\intercal (\hat{\bu}^t)^2 - \frac{1}{p} R^\intercal \(\hat{\bc}_u^t + (\hat{\bu}^t)^2\)\]\rI_p$ \spacecase
    \STATE \emph{Generative layer:}
    \STATE $V^t =\frac{1}{k} (W^2)\hat{\bc}_z^{t} \rI_p  \andcase  {\boldsymbol \omega}^t =  \frac{1}{\sqrt{k}} W \hat{\bz}^{t} - V^t \bg^{t-1}$  \andcase ${\bg}^{t} = f_{\rm out}\({\bB}^{t}_v , A^{t}_v , {\boldsymbol \omega}^{t} , V^{t} \) $  \spacecase
    \STATE $\Lambda^t =  - \frac{1}{k} (W^2)^\intercal \partial_\omega\bg^t \rI_k  \andcase {\boldsymbol \gamma}^t =  \frac{1}{\sqrt{k}} W^\intercal \bg^{t} + \Lambda^t \hat{\bz}^t $\spacecase
    \STATE \emph{Update of the estimated marginals:}
    \STATE $\hat{\bu}^{t+1} = \ud{f}_u(\bB_u^t , A_u^t) \hhspace \andcase  \hhspace \hat{\bc}_u^{t+1} = \partial_{B}\du{f}_u(\bB_u^t , A_u^t) $\spacecase
    \STATE $\hat{\bv}^{t+1} = \ud{f}_v(\bB_v^t , A_v^t,{\boldsymbol \omega}^{t} , V^{t} ) \hhspace  \andcase \hhspace  \hat{\bc}_v^{t+1} = \partial_{B}\du{f}_v(\bB_v^t , A_v^t,{\boldsymbol \omega}^{t} , V^{t}) $\spacecase
    \STATE $\hat{\bz}^{t+1} = f_z({\boldsymbol \gamma}^t , \Lambda^t) \andcase \hat{\bc}_z^{t+1} = \partial_\gamma f_z({\boldsymbol \gamma}^t , \Lambda^t)$ \spacecase
    \STATE ${t} = {t} + 1$
    \UNTIL{Convergence}
    \STATE {\bfseries Output: $\hat{\bu}, \hat{\bv}, \hat{\bz}$}
\end{algorithmic}

%%%%%%%%%%%%%%%%%%%%%%%%%%%%%%%%%%%%%%%%%%%%%%%%%
\paragraph{Wigner model ($\bv\bv^\intercal$):} 
%%%%%%%%%%%%%%%%%%%%%%%%%%%%%%%%%%%%%%%%%%%%%%%%%
The AMP algorithm for the Wigner model can be easily obtained from the one above by simply taking taking $\bu^t=\bv^t$, and removing redundant equations:
\begin{algorithmic}
    \STATE {\bfseries Input:} vector $Y \in \bbR^{p \times  p}$ and matrix $W\in \bbR^{p \times k}$:
    \STATE \emph{Initialize to zero:} $( \bg, \hat{\bv}, \bB_{v} , A_{v})^{t=0}$
	\STATE \emph{Initialize with:} $\hat{\bv}^{t=1}=\mN(0,\sigma^2)$, $\hat{\bz}^{t=1}=\mN(0,\sigma^2)$,\\
	\hspace{2.2cm}$\hat{\bc}^{t=1}_v = \id_p$, $\hat{\bc}^{t=1}_z=\id_k$. $t=1$
    \REPEAT
    \STATE \emph{Spiked layer:}
    \STATE  $\bB_v^{t} = \frac{1}{\sqrt{p}} S \hat{\bv}^t  - \frac{1}{p} S^2  \hat{\bc}_v^t \rI_p \hat{\bv}^{t-1}$ \andcase $A_v^t = \[\frac{1}{p} S^2 (\hat{\bv}^t)^2 - \frac{1}{p} R \(\hat{\bc}_v^t + (\hat{\bv}^t)^2\)\]\rI_p$ \spacecase
    \STATE \emph{Generative layer:}
    \STATE $V^t =\frac{1}{k} (W^2)\hat{\bc}_z^{t} \rI_p  \andcase  {\boldsymbol \omega}^t =  \frac{1}{\sqrt{k}} W \hat{\bz}^{t} - V^t \bg^{t-1}$  \andcase ${\bg}^{t} = f_{\rm out}\({\bB}^{t}_v , A^{t}_v , {\boldsymbol \omega}^{t} , V^{t} \) $  \spacecase
    \STATE $\Lambda^t =  - \frac{1}{k} (W^2)^\intercal \partial_\omega\bg^t \rI_k  \andcase {\boldsymbol \gamma}^t =  \frac{1}{\sqrt{k}} W^\intercal \bg^{t} + \Lambda^t \hat{\bz}^t $\spacecase
    \STATE \emph{Update of the estimated marginals:}
    \STATE $\hat{\bv}^{t+1} = \ud{f}_v(\bB_v^t , A_v^t,{\boldsymbol \omega}^{t} , V^{t} ) \hhspace  \andcase \hhspace  \hat{\bc}_v^{t+1} = \partial_{B}\du{f}_v(\bB_v^t , A_v^t,{\boldsymbol \omega}^{t} , V^{t}) $\spacecase
    \STATE $\hat{\bz}^{t+1} = f_z({\boldsymbol \gamma}^t , \Lambda^t) \andcase \hat{\bc}_z^{t+1} = \partial_\gamma f_z({\boldsymbol \gamma}^t , \Lambda^t)$ \spacecase
    \STATE ${t} = {t} + 1$
    \UNTIL{Convergence}
    \STATE {\bfseries Output: $\hat{\bu}, \hat{\bv}, \hat{\bz}$}
\end{algorithmic}

\subsection{Simplified algorithms in the Bayes-optimal setting}
In the Bayes-optimal setting, it can be shown using Nishimori property (see sec.~\ref{appendix:proof_uu} or \cite{lesieur2017constrained}) that:
\begin{align}
	 \langle R \rangle = 0 \iff \langle S^2\rangle = \frac{1}{\Delta}, &&
	\langle \partial_\omega\bg^t \rangle = -\langle (\bg^t)^2 \rangle\,
\end{align}
where $\langle\cdot\rangle$ denotes the average with respect to the posterior distribution in eq.~\eqref{eq:app:defs:posterioruv}. 

Note that the AMP algorithm derived above is also valid for arbitrary weight matrix $W\in\mathbb{R}^{p\times k}$. In the case of interest where $W_{il}\underset{\iid}{\sim}\mathcal{N}(0,1)$,
we can further simplify $\EE\[W_{il}^2\]=1$. Together, these simplifications give:
%%%%%%%%%%%%%%%%%%%%%%%%%%%%%%%%%%%%%%%%%%%%%%%%%%%%%%%%%%%%%%%%%%%%
\paragraph{Wishart model ($\bv\bv^\intercal$) - Bayes-optimal} 
%%%%%%%%%%%%%%%%%%%%%%%%%%%%%%%%%%%%%%%%%%%%%%%%%%%%%%%%%%%%%%%%%%%%
%\begin{algorithm}
\begin{algorithmic}
    \STATE {\bfseries Input:} vector $Y \in \bbR^{n \times  p}$ and matrix $W\in \bbR^{p \times k}$:
    \STATE \emph{Initialize to zero:} $( \bg, \hat{\bu}, \hat{\bv}, \bB_{v} , A_{v}, \bB_{u} , A_{u})^{t=0}$
	\STATE \emph{Initialize with:} $\hat{\bu}^{t=1}=\mN(0,\sigma^2)$, $\hat{\bv}^{t=1}=\mN(0,\sigma^2)$, $\hat{\bz}^{t=1}=\mN(0,\sigma^2)$,\\
	\hspace{2.2cm}$\hat{\bc}^{t=1}_u = \id_n$, $\hat{\bc}^{t=1}_v = \id_p$, $\hat{\bc}^{t=1}_z=\id_k$. $t=1$
    \REPEAT
    \STATE \emph{Spiked layer:}
    \STATE  $\bB_u^{t} = \frac{1}{\Delta} \frac{Y}{\sqrt{p}} \hat{\bv}^t  - \frac{1}{\Delta}  \frac{\id_p^\intercal \hat{\bc}_v^t}{p} \rI_n \hat{\bu}^{t-1}$ \andcase $A_u^t =  \frac{1}{\Delta} \frac{\|\hat{\bv}^t\|_2^2}{p}  \rI_n$ \spacecase
    \STATE  $\bB_v^{t} = \frac{1}{\Delta} \frac{Y^\intercal}{\sqrt{p}}  \hat{\bu}^t  - \frac{1}{\Delta}\frac{\id_n^\intercal \hat{\bc}_u^t}{p}   \rI_p \hat{\bv}^{t-1}$ \andcase $A_v^t =  \frac{1}{\Delta} \frac{\|\hat{\bu}^t\|_2^2}{p}  \rI_p$ \spacecase
    \STATE \emph{Generative layer:}
    \STATE $V^t =\frac{1}{k} \(\id_k^\intercal\hat{\bc}_z^{t} \) \rI_p  \andcase  {\boldsymbol \omega}^t =  \frac{1}{\sqrt{k}} W \hat{\bz}^{t} - V^t \bg^{t-1}$  \andcase ${\bg}^{t} = f_{\rm out}\({\bB}^{t}_v , A^{t}_v , {\boldsymbol \omega}^{t} , V^{t} \) $  \spacecase
    \STATE $\Lambda^t =  \frac{1}{k} \|\bg^t\|_2^2 \rI_k  \andcase {\boldsymbol \gamma}^t =  \frac{1}{\sqrt{k}} W^\intercal \bg^{t} + \Lambda^t \hat{\bz}^t $\spacecase
    \STATE \emph{Update of the estimated marginals:}
    \STATE $\hat{\bu}^{t+1} = \ud{f}_u(\bB_u^t , A_u^t) \hhspace \andcase  \hhspace \hat{\bc}_u^{t+1} = \partial_{B}\du{f}_u(\bB_u^t , A_u^t) $\spacecase
    \STATE $\hat{\bv}^{t+1} = \ud{f}_v(\bB_v^t , A_v^t,{\boldsymbol \omega}^{t} , V^{t} ) \hhspace  \andcase \hhspace  \hat{\bc}_v^{t+1} = \partial_{B}\du{f}_v(\bB_v^t , A_v^t,{\boldsymbol \omega}^{t} , V^{t}) $\spacecase
    \STATE $\hat{\bz}^{t+1} = f_z({\boldsymbol \gamma}^t , \Lambda^t) \andcase \hat{\bc}_z^{t+1} = \partial_\gamma f_z({\boldsymbol \gamma}^t , \Lambda^t)$ \spacecase
    \STATE ${t} = {t} + 1$
    \UNTIL{Convergence}
    \STATE {\bfseries Output: $\hat{\bu}, \hat{\bv}, \hat{\bz}$}
\end{algorithmic}
%\caption{AMP algorithm for ${\bu}{\bv}^\intercal$ in the Bayes-optimal case.}
\label{appendix:AMP_uv_bayes}
%\end{algorithm}

%%%%%%%%%%%%%%%%%%%%%%%%%%%%%%%%%%%%%%%%%%%%%%%%%%%%%%%%%%%%%%%%%%%%%%%%
\paragraph{Wigner model ($\bv\bv^\intercal$) - Bayes-optimal} 
%%%%%%%%%%%%%%%%%%%%%%%%%%%%%%%%%%%%%%%%%%%%%%%%%%%%%%%%%%%%%%%%%%%%%%%%
%\begin{algorithm}[H]
\begin{algorithmic}
    \STATE {\bfseries Input:} vector $Y \in \bbR^{p \times  p}$ and matrix $W\in \bbR^{p \times k}$:
    \STATE \emph{Initialize to zero:} $( \bg, \hat{\bv}, \bB_{v} , A_{v})^{t=0}$
	\STATE \emph{Initialize with:} $\hat{\bv}^{t=1}=\mN(0,\sigma^2)$, $\hat{\bz}^{t=1}=\mN(0,\sigma^2)$,\\
	\hspace{2.2cm}$\hat{\bc}^{t=1}_v = \id_p$, $\hat{\bc}^{t=1}_z=\id_k$. $t=1$
    \REPEAT
    \STATE \emph{Spiked layer:}
    \STATE  $\bB_v^{t} = \frac{1}{\Delta} \frac{Y}{\sqrt{p}} \hat{\bv}^t  - \frac{1}{\Delta}    \frac{\id_p^\intercal \hat{\bc}_v^t}{p} \hat{\bv}^{t-1}$ \andcase $A_v^t = \frac{1}{\Delta} \frac{\|\hat {\bv}^t \|_2^2}{p} \rI_p$ \spacecase
    \STATE \emph{Generative layer:}
    \STATE $V^t =\frac{1}{k} \(\id_k^\intercal\hat{\bc}_z^{t}\) \rI_p  \andcase  {\boldsymbol \omega}^t =  \frac{1}{\sqrt{k}} W \hat{\bz}^{t} - V^t \bg^{t-1}$  \andcase ${\bg}^{t} = f_{\rm out}\({\bB}^{t}_v , A^{t}_v , {\boldsymbol \omega}^{t} , V^{t} \) $  \spacecase
    \STATE $\Lambda^t =  \frac{1}{k} \|\hat {\bg}^t \|_2^2 \rI_k  \andcase {\boldsymbol \gamma}^t =  \frac{1}{\sqrt{k}} W^\intercal \bg^{t} + \Lambda^t \hat{\bz}^t $\spacecase
    \STATE \emph{Update of the estimated marginals:}
    \STATE $\hat{\bv}^{t+1} = \ud{f}_v(\bB_v^t , A_v^t,{\boldsymbol \omega}^{t} , V^{t} ) \hhspace  \andcase \hhspace  \hat{\bc}_v^{t+1} = \partial_{B}\du{f}_v(\bB_v^t , A_v^t,{\boldsymbol \omega}^{t} , V^{t}) $\spacecase
    \STATE $\hat{\bz}^{t+1} = f_z({\boldsymbol \gamma}^t , \Lambda^t) \andcase \hat{\bc}_z^{t+1} = \partial_\gamma f_z({\boldsymbol \gamma}^t , \Lambda^t)$ \spacecase
    \STATE ${t} = {t} + 1$
    \UNTIL{Convergence}
    \STATE {\bfseries Output: $\hat{\bv}, \hat{\bz}$}
\end{algorithmic}
\label{appendix:AMP_vv_bayes}

\subsection{Derivation of the state evolution equations}
\label{appendix:derivation_SE_from_AMP}
The AMP algorithms above are valid for any large but finite sizes $k,n,p$. A central object of interest are the \emph{state evolution equations} (SE) that predict the algorithm's behaviour in the infinite size limit $k\to \infty$. We show in this section the derivation of these equations, directly from the algorithm to explicitly show that it provides the same set of equations as the saddle point equations obtained from the replica free entropy eq.~\eqref{eq:app:replicas:SE}. As before, we focus on the derivation of the more general Wishart model $\bu\bv^{\intercal}$, and quote the result for the symmetric $\bv\bv^{\intercal}$. We first derive the SE equations without loss of generality in the non-Bayes-optimal case, and we will state them in their simplified formulation in the Bayes-optimal case.

The idea is to compute the average distributions of the messages involved in the AMP algorithm updates in sec.~\eqref{appendix:amp:composed}, namely $\bB_u,A_u,\bB_v,A_v, \bomega,V, \bgamma$ and $V$. The usual derivation starts with rBP equations that we did not present here, see \cite{lesieur2017constrained}. However this equations are roughly equivalent to AMP messages if we remove the Onsager terms containing messages with delayed time indices $(\cdot)^{t-1}$. 

\paragraph{Definition of the overlap parameters:}
We first define the order parameters, called overlaps in the physics literature, that will measure the correlation of the Bayesian estimator with the ground truth signals
\begin{align}
\label{appendix:amp:overlap_definition}
	m_u^t &\equiv \EE_{\bu^\star} \lim_{n\to \infty} \frac{(\hat{\bu}^t)^\intercal \bu^\star}{n}\,,\hhspace q_u^t \equiv \EE_{\bu^\star} \lim_{n\to \infty} \frac{(\hat{\bu}^t)^\intercal \hat{\bu}^t}{n}\,, \hhspace \Sigma_u^t \equiv \EE_{\bu^\star}\lim_{n\to \infty} \frac{\id_n^\intercal  \hat{\bc}^{u,t}}{n} \nonumber\,, \spacecase
	m_v^t &\equiv \EE_{\bv^\star} \lim_{p\to \infty} \frac{(\hat{\bv}^t)^\intercal \bv^\star}{p}\,,\hhspace q_v^t \equiv \EE_{\bv^\star} \lim_{p\to \infty} \frac{(\hat{\bv}^t)^\intercal \hat{\bv}^t}{p}\,, \hhspace \Sigma_v^t \equiv \EE_{\bv^\star} \lim_{p\to \infty} \frac{\id_p^\intercal \hat{\bc}^{v,t}}{p}\,,  \spacecase 
	m_z^t &\equiv \EE_{\bz^\star} \lim_{k\to \infty} \frac{(\hat{\bz}^t)^\intercal \bz^\star}{k}\,,\hhspace q_z^t \equiv  \EE_{\bz^\star}\lim_{k\to \infty} \frac{(\hat{\bz}^t)^\intercal \hat{\bz}^t}{k}\,, \hhspace \Sigma_z^t \equiv\EE_{\bz^\star} \lim_{k\to \infty} \frac{\id_k^\intercal \hat{\bc}^{z,t}}{k}\,. \nonumber
\end{align}
As the algorithm performance, such as the mean squared error or the generalization error, can be computed directly from these overlap parameters, our goal is to derive the their average distribution in the infinite size limit.

\paragraph{Messages distributions:}
As stressed above, we compute the average distribution of the messages, taking the average over variables $W$, $\xi$, the planted solutions $\bv^\star, \bu^\star, \bz^\star$ and taking the limit $k \to \infty$. Note that we use the BP independence assumption over the messages and keep only dominant terms in the $1/p$ expansion.

\paragraph{$\bullet$ $B_u,A_u$:}
Starting with the AMP update equations for the $\bu\bv^{\intercal}$ model in sec.~\eqref{appendix:amp:composed}, we obtain
\begin{align}
	\EE\[ \bB_u^t \] &= \frac{1}{\sqrt{p}\Delta} \EE\[ Y \hat{\bv}^t \] = \frac{1}{\sqrt{p}\Delta} \EE\[ \(\frac{\bu^\star(\bv^\star)^\intercal}{\sqrt{p}}  + \sqrt{\Delta}\xi \) \hat{\bv}^t \] \underlim{p}{\infty}  \frac{m_v^t}{\Delta} \bu^\star \,, \spacecase
	\EE\[ \bB_u^t (\bB_u^t)^\intercal \] &= \frac{1}{p\Delta^2} \EE\[ Y \hat{\bv}^t(\hat{\bv}^t)^\intercal Y^\intercal \]  = \frac{1}{\Delta}  \frac{1}{p} \EE\[ \xi \hat{\bv}^t(\hat{\bv}^t)^\intercal \xi^\intercal \]  + o\(1/p\) \underlim{p}{\infty} \frac{q_v^t}{\Delta} \rI_n \,,\spacecase
	\EE \[ A_u^t \] &= \EE\[ \frac{1}{p} S^2 (\hat{\bv}^t)^2 - \frac{1}{p} R \(\hat{\bc}_v^t + (\hat{\bv}^t)^2\)\] \rI_n \underlim{p}{\infty} \frac{q_v^t}{\Delta} \rI_n  - \bar{R}\Sigma_v^t \rI_n \,.
\end{align}
where we defined, see \cite{lesieur2017constrained},
\begin{align*}
	\bar{R}&= \EE_{P\(Y|\omega\)}\[ \partial^2_{\omega}g+\left(\partial_{\omega}g\right)^2\right]_{\omega=0} = \int\prod\limits_{1\leq i \leq p, 1\leq \mu \leq n}\dd Y_{\mu i }~e^{g^\star(Y_{\mu i},0)}\left[\partial^2_{\omega}g+\left(\partial_{\omega}g\right)^2\right]_{Y,\omega=0}\,
\end{align*}
with $P\(Y|\omega\)$, $g$ defined in eq.~\eqref{appendix:replicas:P_Y} and $g^\star$ the ground truth channel function. Note that in the Bayes-optimal case, $g^\star = g$ that yields $\bar{R} = 0$ as mentioned in eq.~\eqref{appendix:replica:bayes_relations}.

\paragraph{$\bullet$ $B_v,A_v$:}
Similarly,
\begin{align}
	\EE\[ \bB_v^t \] &= \frac{1}{\sqrt{p}\Delta} \EE\[ Y^\intercal \hat{\bu}^t \] = \frac{1}{\sqrt{p}\Delta} \EE\[ \(\frac{\bu^\star(\bv^\star)^\intercal}{\sqrt{p}}  + \sqrt{\Delta}\xi \)^\intercal \hat{\bu}^t \] \underlim{p}{\infty}  \beta\frac{m_u^t}{\Delta} \bv^\star \,, \spacecase
	\EE\[ \bB_v^t (\bB_v^t)^\intercal \] &= \frac{1}{p\Delta^2} \EE\[ Y^\intercal \hat{\bu}^t(\hat{\bu}^t)^\intercal Y \] \underlim{p}{\infty} \beta \frac{q_u^t}{\Delta} \rI_p \,, \spacecase
	\EE \[ A_v^t \] &= \EE\[ \frac{1}{p} (S^\intercal)^2 (\hat{\bu}^t)^2 - \frac{1}{p} R \(\hat{\bc}_u^t + (\hat{\bu}^t)^2\)\] \rI_p \underlim{p}{\infty} \beta \( \frac{q_u^t}{\Delta} -\bar{R}\Sigma_u^t \)\rI_p \,.
\end{align}
\paragraph{$\bullet$ $\omega,V$:}
\begin{align}
	\EE\[\bomega^t\] &= \EE\[ \frac{1}{\sqrt{k}} W \hat{\bz}^{t} \] = \bzero_p \,, \spacecase
	\EE\[\bomega^t (\bomega^t)^\intercal \] &= \EE\[ \frac{1}{k} W \hat{\bz}^{t} (\hat{\bz}^{t})^\intercal W^\intercal \] \underlim{n}{\infty} q_z^t \rI_p \,, \spacecase
	\EE \[ V \] &=  \EE\[\frac{1}{k} (W^2)\hat{\bc}_z^{t} \rI_p\] \underlim{k}{\infty} \Sigma_z^t \rI_p \,.
\end{align}

\paragraph{Conclusion:}
Finally we conclude that to leading order:
\begin{align}
	\bB_u &\sim  \frac{m_v^t}{\Delta} \bu^\star + \sqrt{\frac{q_v^t}{\Delta}} \bxi_u \,, && A_u^t \sim \frac{q_v^t}{\Delta} \rI_n  - \bar{R}\Sigma_v^t \rI_n \,,\spacecase
	\bB_v &\sim  \beta\frac{m_u^t}{\Delta} \bv^\star + \sqrt{\beta\frac{q_u^t}{\Delta}} \bxi_v \,,&& A_v^t \sim  \beta \( \frac{q_u^t}{\Delta} -\bar{R}\Sigma_u^t \)\rI_p \,, \spacecase
	 \bomega &\sim  \sqrt{q_z^t} {\boldsymbol \eta} \,,&& V \sim \Sigma_z^t \rI_p \,,
	 \label{appendix:amp:se_from_amp:distributions}
\end{align}
with $\bxi_u \sim \mN\( \bzero_n, \rI_n \), \bxi_v\sim \mN\( \bzero_n, \rI_n \), {\boldsymbol \eta} \sim \mN\( \bzero_p, \rI_p  \)$.

\paragraph{State evolution - Non Bayes-optimal case}
With the averaged limiting distributions of all the messages, we can now compute the state evolution of the overlaps. Using the definition of the overlaps eq.~\eqref{appendix:amp:overlap_definition} and distributions in eq.~\eqref{appendix:amp:se_from_amp:distributions}, we obtain:
\paragraph{Variable $\bu$:}
\begin{align}
	q_u^{t+1} &\equiv  \EE_{\bu^\star} \lim_{n\to\infty} \frac{1}{n} (\hat{\bu}^{t+1})^\intercal	\hat{\bu}^{t+1}	= \EE_{\bu^\star} \lim_{n\to\infty} \frac{1}{n} \ud{f}_u(\bB_u^t , A_u^t)^\intercal	\ud{f}_u(\bB_u^t , A_u^t) \label{appendix:amp:se_from_amp:se_nb_qu} \\
	&= \EE_{u^\star,\xi}\[ \ud{f}_u\(\frac{m_v^t}{\Delta} u^\star + \sqrt{\frac{q_v^t}{\Delta}} \xi , \frac{q_v^t}{\Delta} - \bar{R}\Sigma_v^t \)^2 \] \nonumber  \spacecase
	m_u^{t+1} &\equiv \EE_{\bu^\star} \lim_{n\to\infty} \frac{1}{n} (\hat{\bu}^{t+1})^\intercal	\bu^{\star}	= \EE_{\bu^\star}\lim_{n\to\infty} \frac{1}{n} \ud{f}_u(\bB_u^t , A_u^t)^\intercal	 \bu^\star \label{appendix:amp:se_from_amp:se_nb_mu} \\
	&= \EE_{u^\star,\xi}\[ \ud{f}_u\(\frac{m_v^t}{\Delta} u^\star + \sqrt{\frac{q_v^t}{\Delta}} \xi , \frac{q_v^t}{\Delta} - \bar{R}\Sigma_v^t \) u^\star \] \nonumber\spacecase
	\Sigma_u^{t+1} &\equiv \EE_{\bu^\star} \lim_{n\to\infty} \frac{1}{n} 	\id_n^\intercal \hat{\bc}^{u,t+1} = \EE_{\bu^\star} \lim_{n\to\infty} \frac{1}{n} \partial_B \ud{f}_u(\bB_u^t , A_u^t)^\intercal	\id_n \label{appendix:amp:se_from_amp:se_nb_su} \\
	&=\EE_{u^\star,\xi}\[ \partial_B\ud{f}_u\(\frac{m_v^t}{\Delta} u^\star + \sqrt{\frac{q_v^t}{\Delta}} \xi , \frac{q_v^t}{\Delta} - \bar{R}\Sigma_v^t \)^2 \]\nonumber
\end{align}

\paragraph{Variable $\bv$:}
\begin{align}
	q_v^{t+1} &=  \EE_{\bv^\star} \lim_{p\to\infty} \frac{1}{p} (\hat{\bv}^{t+1})^\intercal	\hat{\bv}^{t+1}	= \EE_{\bv^\star} \lim_{p\to\infty} \frac{1}{p} \ud{f}_v(\bB_v^t , A_v^t,{\boldsymbol \omega}^{t} , V^{t} )^\intercal	f_v(\bB_v^t , A_v^t,{\boldsymbol \omega}^{t} , V^{t} )  \label{appendix:amp:se_from_amp:se_nb_qv}\\
	&= \EE_{v^\star,\xi,\eta}\[ \ud{f}_v\(\frac{\beta m_u^t}{\Delta} v^\star  + \sqrt{\frac{ \beta q_u^t}{\Delta}} \xi , \beta \( \frac{q_u^t}{\Delta} -\bar{R}\Sigma_u^t \), \sqrt{q_z^t}\eta , \Sigma_z^t\)^2 \] \nonumber \spacecase
	m_v^{t+1} &=  \EE_{\bv^\star} \lim_{p\to\infty} \frac{1}{p} (\hat{\bv}^{t+1})^\intercal	\hat{\bv}^{t+1}	= \EE_{\bv^\star} \lim_{p\to\infty} \frac{1}{p} \ud{f}_v(\bB_v^t , A_v^t,{\boldsymbol \omega}^{t} , V^{t} )^\intercal	\bv^\star  \label{appendix:amp:se_from_amp:se_nb_mv}\\
	&= \EE_{v^\star,\xi,\eta}\[ \ud{f}_v\(\frac{\beta m_u^t}{\Delta} v^\star  + \sqrt{\frac{ \beta q_u^t}{\Delta}} \xi , \beta \( \frac{q_u^t}{\Delta} -\bar{R}\Sigma_u^t \), \sqrt{q_z^t}\eta , \Sigma_z^t \) v^\star \] \nonumber\spacecase
	\Sigma_v^{t+1} &=  \EE_{\bv^\star} \lim_{p\to\infty} \frac{1}{p} \id_p^\intercal \hat{\bc}^{z,t+1}	= \EE_{\bv^\star} \lim_{p\to\infty} \frac{1}{p} \partial_\gamma f_v(\bB_v^t , A_v^t,{\boldsymbol \omega}^{t} , V^{t} )^\intercal	\id_p  \label{appendix:amp:se_from_amp:se_nb_sv}\\
	&= \EE_{v^\star,\xi,\eta}\[ \partial_\gamma f_v\(\frac{\beta m_u^t}{\Delta} v^\star  + \sqrt{\frac{ \beta q_u^t}{\Delta}} \xi , \beta \( \frac{q_u^t}{\Delta} -\bar{R}\Sigma_u^t \), \sqrt{q_z^t}\eta , \Sigma_z^t \)^2 \] \nonumber
\end{align}

\paragraph{Variable $\hat{\bz}$:}
We define intermediate \emph{hat} overlap parameters\footnote{These variables appear as well in the replica computation through Dirac delta Fourier representation.} that will be useful in the following. The hat overlaps don't have as much physical meaning as the standard overlaps that quantify the reconstruction performances. Though we might notice anyway that all the overlap parameters are built similarly as function of the update functions $f_u, f_v, f_z$ and $f_{\rm out}$ (see eq.~\eqref{eq:app:defs:fvfout}
\begin{align}
	\hat{q}_z^t &\equiv \alpha \EE_{v^\star,\xi,\eta}\[ f_{\rm out}\(\frac{\beta m_u^t}{\Delta} v^\star  + \sqrt{\frac{ \beta q_u^t}{\Delta}} \xi , \beta \( \frac{q_u^t}{\Delta} -\bar{R}\Sigma_u^t \), \sqrt{q_z^t}\eta , \Sigma_z^t \)^2 \] \label{appendix:amp:se_from_amp:se_nb_qhat} \spacecase
	\hat{m}_z^t & \equiv \alpha \EE_{v^\star,\xi,\eta}\[ \partial_x f_{\rm out}\(\frac{\beta m_u^t}{\Delta} v^\star  + \sqrt{\frac{ \beta q_u^t}{\Delta}} \xi , \beta \( \frac{q_u^t}{\Delta} -\bar{R}\Sigma_u^t \), \sqrt{q_z^t}\eta , \Sigma_z^t \) v^\star \] \label{appendix:amp:se_from_amp:se_nb_mhat} \spacecase
	\hat{\Sigma}_z^t & \equiv \alpha \EE_{v^\star,\xi,\eta}\[ - \partial_\omega f_{\rm out}\(\frac{\beta m_u^t}{\Delta} v^\star  + \sqrt{\frac{ \beta q_u^t}{\Delta}} \xi , \beta \( \frac{q_u^t}{\Delta} -\bar{R}\Sigma_u^t \), \sqrt{q_z^t}\eta , \Sigma_z^t \)  \]\label{appendix:amp:se_from_amp:se_nb_shat}
\end{align}

\paragraph{Variable $\bz$:}
Averages are explicitly expressed as a function of the \emph{hat} overlaps introduced just above: 
\begin{align}
	\EE\[\bgamma^t\] & \sim \hat{m}_z^t  \bz^\star \spacecase
	\EE\[\bgamma^t (\bgamma^t)^\intercal \] &\sim   \hat{q}_z^t \rI_k \spacecase
	\EE\[ \Lambda^t \] &\sim  \hat{\Sigma}_z^t \rI_k
\end{align}
And we conclude that at the leading order:
\begin{align}
	\bgamma^t &\sim \hat{m}_z^t  \bz^\star + \sqrt{\hat{q}_z^t} \bxi \,, && \Lambda^t \sim \hat{\Sigma}_z^t \rI_k\,.
\end{align}
with $\bxi \sim \mN\( \bzero_k, \rI_k \)$.

From these later equations, we obtain
\begin{align}
	q_z^{t+1} &=  \EE_{\bz^\star} \lim_{k\to\infty} \frac{1}{k} (\hat{\bz}^{t+1})^\intercal	\hat{\bz}^{t+1}	= \EE_{\bz^\star} \lim_{k\to\infty} \frac{1}{k} \ud{f}_z(\bgamma^t, \Lambda^t)^\intercal\ud{f}_z(\bgamma^t, \Lambda^t) \label{appendix:amp:se_from_amp:se_nb_qz}\\
	&= \EE_{z^\star,\xi}\[ \ud{f}_z\( \hat{m}_z^t  z^\star + \sqrt{\hat{q}_z^t} \xi , \hat{\Sigma}_z^t \)^2 \] \nonumber\spacecase
	m_z^{t+1} &=  \EE_{\bz^\star} \lim_{k\to\infty} \frac{1}{k} (\hat{\bz}^{t+1})^\intercal	\bz^\star	= \EE_{\bz^\star} \lim_{k\to\infty} \frac{1}{k} \ud{f}_z(\bgamma^t, \Lambda^t)^\intercal\bz^\star
	\label{appendix:amp:se_from_amp:se_nb_mz}\\
	&= \EE_{z^\star,\xi}\[ \ud{f}_z\( \hat{m}_z^t  z^\star + \sqrt{\hat{q}_z^t} \xi , \hat{\Sigma}_z^t \) z^\star \] \nonumber\spacecase
	\Sigma_z^{t+1} &=  \EE_{\bz^\star} \lim_{k\to\infty} \frac{1}{k} \id_k^\intercal \hat{\bc}^{z,t+1}		= \EE_{\bz^\star} \lim_{k\to\infty} \frac{1}{k} \id_k^\intercal \partial_\gamma\ud{f}_z(\bgamma^t, \Lambda^t) \label{appendix:amp:se_from_amp:se_nb_sz} \\
	&= \EE_{z^\star,\xi}\[ \partial_\gamma \ud{f}_z\( \hat{m}_z^t  z^\star + \sqrt{\hat{q}_z^t} \xi , \hat{\Sigma}_z^t \) \] \nonumber
\end{align}

Equations (\ref{appendix:amp:se_from_amp:se_nb_qu}-
	\ref{appendix:amp:se_from_amp:se_nb_shat}, 
	\ref{appendix:amp:se_from_amp:se_nb_qz}-\ref{appendix:amp:se_from_amp:se_nb_sz}) constitute the closed set of AMP \emph{state evolution equations} in the non-Bayes-optimal case.
	
\paragraph{State evolution - Bayes-optimal case}
In the Bayes-optimal case, the Nishimori property (See sec.~\ref{app:Wigner}) implies $m_u=q_u$, $m_z=q_z$, $m_v=q_v$ and $\hat{m}_z=\hat{q}_z$, $\bar{R}=0$ and we also note that $\Sigma_z^t = \rho_z - q_z^t$, $\hat{\Sigma}_z^t=\hat{q}_z^t$. The set of twelve state evolution equations reduce to only four, and they can be rewritten using a change of variable.
\paragraph{Wishart model}
\begin{align}
	\label{appendix:se_from_amp_bayes_uv}
	q_u^{t+1} &= \displaystyle \EE_{\xi}\[ \mZ_u \(\sqrt{\frac{q_v^t}{\Delta}} \xi , \frac{q_v^t}{\Delta} \) \ud{f}_u\(\sqrt{\frac{q_v^t}{\Delta}} \xi , \frac{q_v^t}{\Delta} \)^2 \]	 \\
	&= \displaystyle  2\partial_{q_{v}}\Psi_{u}\left(q_{v}^t\right) \,, \nonumber \spacecase
	q_z^{t+1} &= \displaystyle \EE_{\xi}\[ \mZ_z\(\sqrt{\hat{q}_z^t} \xi , \hat{q}_z^t \) \ud{f}_z\(\sqrt{\hat{q}_z} \xi , \hat{q}_z^t \)^2\] \\
	&= \displaystyle  2\partial_{\hat{q}_{z}}\Psi_{z}\left(\hat{q}_{z}^t\right) \,, \nonumber\spacecase
	\hat{q}_z^t &= \alpha \EE_{\xi,\eta}\[ \mZ_{\rm out}\( \sqrt{\frac{ \beta q_u^t}{\Delta}} \xi , \beta\frac{q_u^t}{\Delta} , \sqrt{q_z^t}\eta , \rho_z - q_z^t \) f_{\rm out}\( \sqrt{\frac{ \beta q_u^t}{\Delta}} \xi , \beta\frac{q_u^t}{\Delta} , \sqrt{q_z^t}\eta , \rho_z - q_z^t \)^2 \]  \\
	&= \displaystyle 2\alpha\partial_{q_{z}}\Psi_{\out}\left(\frac{\beta q_{u}^t}{\Delta},q_{z}^t\right) \,, \nonumber \spacecase
	q_v^{t+1} &= \EE_{\xi,\eta}\[ \mZ_{\rm out} \(\sqrt{\frac{ \beta q_u^t}{\Delta}} \xi , \beta\frac{q_u^t}{\Delta}, \sqrt{q_z^t}\eta , \rho_z - q_z^t\) \ud{f}_v\(\sqrt{\frac{ \beta q_u^t}{\Delta}} \xi , \beta\frac{q_u^t}{\Delta}, \sqrt{q_z^t}\eta , \rho_z - q_z^t\)^2 \]  \\
	&= \displaystyle 2 \partial_{q_{u}}\Psi_{\out}\left(\frac{\beta q_{u}^t}{\Delta},q_{z}^t\right)\,. \nonumber
\end{align}

\paragraph{Wigner model}
The state evolution for the Wigner model ($\bv\bv^{\intercal}$) is a particular case of the state evolution of the Wishart model discussed above, obtained by simply restricting $q_u = q_v$ and $\beta=1$. It finally reads
\begin{align}
\label{appendix:se_from_amp_bayes_vv}
	q_z^{t+1} &= \displaystyle \EE_{\xi}\[ \mZ_z\(\sqrt{\hat{q}_z^t} \xi , \hat{q}_z^t \) \ud{f}_z\(\sqrt{\hat{q}_z^t} \xi , \hat{q}_z^t \)^2\] \\
	&= \displaystyle  2\partial_{\hat{q}_{z}}\Psi_{z}\left(\hat{q}_{z}^t\right) \,, \nonumber \spacecase
	\hat{q}_z^t &= \displaystyle\alpha \EE_{\xi,\eta}\[ \mZ_{\rm out}\( \sqrt{\frac{q_v^t}{\Delta}} \xi , \frac{q_v^t}{\Delta} , \sqrt{q_z^t}\eta , \rho_z - q_z^t \) f_{\rm out}\( \sqrt{\frac{q_v^t}{\Delta}} \xi , \frac{q_v^t}{\Delta} , \sqrt{q_z^t}\eta , \rho_z - q_z^t \)^2 \]  \\
	&= \displaystyle 2\alpha\partial_{q_{z}}\Psi_{\out}\left(\frac{q_{v}^t}{\Delta},q_{z}^t\right) \,, \nonumber \spacecase
	q_v^{t+1} &= \displaystyle\EE_{\xi,\eta}\[ \mZ_{\rm out} \(\sqrt{\frac{q_v^t}{\Delta}} \xi , \frac{q_v^t}{\Delta}, \sqrt{q_z^t}\eta , \rho_z - q_z^t\) \ud{f}_v\(\sqrt{\frac{q_v^t}{\Delta}} \xi ,\frac{q_v^t}{\Delta}, \sqrt{q_z^t}\eta , \rho_z - q_z^t\)^2 \] \\
	&= \displaystyle 2\partial_{q_{v}}\Psi_{\out}\left(\frac{q_{v}^t}{\Delta},q_{z}^t\right) \,,\nonumber
\end{align}
\noindent which are precisely the state evolution equations derived from the replica trick in sec.~\ref{sec:appendix:replicafreeen}, eq.~\eqref{eq:app:replicas:SE}, except that the algorithm provides the correct time indices in which the iterations should be taken.

%% file: files/supplementary/lamp_derivation.tex
\label{appendix:lamp_derivation}
We present in this section the derivation of the linearized-AMP (LAMP) spectral algorithm. This method, pioneered in \cite{krzakala_spectral_2013}, relies on the existence of the non-informative fixed point of the SE equations eq.~\eqref{main:SE_uu}, $q_v=0$ that translates to $\hat{\bv}=0$ in the AMP equations. Linearizing the Bayes-optimal AMP equations for the Wigner and Wishart models sec.~\ref{appendix:AMP_vv_bayes} around this trivial fixed point will lead to the LAMP spectral method. First, we detail the calculation for the simpler Wigner model, and then generalize the spectral algorithm in the Wishart case. Finally, we derive the state evolution associated to spectral method in the case of linear activation function.
\subsection{Wigner model: $\bv\bv^\intercal$}
We start deriving the existence conditions of the trivial non-informative fixed point in the Wigner model eq.~\eqref{app:Wigner}, that refers to eq.~\eqref{trivial_fixed} in the main part. These conditions can be alternatively derived from the SE eqs.~\eqref{eq:condition:one}-\eqref{eq:condition:two} - see sec.~\ref{sec:app:stability}.
\paragraph{Existence of the uninformative fixed point:}
Consider $\hat{\bv}=\bzero$. We obtain easily from the algorithm~\ref{appendix:AMP_vv_bayes}, $(\bB_v, A_v) = (\bzero,0)$, leading to $\bg = f_{\rm out}\(\bzero, 0, \bomega, V \) = \EE_{Q_{\rm out}^0}[ (\bx - {\bomega ) } ] = \bzero$, and $(\bgamma, \Lambda) =  (\bzero,0)$. Finally, inserting these values in the update functions $f_{\rm \out}$ and $f_v$, defined in eq.~\eqref{eq:app:defs:fvfout}, we obtain sufficient conditions to get the trivial fixed point in the Wigner model:
\begin{align}
		(\hat{\bv}, \hat{\bz}) = ( \bzero, \bzero ) \hhspace \textrm{ if } \mC \equiv \left\{ \hhspace
			\EE_{Q_{\rm out}^0} \[ v \] = 0 \andcase
			\EE_{P_v} \[ z \] = 0 \right \} \,. 
\end{align}
\paragraph{Linearization:}
To lighten notation, we denote with $|_{\star}$ quantities that are evaluated at $(\bB_v, A_v, \bomega, V, \bgamma, \Lambda) = (\bzero, 0, \bzero, \rho_z\rI_p, \bzero ,0)$, and we linearize the equations of the AMP algorithm~\ref{appendix:AMP_vv_bayes} around the fixed point
\begin{align}
	&(\hat{\bv}, \hat{\bc}_v)= (\bzero, \rho_v \rI_p  ),\hhspace (\hat{\bz}, \hat{\bc}_z) = ( \bzero, \rho_z \rI_k), \hhspace  \spacecase
	&(\bB_v,A_v) = (\bzero ,0),\hhspace (\bgamma,\Lambda) = (\bzero, 0), \hhspace (\bomega, V, \bg) = (\bzero, \rho_z \rI_p, \bzero )\,.
\end{align}
In a scalar formulation, the linearization yields
\begin{align}
		\delta \hat{\bv}_i^{t+1} &= \partial_B f_v|_{\star} \delta \bB^{v,t}_{i}   + \partial_A f_v|_{\star}  \delta A^{v,t}_{i}  + \partial_\omega f_v|_{\star} \delta \bomega^{t}_{i}   + \partial_V f_v|_{\star}  \delta V^{t}_{i} \,, \label{eq:lamp_uu:v} \spacecase
		\delta \hat{c}_{i}^{v,t+1} &= \partial_{B,B}^2 f_v|_{\star}  \delta \bB^{v,t}_{i}  + \partial_{A,B}^2 f_v|_{\star}  \delta A^{v,t}_{i}  + \partial_{\omega,B}^2 f_v|_{\star}  \delta \bomega^{t}_{i}   + \partial_{V,B}^2 f_v|_{\star}  \delta V^{t}_{i} \,, \label{eq:lamp_uu:cv}
			\spacecase
		\delta \hat{\bz}_{l}^{t+1} &= \partial_\bgamma f_z|_{\star}   \delta \bgamma_l^t  + \partial_\Lambda f_z|_{\star} \delta \Lambda_l^t \,, \label{eq:lamp_uu:z}
			\spacecase
		\delta \hat{c}_{i}^{z,t+1} &=  \partial_{\bgamma,\bgamma}^2 f_z|_{\star}  \delta \bgamma_l^t  + \partial_{\Lambda,\bgamma}^2 f_z|_{\star}  \delta \Lambda_l^t \,, \label{eq:lamp_uu:cz}\spacecase
		\delta \bg_i^{t} &=  \partial_B f_{\rm out}|_{\star}   \delta \bB^{v,t}_{i}  + \partial_A f_{\rm out}|_{\star} \delta A^{v,t}_{i}  + \partial_\omega f_{\rm out}|_{\star} \delta \bomega^{t}_{i}  + \partial_V f_{\rm out}|_{\star}  \delta V^{t}_{i} \, , \label{eq:lamp_uu:g}
\end{align}
with
\begin{align}
		\delta \bB^{v,t}_{i} &= \frac{1}{\Delta}	 \sum_{j=1}^p \frac{Y_{j i}}{\sqrt{p}}  \delta \hat{\bv}_{j}^{t}  -  \frac{1}{\Delta} \( \displaystyle \sum_{j=1}^p  \frac{\hat{c}_{j}^{v,t}|_\star}{p}  \) \delta\hat{\bv}_{i}^{t-1} -  \frac{1}{\Delta} \( \displaystyle \sum_{j=1}^p  \frac{ \delta\hat{c}_{j}^{v,t}}{p}  \) \hat{\bv}_{i}^{t-1}|_\star \,, \label{eq:lamp_uu:B}
			\spacecase
		\delta A^{v,t} &=  \frac{2}{\Delta} \displaystyle \sum_{j=1}^p \frac{\hat{\bv}_{j}^t|_\star \delta \hat{\bv}_{j}^t}{p} = 0 \,, \label{eq:lamp_uu:A}
			\spacecase
		\delta \bomega_{i}^t &= \frac{1}{\sqrt{k}}\displaystyle \sum_{l=1}^k W_{i l} \delta \hat{\bz}_{l}^{t} - \delta V_{i}^t g_i^{t-1}|_\star - V_{i}^t|_\star \delta \bg_i^{t-1} \,, \label{eq:lamp_uu:omega}
			\spacecase
		\delta V^t & = \frac{1}{k} \displaystyle   \sum_{l=1}^k \delta \hat{c}^{z,t}_{l}\,, \label{eq:lamp_uu:V}
			\spacecase
		\delta \Lambda^t &=  \frac{2}{k} \sum_{i=1}^p \bg_i^{t}|_\star  \delta \bg_i^{t} = 0 \,, \label{eq:lamp_uu:Lambda}
			\spacecase 
		\delta \bgamma_{l}^t &= \frac{1}{\sqrt{k}} \sum_{i=1}^p  W_{i l} \delta \bg_i^{t} + \delta\Lambda_{l}^t \hat{\bz}_{l}^t|_\star + \Lambda_{l}^t|_\star \delta\hat{\bz}_{l}^t \label{eq:lamp_uu:gamma} \,.
\end{align}
These equations can be simplified and closed over three vectorial variables $\hat{\bv} \in \bbR^{p}$, $\hat{\bz} \in \bbR^{k}$ and $\bomega \in \bbR^{p}$, where we used the existence condition $\mC$ that leads to $\partial_\omega f_{\rm out}|_\star = \partial_V f_{\rm out}|_\star = 0$. Finally, injecting eq.~\eqref{eq:lamp_uu:B}-\eqref{eq:lamp_uu:gamma} in \eqref{eq:lamp_uu:v}, \eqref{eq:lamp_uu:z}, \eqref{eq:lamp_uu:omega} we obtain
\begin{align}
	\delta \hat{\bv}^{t+1} =& \displaystyle  \frac{1}{\Delta} \partial_B f_v|_{\star} \( \displaystyle 	 \frac{Y}{\sqrt{p}}  \delta \hat{\bv}^{t}  - \partial_B f_{v}|_\star \rI_{p} \delta\hat{\bv}^{t-1} \)  + \partial_\omega f_v|_{\star} \rI_p \delta \bomega^{t} + \frac{\partial_V f_v|_{\star} \partial_{\bgamma,\bgamma}^2 f_z|_{\star}}{\partial_\bgamma f_z|_{\star}}  \frac{\id_p \id_k^\intercal }{k} \delta \hat{\bz}^{t} \label{eq:lamp_uu:v_final} \,, \spacecase
	\delta \hat{\bz}^{t+1} =& \displaystyle \frac{1}{\Delta} \partial_\bgamma f_z|_{\star}\partial_B f_{\rm out}|_{\star}   \frac{W^\intercal}{\sqrt{k}}     \[	  \frac{Y}{\sqrt{p}}  \delta \hat{\bv}^{t}  - \partial_B f_{v}|_\star \rI_p \delta \hat{\bv}^{t-1} \]  \,,
	\label{eq:lamp_uu:z_final}
		\spacecase
	\delta \bomega^{t+1} =& \displaystyle  \frac{1}{\Delta}  \(\displaystyle \partial_\bgamma f_z|_{\star}\partial_B f_{\rm out}|_{\star}   \frac{WW^\intercal }{k}     \[	  \frac{Y}{\sqrt{p}}  \delta \hat{\bv}^{t}  - \partial_B f_v|_\star \rI_p \delta\hat{\bv}^{t-1} \] \) - \label{eq:lamp_uu:omega_final} \\
	& \hspace{1cm} \partial_\bgamma f_z|_\star \partial_B f_{\rm out}|_{\star} \[ \displaystyle	 \frac{Y }{\sqrt{p}}  \delta \hat{\bv}^{t-1}  -   \partial_B f_v|_\star \rI_p \delta\hat{\bv}^{t-2} \] \nonumber\,.
\end{align}

\paragraph{Conclusion:}
This set of equations involves partial derivatives of $f_v$, $f_z$ and $f_{\rm out}$ that can be simplified using the condition $\mC$, and rewritten as moments of the distributions $P_z$ and $Q_{\rm out}$:
\begin{align}
\begin{cases}
	\partial_\bgamma f_z|_\star  &= \EE_{P_{z}} \[z^2\] = \rho_z \,,\spacecase
	\partial_{\bgamma,\bgamma}^2 f_z|_\star &= -2\partial_\Lambda f_z|_\star  =  \EE_{P_{z}} \[z^3\] \,, \spacecase
	\partial_\omega f_{\rm out}|_\star &= \partial_V f_{\rm out}|_\star = 0 \,,\spacecase
\end{cases}
\andcase
\begin{cases}
	\partial_B f_v|_\star &= \EE_{Q_{\rm out}^0} [v^2] = \rho_v \,, \spacecase
	\partial_\omega f_v|_\star &= \partial_B f_{\rm out}|_\star = \rho_z^{-1} \EE_{Q_{\rm out}^0}[v x] \,, \spacecase
	\partial_V f_v|_\star &= \frac{1}{2} \rho_z^{-2} \EE_{Q_{\rm out}^0}[v x^2]  \,.\spacecase
\end{cases}
\end{align}

Injecting eq.~\eqref{eq:lamp_uu:omega_final}-\eqref{eq:lamp_uu:z_final} in \eqref{eq:lamp_uu:v_final}, we finally obtain a closed equation over $\hat{\bv}$. Forgetting time indices, it leads the definition of the LAMP operator as
\begin{align}
	\Gamma^{vv}_p &=  \frac{1}{\Delta}  \( (a-b) \rI_p  +  b  \frac{W W^\intercal}{k} + c  \frac{\id_p \id_k^\intercal}{k} \frac{W^\intercal}{\sqrt{k} }  \) \times \( \frac{Y}{\sqrt{p}}  -  a \rI_p  \) \,, 
\end{align}
with
\begin{align}
		a \equiv \EE_{Q_{\rm out}^0} [v^2] = \rho_v\,,\hhspace 
		b \equiv \rho_z^{-1} \EE_{Q_{\rm out}^0} [v x]^2 \,,\hhspace
		c \equiv \frac{1}{2} \rho_z^{-3} \EE_{P_{z}} \[z^3\]  \EE_{Q_{\rm out}^0}[v x^2] \EE_{Q_{\rm out}^0}[v x]\,.
\end{align}
Note that in most of the cases we studied, the parameter $c$, taking into account the skewness of the variable $\bz$, is zero, simplifying considerably the structured matrix as discussed in the main part. Taking the leading eigenvector of the operator $\Gamma_p^{vv}$ leads to the LAMP algorithm.

\paragraph{Applications: }
Consider a gaussian $P_z = \mN_z\(0,1\)$ or binary $P_z = \frac{1}{2} \(\delta(z-1) + \delta(z+1)\)$ prior, for which $\rho_z=1$. Taking a noiseless channel $P_{\rm out} (v|x) = \delta\(v - \varphi(x)\)$, condition $\mC$ is verified, and we obtain simple and explicit coefficients
\begin{itemize}
	\item Linear activation ($\varphi(x) = x$):  $\(a,b,c\) = \(1,1,0\)$\,.
	\item Sign activation  ($\varphi(x) = \textrm{sgn}(x)$): $\(a,b,c\) = \(1,2/\pi,0\)$\,.
\end{itemize}  

\subsection{Wishart model: $\bu \bv^\intercal$}
In this section, we generalize the previous derivation of the LAMP spectral algorithm for the Wishart model in eq.~\eqref{app:Wishart}. The strategy is exactly the same: it follows from linearizing the AMP algorithm~\ref{appendix:AMP_uv_bayes} in its Bayes-optimal version around the trivial fixed point. Except that in this case there are more equations to deal with.
\paragraph{Existence of the uninformative fixed point:}
Consider $\(\hat{\bu},\hat{\bv}\)=\(\bzero,\bzero\)$. Injecting this condition in the algorithm's equations, we simply obtain $\(\bB_u, A_u, \bB_v, A_v \) = \(\bzero, 0, \bzero, 0\)$. However, we now need $\EE_{P_u} \[u\]=0$ for this to be consistent with the update equation for $\hat{\bu}^{t+1}$. Besides, this also implies $\bg = f_{\rm out}\(\bzero, 0, \bomega, V \) = \EE_{Q_{\rm out}^0}[ (\bx - {\bomega ) } ] = \bzero$, and $(\bgamma, \Lambda) =  (\bzero,0)$. Finally, putting all conditions together in the update equations involving $f_v$, $f_u$ and $f_{\rm out}$, defined in eq.~\eqref{eq:app:defs:fvfout}, we arrive at the following sufficient conditions for the existence of the uninformative fixed point in the Wishart model: 
\begin{align}
		(\hat{\bv}, \hat{\bz}) = ( \bzero, \bzero ) \hhspace \textrm{ if } \mC \equiv \left\{ \hhspace
			\EE_{Q_{\rm out}^0} \[ v \] = 0\,, \hhspace
			\EE_{P_v} \[ z \] = 0 \andcase \EE_{P_u} \[u\]=0 \right \} \,. 
\end{align}

\paragraph{Linearization:}
As previously, to lighten notations we denote $|_{\star}$ quantities that are evaluated at 
\begin{align*}
	(\bB_u, A_u, \bB_v, A_v, \bomega, V, \bgamma, \Lambda) = (\bzero, 0, \bzero, 0, \bzero, \rho_z\rI_p, \bzero ,0).
\end{align*}
We linearize AMP equations algorithm~\ref{appendix:AMP_uv_bayes} around the fixed point
\begin{align}
	&(\hat{\bu}, \hat{\bc}_u)= (\bzero, \rho_u \rI_n  ),\hhspace (\hat{\bv}, \hat{\bc}_v)= (\bzero, \rho_v \rI_p  ),\hhspace (\hat{\bz}, \hat{\bc}_z) = ( \bzero, \rho_z \rI_k), \hhspace  \spacecase
	&(\bB_u,A_u) = (\bzero ,0),\hhspace (\bB_v,A_v) = (\bzero ,0),\hhspace (\bgamma,\Lambda) = (\bzero, 0), \hhspace (\bomega, V, \bg) = (\bzero, \rho_z \rI_p, \bzero )\,.
\end{align}
In a scalar formulation, linearization yields four additional equations over the $\bu$ variable:
\begin{align}
		\delta \hat{\bu}_\mu^{t+1} &= \partial_B f_u|_{\star} \delta \bB^{u,t}_{\mu}   + \partial_A f_u|_{\star}  \delta A^{u,t}_{\mu}  \,, \label{eq:lamp_uv:u} \spacecase
		\delta \hat{c}_{\mu}^{u,t+1} &= \partial_{B,B}^2 f_u|_{\star}  \delta \bB^{u,t}_{\mu}  + \partial_{A,B}^2 f_u|_{\star}  \delta A^{u,t}_{\mu} \,, \label{eq:lamp_uv:cu}\spacecase
		\delta \hat{\bv}_i^{t+1} &= \partial_B f_v|_{\star} \delta \bB^{v,t}_{i}   + \partial_A f_v|_{\star}  \delta A^{v,t}_{i}  + \partial_\omega f_v|_{\star} \delta \bomega^{t}_{i}   + \partial_V f_v|_{\star}  \delta V^{t}_{i} \,, \label{eq:lamp_uv:v} \spacecase
		\delta \hat{c}_{i}^{v,t+1} &= \partial_{B,B}^2 f_v|_{\star}  \delta \bB^{v,t}_{i}  + \partial_{A,B}^2 f_v|_{\star}  \delta A^{v,t}_{i}  + \partial_{\omega,B}^2 f_v|_{\star}  \delta \bomega^{t}_{i}   + \partial_{V,B}^2 f_v|_{\star}  \delta V^{t}_{i} \,, \label{eq:lamp_uv:cv}
			\spacecase
		\delta \hat{\bz}_{l}^{t+1} &= \partial_\bgamma f_z|_{\star}   \delta \bgamma_l^t  + \partial_\Lambda f_z|_{\star} \delta \Lambda_l^t \,, \label{eq:lamp_uv:z}
			\spacecase
		\delta \hat{c}_{i}^{z,t+1} &=  \partial_{\bgamma,\bgamma}^2 f_z|_{\star}  \delta \bgamma_l^t  + \partial_{\Lambda,\bgamma}^2 f_z|_{\star}  \delta \Lambda_l^t \,, \label{eq:lamp_uv:cz}\spacecase
		\delta \bg_i^{t} &=  \partial_B f_{\rm out}|_{\star}   \delta \bB^{v,t}_{i}  + \partial_A f_{\rm out}|_{\star} \delta A^{v,t}_{i}  + \partial_\omega f_{\rm out}|_{\star} \delta \bomega^{t}_{i}  + \partial_V f_{\rm out}|_{\star}  \delta V^{t}_{i} \, , \label{eq:lamp_uv:g}
\end{align}
and
\begin{align}
		\delta \bB^{u,t}_{\mu} &= \frac{1}{\Delta}	 \sum_{i=1}^p \frac{Y_{\mu i}}{\sqrt{p}}  \delta \hat{\bv}_{i}^{t}  -  \frac{1}{\Delta} \( \displaystyle \sum_{i=1}^p  \frac{\hat{c}_{i}^{v,t}|_\star}{p}  \) \delta\hat{\bu}_{\mu}^{t-1} -  \frac{1}{\Delta} \( \displaystyle \sum_{i=1}^p  \frac{ \delta\hat{c}_{i}^{v,t}}{p}  \) \hat{\bu}_{\mu}^{t-1}|_\star \,, \label{eq:lamp_uv:Bu}
			\spacecase
		\delta A^{u,t} &=  \frac{2}{\Delta} \displaystyle \sum_{i=1}^p \frac{\hat{\bv}_{i}^t|_\star \delta \hat{\bv}_{i}^t}{p} = 0 \,, \label{eq:lamp_uv:Au}
			\spacecase
		\delta \bB^{v,t}_{i} &= \frac{1}{\Delta}	 \sum_{\mu=1}^n \frac{Y_{\mu i}}{\sqrt{p}}  \delta \hat{\bu}_{\mu}^{t}  -  \frac{1}{\Delta} \( \displaystyle \sum_{\mu=1}^n  \frac{\hat{c}_{\mu}^{u,t}|_\star}{p}  \) \delta\hat{\bv}_{i}^{t-1} -  \frac{1}{\Delta} \( \displaystyle \sum_{\mu=1}^n  \frac{ \delta\hat{c}_{\mu}^{u,t}}{p}  \) \hat{\bv}_{i}^{t-1}|_\star \,, \label{eq:lamp_uv:Bv}
			\spacecase
		\delta A^{v,t} &=  \frac{2}{\Delta} \displaystyle \sum_{\mu=1}^n \frac{\hat{\bu}_{\mu}^t|_\star \delta \hat{\bu}_{\mu}^t}{p} = 0 \,, \label{eq:lamp_uv:Av}
			\spacecase
		\delta \bomega_{i}^t &= \frac{1}{\sqrt{k}}\displaystyle \sum_{l=1}^k W_{i l} \delta \hat{\bz}_{l}^{t} - \delta V_{i}^t \bg_i^{t-1}|_\star - V_{i}^t|_\star \delta \bg_i^{t-1} \,, \label{eq:lamp_uv:omega}
			\spacecase
		\delta V^t & = \frac{1}{k} \displaystyle   \sum_{l=1}^k \delta \hat{c}^{z,t}_{l}\,, \label{eq:lamp_uv:V}
			\spacecase
		\delta \Lambda^t &=  \frac{2}{k} \sum_{i=1}^p g_i^{t}|_\star  \delta \bg_i^{t} = 0 \,, \label{eq:lamp_uv:Lambda}
			\spacecase 
		\delta \bgamma_{l}^t &= \frac{1}{\sqrt{k}} \sum_{i=1}^p  W_{i l} \delta \bg_i^{t} + \delta\Lambda_{l}^t \hat{\bz}_{l}^t|_\star + \Lambda_{l}^t|_\star \delta\hat{\bz}_{l}^t \label{eq:lamp_uv:gamma} \,.
\end{align}
These equations can be closed over four vectorial variables $\hat{\bu} \in \bbR^n, \hat{\bv} \in \bbR^{p}$, $\hat{\bz} \in \bbR^{k}$ and $\bomega \in \bbR^{p}$, where we used the existence condition $\mC$ leading again to $\partial_\omega f_{\rm out}|_\star = \partial_V f_{\rm out}|_\star = 0$. Finally, injecting eq.~\eqref{eq:lamp_uv:Bu}-\eqref{eq:lamp_uv:gamma} in \eqref{eq:lamp_uv:u}, \eqref{eq:lamp_uv:v}, \eqref{eq:lamp_uv:z}, \eqref{eq:lamp_uv:omega} we obtain:
\begin{align}
	\delta \hat{\bu}^{t+1} =& \displaystyle  \frac{1}{\Delta} \partial_B f_u|_{\star} \( \displaystyle 	 \frac{Y}{\sqrt{p}}  \delta \hat{\bv}^{t}  - \partial_B f_{v}|_\star \rI_{n} \delta\hat{\bu}^{t-1} \)\,,  \label{eq:lamp_uv:u_final} \spacecase
	\delta \hat{\bv}^{t+1} =& \displaystyle  \frac{1}{\Delta} \partial_B f_v|_{\star} \( \displaystyle 	 \frac{Y^\intercal}{\sqrt{p}}  \delta \hat{\bu}^{t}  - \beta \partial_B f_{u}|_\star \rI_{p} \delta\hat{\bv}^{t-1} \)  + \partial_\omega f_v|_{\star} \rI_p \delta \bomega^{t} + \frac{\partial_V f_v|_{\star} \partial_{\bgamma,\bgamma}^2 f_z|_{\star}}{\partial_\bgamma f_z|_{\star}}  \frac{\id_p \id_k^\intercal }{k} \delta \hat{\bz}^{t} \,,\label{eq:lamp_uv:v_final} \spacecase
	\delta \hat{\bz}^{t+1} =& \displaystyle \frac{1}{\Delta} \partial_\bgamma f_z|_{\star}\partial_B f_{\rm out}|_{\star}   \frac{W^\intercal}{\sqrt{k}}     \[	  \frac{Y^\intercal}{\sqrt{p}}  \delta \hat{\bu}^{t}  - \beta\partial_B f_{u}|_\star \rI_p \delta \hat{\bv}^{t-1} \]  \,,
	\label{eq:lamp_uv:z_final}
		\spacecase
	\delta \bomega^{t+1} =& \displaystyle  \frac{1}{\Delta}  \(\displaystyle \partial_\bgamma f_z|_{\star}\partial_B f_{\rm out}|_{\star}   \frac{W^\intercal}{\sqrt{k}}     \[	  \frac{Y^\intercal}{\sqrt{p}}  \delta \hat{\bu}^{t}  - \beta\partial_B f_{u}|_\star \rI_p \delta \hat{\bv}^{t-1} \] \) - \label{eq:lamp_uv:omega_final} \\
	& \hspace{4cm} \partial_\bgamma f_z|_\star \partial_B f_{\rm out}|_{\star} \[	  \frac{Y^\intercal}{\sqrt{p}}  \delta \hat{\bu}^{t-1}  - \beta\partial_B f_{u}|_\star \rI_p \delta \hat{\bv}^{t-2} \] \,.\nonumber
\end{align}

\paragraph{Conclusion:}
This set of equations involves partial derivatives of $f_u$, $f_v$, $f_{\rm out}$ that can be simplified using the condition $\mC$ and rewritten as moments of distributions $P_u, P_z$ and $Q_{\rm out}$:
\begin{align}
\begin{cases}
	\partial_\bgamma f_z|_\star  &= \EE_{P_{z}} \[z^2\] = \rho_z \,,\spacecase
	\partial_{\bgamma,\bgamma}^2 f_z|_\star &= -2\partial_\Lambda f_z|_\star  =  \EE_{P_{z}} \[z^3\] \,, \spacecase
	\partial_\omega f_{\rm out}|_\star &= \partial_V f_{\rm out}|_\star = 0 \,,\spacecase
	\partial_B f_u|_\star &= \EE_{P_u} [u^2] = \rho_u \,, \spacecase
\end{cases}
\andcase
\begin{cases}
	\partial_B f_v|_\star &= \EE_{Q_{\rm out}^0} [v^2] = \rho_v \,, \spacecase
	\partial_\omega f_v|_\star &= \partial_B f_{\rm out}|_\star = \rho_z^{-1} \EE_{Q_{\rm out}^0}[v x] \,, \spacecase
	\partial_V f_v|_\star &= \frac{1}{2} \rho_z^{-2} \EE_{Q_{\rm out}^0}[v x^2]  \,.\spacecase
\end{cases}
\end{align}

Injecting eq.~\eqref{eq:lamp_uv:omega_final},\eqref{eq:lamp_uv:z_final}-\eqref{eq:lamp_uv:u_final} in \eqref{eq:lamp_uv:v_final}, we finally obtain a self-consistent equation over $\hat{\bv}$ that, forgetting time indices, leads to define the following LAMP structured matrix, from which we need to compute the top eigenvector: 
\begin{align}
	\Gamma^{uv}_p =  \frac{1}{\Delta}  \( (a-b) \rI_p  +  b  \frac{W W^\intercal}{k} + c  \frac{\id_p \id_k^\intercal}{k} \frac{W^\intercal}{\sqrt{k} }  \) \times \( \frac{1}{a + \frac{\Delta}{d}} \frac{Y^\intercal Y}{p}  -  d \beta \rI_p  \) \,,
\end{align}
with
\begin{align}
	a \equiv \rho_v\,,  \hhspace c \equiv \frac{1}{2} \rho_z^{-3} \EE_{P_{z}} \[z^3\]  \EE_{Q_{\rm out}^0}[v x^2] \EE_{Q_{\rm out}^0}[v x], \hhspace
	b \equiv  \rho_z^{-1} \EE_{Q_{\rm out}^0} [v x]^2 \,, \hhspace d \equiv \rho_u \,.
\end{align}

\paragraph{Applications:}
Consider a gaussian $P_z, P_u = \mN\(0,1\)$ or binary $P_z, P_u = \frac{1}{2} \(\delta(z-1) + \delta(z+1)\)$ prior, for which $\rho_z=\rho_u=1$. For a noiseless channel $P_{\rm out} (v|x) = \delta\(v - \varphi(x)\)$, we obtain the following simple and explicit coefficients:
\begin{itemize}
	\item Linear, $\varphi(x) = x$:  $\(a,b,c, d\) = \(1,1,0,1\)$
	\item Sign, $\varphi(x) = \textrm{sgn}(x)$: $\(a,b,c,d\) = \(1,2/\pi,0,1\)$
\end{itemize}

\subsection{State evolution equations of LAMP and PCA - linear case}
\label{appendix:subsec:SE_lamp}
In this section we describe how to obtain the limiting behaviour of the LAMP spectral method for the Wigner model in the large size limit $p \to \infty$. We will show that in the linear case, mean squared errors of LAMP and PCA are directly obtained from the optimal overlap performed by AMP or its state evolution. Recall that the numerical simulations of LAMP and PCA are compared with their state evolution in Fig.~\ref{main:bbp_lamp_amp_se}, with green and red lines respectively. 

\paragraph{LAMP:}
For the noiseless linear channel $P_{\rm out} (v|x) = \delta\(v - x\)$, the set of eqs.~(\ref{eq:lamp_uu:v_final}-\ref{eq:lamp_uu:omega_final}) are already linear, and do not require linearizing as above. Hence the LAMP spectral method flows directly from the AMP eqs.~\eqref{appendix:AMP_vv_bayes}. As a consequence, this means that the state evolution equations associated to the spectral method are simply dictated by the set of AMP state evolution equations from sec.~\ref{appendix:se_from_amp_bayes_vv}. However, it is worth stressing that the LAMP MSE is not given by the AMP mean squared error, as LAMP returns a normalized estimator. We now compute the overlaps and mean squared error performed by this spectral algorithm.

Recall that $m_v$ and $q_v$ are the parameters defined in eq.~\eqref{appendix:amp:overlap_definition}, that respectively measure the overlap between the ground truth $\bv^\star$ and the estimator $\hat{\bv}$, and the norm of the estimator. In eq.~\eqref{appendix:proof_mmse}, the MSE is given by:
\begin{align}
     {\rm MSE}_v &= \rho_v + \EE_{\bv^\star} \lim_{p\to \infty} \frac{1}{p} \|\hat{\bv}\|_2^2 - 2 \EE_{\bv^\star} \lim_{p\to \infty} \frac{1}{p} \hat{\bv}^\intercal \bv^\star\\
                            &=\rho_v + q_v - 2 m_v \,,
\end{align} 
However the LAMP spectral method computes the normalized top eigenvector of the structured matrix $\Gamma_p$. Hence the norm of the LAMP estimator is $\|\hat{\bv}\|_{\textrm{LAMP}}^2 = q_{v,\textrm{LAMP}} = 1$, while the Bayes-optimal AMP estimator is not normalized with $\|\hat{\bv}\|_{\textrm{AMP}}^2=q_{v,\textrm{AMP}}^\star=m_{v,\textrm{AMP}}^\star \ne 1$, solutions of eq.~\eqref{appendix:se_from_amp_bayes_vv}. As the non-normalized LAMP estimator follows AMP state evolutions in the linear case, the overlap with the ground truth is thus given by:
\begin{align}
	m_{v,\textrm{LAMP}} &\equiv  \EE_{\bv^\star} \lim_{p\to \infty} \frac{1}{p} \hat{\bv}_{\textrm{LAMP}}^\intercal \bv^\star  = \EE_{\bv^\star} \lim_{p\to \infty} \frac{1}{p} \(\frac{\hat{\bv}_{\textrm{AMP}}}{\|\hat{\bv}\|_{\textrm{AMP}}}\)^\intercal \bv^\star \\
	&= \frac{m_{v,\textrm{AMP}}^\star}{\(q_{v,\textrm{AMP}}^\star\)^{1/2}} = \(m_{v,\textrm{AMP}}^\star\)^{1/2}\,.
\end{align}
Finally the mean squared error performed by the LAMP method is easily obtained from the optimal overlap reached by the AMP algorithm and yields
\begin{align}
     {\rm MSE}_{v,\rm LAMP} = \rho_v + 1 - 2 \(q_{v,\textrm{AMP}}^\star\)^{1/2} \,.
\end{align}

\paragraph{PCA:}
Similarly, in the noiseless linear channel case, we note that at $\alpha=0$, LAMP reduces exactly to PCA, i.e. it consists in finding the top eigenvector of $Y$, instead $\Gamma_p$. As LAMP follows AMP in this case, we can simply state that the mean squared error performed by PCA is computed using the optimal overlap reached by AMP at $\alpha=0$:
\begin{align}
     {\rm MSE}_{v,\rm PCA} = \rho_v + 1 - 2 \(q_{v,\textrm{AMP}}^\star|_{\alpha=0}\)^{1/2} \,.
\end{align}

%% file: files/supplementary/transition_se.tex
\label{}
In this section we derive sufficient conditions for the existence of the uninformative fixed point $(q_{v}, \hat{q}_{z},q_{z}) = (0,0,0)$ from the state evolution eqs.~\eqref{main:SE_AMP_uu}. In the case $(0,0,0)$ is a fixed point, we derive its stability, obtaining the Jacobian in eq.~\eqref{eq:phasetransition:jacobian}. Its eigenvalues determine the regions for which $(0,0,0)$ is stable and unstable, and therefore the critical point $\Delta_{c}$ where the transition occurs.

For the purpose of our analysis we define the following shorthand notation for the update functions,
\begin{align}
    \textbf{f}\left(r,t,s\right) \equiv
    \begin{pmatrix}
        f_{1}\left(r,s\right)\\
        f_{2}\left(r,s\right)\\
        f_{3}\left(t\right)
    \end{pmatrix}
\end{align}
\noindent where $(f_1,f_2,f_3)$ are explicitly given by
\begin{align}
  f_{1}(r,s) &= 2\partial_{r}\Psi_{\out}(r,s) =  \mathbb{E}_{\xi,\eta}\left[\frac{\left(\int\dd v ~e^{-\frac{r}{2}v^2+\sqrt{r}v\xi}\int\frac{\dd x}{\sqrt{2\pi (\rho_{z}-s)}}e^{-\frac{1}{2}\frac{(x-\sqrt{s}\eta)^2}{\rho_{z}-s}}P_{\out}(v|x) v\right)^2}{\int\dd v ~e^{-\frac{r}{2}v^2+\sqrt{r}v\xi}\int\frac{\dd x}{\sqrt{2\pi (\rho_{z}-s)}}e^{-\frac{1}{2}\frac{(x-\sqrt{s}\eta)^2}{\rho_{z}-s}}P_{\out}(v|x)}\right]\notag\\
 f_{2}(r,s) &= 2\alpha\partial_{s}\Psi_{\out}(r,s) =\alpha \mathbb{E}_{\xi,\eta}\left[\frac{\left(\int\dd v ~e^{-\frac{r}{2}v^2+\sqrt{r}v\xi}\int\frac{\dd x}{\sqrt{2\pi (\rho_{z}-s)}}e^{-\frac{1}{2}\frac{(x-\sqrt{s}\eta)^2}{\rho_{z}-s}}P_{\out}(v|x) (x-\sqrt{s}\eta)\right)^2}{\int\dd v ~e^{-\frac{r}{2}v^2+\sqrt{r}v\xi}\int\frac{\dd x}{\sqrt{2\pi (\rho_{z}-s)}}e^{-\frac{1}{2}\frac{(x-\sqrt{s}\eta)^2}{\rho_{z}-s}}P_{\out}(v|x)}\right]\notag\\
  f_{3}(t) &= 2\partial_{t}\Psi_{z}(t)=\mathbb{E}_{\xi}\left[\frac{\left(\int\dd x~P_{z}(z)e^{-\frac{t}{2}z^2+\sqrt{t}z\xi}z\right)^2}{\int\dd x~P_{z}(z)e^{-\frac{t}{2}z^2+\sqrt{t}\xi z}}\right]\label{eq:se:qhat}
\end{align}
In terms of these, the right-hand side of the state evolution equations is given by evaluating $(r,t, s)=\left(\frac{q_{v}}{\Delta},\hat{q}_{z},q_{z}\right)$.

\subsection{Conditions for fixed point}
Note that the denominator in the first two state evolution equations is actually constant at $r=0$,
\begin{align}
\int\dd v\int\frac{\dd x}{\sqrt{2\pi \rho_{z}}}e^{-\frac{1}{2\rho_{z}}x^2}P_{\out}(v|x) =\int\frac{\dd x}{\sqrt{2\pi \rho_{z}}}e^{-\frac{1}{2\rho_{z}}x^2} \left(\int\dd v~P_{\out}(v|x) =\int\frac{\dd x}{\sqrt{2\pi \rho_{z}}}e^{-\frac{1}{2\rho_{z}}x^2}\right) = 1.
\end{align}
And in particular, this means that
\begin{align}
    f_{2}(0,s) &= \mathbb{E}_{\xi,\eta}\left(\int\dd v\int\frac{\dd x}{\sqrt{2\pi \rho_{z}}}e^{-\frac{1}{2\rho_{z}}x^2}P_{\out}(v|x)\left(x-\sqrt{s}\eta\right)\right)^2 \notag\\
    &= \mathbb{E}_{\xi,\eta}\left(\int\frac{\dd x}{\sqrt{2\pi \rho_{z}}}e^{-\frac{1}{2\rho_{z}}x^2}\left(x-\sqrt{s}\eta\right)\int\dd v~ P_{\out}(v|x)\right)^2 \notag\\
    &= \mathbb{E}_{\xi,\eta}\left(\int\frac{\dd x}{\sqrt{2\pi \rho_{z}}}e^{-\frac{1}{2\rho_{z}}x^2}\left(x-\sqrt{s}\eta\right)\right)^2 = 0
\end{align}
\noindent for any value of $s\in\mathbb{R}$. In terms of the overlaps, this means that if $q_{u}$ is a fixed point, we necessarily have $\hat{q}_{z} = 0$. What is the implication for $q_{z}$? We need to look at $f_{3}(\hat{q}_{z}=0)$, which is simply given by
\begin{align}
    f_3(0) = \mathbb{E}_{\xi}\left(\int\dd x~P_{z} z \right)^2.
\end{align}
This means that if $q_{u} = 0$ and $P_{z}$ has zero mean, then $q_{z} =0$. It remains to check what is a sufficient condition for $q_{u}=0$ to be a fixed point. This is the case if
\begin{align}
    f_{1}(0,0) = \mathbb{E}_{\xi,\eta}\left(\int\dd v\int\frac{\dd x}{\sqrt{2\pi \rho_{z}}}e^{-\frac{1}{2\rho_{z}}x^2}P_{\out}(v|x)v\right)^2 \overset{!}{=} 0
\end{align}
\noindent implying
\begin{align}
\int\dd v\int\frac{\dd x}{\sqrt{2\pi \rho_{z}}}e^{-\frac{1}{2\rho_{z}}x^2}P_{\out}(v|x)v = \int\frac{\dd x}{\sqrt{2\pi \rho_{z}}}e^{-\frac{1}{2\rho_{z}}x^2}\left(\int\dd v~P_{\out}(v|x)v \right) \overset{!}{=} 0
\end{align}
Therefore a set of sufficient conditions for $(q_{u},\hat{q}_{z},q_{z})=(0,0,0)$ to be a fixed point of the state evolution equations are
\begin{align}
    \mathbb{E}_{P_{z}}z &= \int\dd x~P_{z}(z) z = 0\label{eq:condition:one}\\
    \mathbb{E}_{Q^{0}_{\out}}v &= \int\dd v\int \frac{\dd x}{\sqrt{2\pi \rho_{z}}}e^{-\frac{1}{2\rho_{z}}x^2}~P_{\out}(v|x) v =  0\label{eq:condition:two}
\end{align}
\noindent note that the last condition is equivalent to requiring the function $m(x) = \mathbb{E}_{P_{\out}}v$ to be odd.

%%%%%%%%%%%%%%%%%%%%%%%%%%%%%%%%%%
\subsection{Stability analysis}
%%%%%%%%%%%%%%%%%%%%%%%%%%%%%%%%%%
We now study the stability of the fixed point $(r,t,s)=(0,0,0)$, which is determined by the linearisation of the state evolution equations. But before, to help in the analysis we introduce notation.

%%%%%%%%%%%%%%%%%%%%%%%%%%%%%%%%%%
\paragraph{Some notation}
%%%%%%%%%%%%%%%%%%%%%%%%%%%%%%%%%%
It will be useful to introduce the following notation for the denoising functions in eq.~\eqref{eq:app:intro:defQ} evaluated at the overlaps:
\begin{align}
    Q_{\out}^{(r,s)}(v,x;\xi,\eta) &=\frac{1}{\mathcal{Z}_{\out}^{(r,s)}(\xi,\eta)} e^{-\frac{r}{2}u^2+\sqrt{r}\xi u}\frac{1}{\sqrt{2\pi(\rho_{z}-s)}}e^{-\frac{1}{2}\frac{(x-\sqrt{s}\eta)^2}{\rho_{z}-s}}P_{\out}(v|x) \\
    Q_{z}^{t}(z;\xi) &= \frac{1}{\mathcal{Z}_{z}^{t}(\xi)}e^{-\frac{t}{2}z^2+\sqrt{t}\xi z}P_{z}(z)
\end{align}
\noindent where $\mathcal{Z}^{(r,s)}_{\out}$ and $\mathcal{Z}_{z}$ are the normalisation of the distributions, given explicitly by
\begin{align}
    \mathcal{Z}_{\out}^{(r,s)}(\xi,\eta) &= \int \dd v ~e^{-\frac{r}{2}v^2+\sqrt{r}v\xi}\int\frac{\dd x}{\sqrt{2\pi (\rho_{z}-s)}}e^{-\frac{1}{2}\frac{(x-\sqrt{s}\eta)^2}{\rho_{z}-s}}P_{\out}(v|x)\notag\\
    \mathcal{Z}_{z}^{t}(\xi) &= \int\dd x~Q_{z}^{t}(z;\xi) = \int\dd x~P_{z}(z)e^{-\frac{t}{2}z^2+\sqrt{t}\xi z}
\end{align}
Note that $Q_{\out}$ is a family of joint distributions over $(v,x)$, indexed by $r,s\in\mathbb[0,1]$. It will be useful to have in mind the following particular cases,
\begin{align}
    Q_{\out}^{(0,s)}(v,x;\eta) &= \frac{1}{\sqrt{2\pi(\rho_{z}-s)}}e^{-\frac{1}{2}\frac{\left(x-\sqrt{s}\eta\right)^2}{\rho_{z}-s}}P_{\out}(v|x)\\
    Q_{\out}^{(r,0)}(v,x;\xi) &=\frac{1}{\mathcal{Z}_{\out}^{(r,0)}(\xi,\eta)}e^{-\frac{r}{2}v^2+\sqrt{r}v\xi}\frac{1}{\sqrt{2\pi \rho_{z}}}e^{-\frac{1}{2\rho_{z}}x^2}
\end{align}
\noindent where we have used that $\mathcal{Z}_{\out}^{(0,s)}(\eta,\xi) = 1$ (as shown above). It is also useful to define short hands to the associated distributions when we evaluate both $(r,s)=(0,0)$,
\begin{align}
    Q_{\out}^{0}(v,x) = Q_{\out}^{(0,0)}(v,x;\xi,\eta) &= \frac{1}{\sqrt{2\pi \rho_{z}}}e^{-\frac{1}{2\rho_{z}}x^2}P_{\out}(v|x)
\end{align}
\noindent while $Q_{z}^{0}(z;\xi) = P_{z}(z)$. Note that they are indeed independent of the noises, and that in particular we have $\mathcal{Z}_{z}^{0}(\xi) = 1$.

In this notation the condition in eq.~\eqref{eq:condition:two} simply reads that $v$ has mean zero with respect to the $Q_{\out}^{0}$,
\begin{align}
    \mathbb{E}_{Q_{\out}^{0}} v = 0\label{eq:condition:marginal}
\end{align}

%%%%%%%%%%%%%%%%%%%%%%%%%%%%%%%%%%%%%%%%%%%%%%%
\subsection*{Expansion around the fixed point}
%%%%%%%%%%%%%%%%%%%%%%%%%%%%%%%%%%%%%%%%%%%%%%%
We now suppose $(r,t,s)=(0,0,0)$ is a fixed point of the state evolution equations, i.e. that the conditions in eqs.~\eqref{eq:condition:one} and \eqref{eq:condition:two} hold. We are interested in the leading order expansion of the update functions $(f_1, f_2, f_3)$ around this point.

%%%%%%%%%%%%%%%%%%%%%%%%%%%%%%%%%%%%
\paragraph{Expansion of $f_{1}$:}
%%%%%%%%%%%%%%%%%%%%%%%%%%%%%%%%%%%%
Since $(f_1,f_2)$ are functions of $(r,s)$ only, we look them separately first. Instead of expanding around $(r,s)=(0,0)$ together, we first expand around $r=0$ keeping $s$ fixed. This allow us to take the average over $\xi$ explicitly simplifying the expansion considerably,
\begin{align}
    f_{1}(r,s) \underset{r\ll 1}{=} \mathbb{E}_{\eta}\left\{\left(\mathbb{E}_{Q_{\out}^{(0,s)}}v\right)^2+\left[\left(\mathbb{E}_{Q_{\out}^{(0,s)}}v\right)^4+\left(\mathbb{E}_{Q_{\out}^{(0,s)}}v^2\right)^2-2\left(\mathbb{E}_{Q_{\out}^{(0,s)}}v\right)^2\mathbb{E}_{Q_{\out}^{(0,s)}}v^2\right]r+O\left(r^{3/2}\right)\right\}\label{eq:f1:xexpansion}
\end{align}
We can now focus on the leading order expansion around $s=0$. Note we have,
\begin{align}
    \mathbb{E}_{Q_{\out}^{(0,s)}}v &= \int\dd v\int\frac{\dd x}{\sqrt{2\pi(\rho_{z}-s)}}e^{-\frac{1}{2}\frac{(x-\sqrt{s}\eta)^2}{\rho_{z}-s}}P_{\out}(v|x)~v\notag\\
    &\underset{s\ll 1}{=} \mathbb{E}_{\rho_{0}^{v}}v+\frac{\sqrt{s}\eta}{\rho_{z}}  \mathbb{E}_{Q_{\out}^{0}}vx-\frac{s}{2}\frac{\eta^2-1}{\rho_{z}^2}\left(\rho_{z}\mathbb{E}_{Q_{\out}^{0}}v-\mathbb{E}_{Q_{\out}^{0}}x^2v\right)+O\left(s^{3/2}\right)\\
    &=\frac{\sqrt{s}\eta}{\rho_{z}}  \mathbb{E}_{Q_{\out}^{0}}vx+\frac{s}{2}\frac{\eta^2-1}{\rho_{z}^2}\mathbb{E}_{Q_{\out}^{0}}x^2v+O\left(s^{3/2}\right)
\end{align}
\noindent where we used the consistency condition in eq.~\eqref{eq:condition:marginal} that ensures $(r,s) =(0,0)$ is indeed a fixed point. Moreover, the leading order term in the expansion of $\mathbb{E}_{Q_{\out}^{(0,s)}}v$ is $O(s^{1/2})$, therefore ${\left(\mathbb{E}_{Q_{\out}^{(0,s)}}v\right)^2 \sim O(s)}$ and $\left(\mathbb{E}_{Q_{\out}^{(0,s)}}v\right)^4 \sim O\left(s^2\right)$. Expanding now eq.~\eqref{eq:f1:xexpansion} to leading order in $y$ gives
\begin{align}
    f_{1}(r,s) &\underset{r, s\ll 1}{=} \mathbb{E}_{\eta}\left[\frac{s}{\rho_{z}^2}\eta^2\left(\mathbb{E}_{Q_{\out}^{0}}vx\right)^2+r\left(\mathbb{E}_{\rho_{0}^{v}}v^2\right)^2+O\left(r^{3/2},s^{3/2}\right)\right]\notag\\
    & = \frac{s}{\rho_{z}^2}\left(\mathbb{E}_{Q_{\out}^{0}}vx\right)^2+r\left(\mathbb{E}_{\rho_{0}^{v}}v^2\right)^2+O\left(r^{3/2},s^{3/2}\right)
\end{align}
From this expansion we read the first two entries of the Jacobian,
\begin{align}
    \partial_{r}f_{1}|_{(0,0)} = \left(\mathbb{E}_{Q_{\out}^{0}}v^2\right)^2 && \partial_{s}f_{1}|_{(0,0)} =\frac{1}{\rho_{z}^2} \left(\mathbb{E}_{Q_{\out}^{0}}vx\right)^2
\end{align}

%%%%%%%%%%%%%%%%%%%%%%%%%%%%%%%%%%%%
\paragraph{Expansion of $f_{2}$:}
%%%%%%%%%%%%%%%%%%%%%%%%%%%%%%%%%%%%
For $f_2$, we start by expanding with respect to $s$, allowing us to take the average with respect to $\eta$ explicitly,
\begin{align}
    f_{2}(r,s) &\underset{s\ll 1}{=}\alpha \mathbb{E}_{\xi}\left\{\left(\mathbb{E}_{Q_{\out}^{(r,0)}}x\right)^2+\frac{s}{2\rho_{z}^2}\left[2\left(\mathbb{E}_{Q_{\out}^{(r,0)}}x\right)^4-4\left(\mathbb{E}_{Q_{\out}^{(r,0)}}x\right)^2\mathbb{E}_{Q_{\out}^{(r,0)}}x^2+2\left(\mathbb{E}_{Q_{\out}^{(r,0)}}x^2-\rho_{z}\right)^2\right]\right\}\label{eq:expansionx:f2}
\end{align}
We can now focus on the leading order expansion around $r=0$. Note that
\begin{align}
   \mathbb{E}_{Q_{\out}^{(r,0)}}x &\underset{r\ll 1}{=} \mathbb{E}_{Q_{\out}^{0}}x +\sqrt{r}\xi  \mathbb{E}_{Q_{\out}^{0}}xv+\frac{r}{2}(\xi^2-1)\mathbb{E}_{Q_{\out}^{0}}xv^2+O\left(r^{3/2}\right)\\
   &=\sqrt{r}\xi  \mathbb{E}_{Q_{\out}^{0}}xv+\frac{r}{2}(\xi^2-1)\mathbb{E}_{Q_{\out}^{0}}xv^2+O\left(r^{3/2}\right)
\end{align}
\noindent since
\begin{align}
    \mathbb{E}_{Q_{\out}^{0}}x = \int\dd v\int\frac{\dd x}{\sqrt{2\pi \rho_{z}}}e^{-\frac{1}{2\rho_{z}}x^2}P_{\out}(v|x) x = \int\frac{\dd x}{\sqrt{2\pi \rho_{z}}}e^{-\frac{1}{2\rho_{z}}x^2} x = 0.
\end{align}
Therefore the leading order term is of order $O(r^{1/2})$, and $\left(\mathbb{E}_{Q_{\out}^{0}}x\right)^2 \sim O(s)$, $\left(\mathbb{E}_{Q_{\out}^{0}}x\right)^4 \sim O(s^2)$. Expanding now eq.~\eqref{eq:expansionx:f2} in $r\ll 1$,
\begin{align}
    f_{2}(r,s)&\underset{x,s\ll 1}{=}\alpha\mathbb{E}_{\xi} \left[x\xi^2 \left(\mathbb{E}_{Q_{\out}^{0}}vx\right)^2 +\frac{s}{\rho_{z}^2}\left( \mathbb{E}_{Q_{\out}^{0}}x^2-\rho_{z}\right)^2\right]+O\left(r^{3/2},s^{3/2}\right)\\
    &=r\alpha\left(\mathbb{E}_{Q_{\out}^{0}}vx\right)^2 +\frac{s}{\rho_{z}^2}\alpha\left( \mathbb{E}_{Q_{\out}^{0}}x^2-\rho_{z}\right)^2+O\left(r^{3/2},s^{3/2}\right)
\end{align}
From this expansion we can read the second two entries of the Jacobian,
\begin{align}
    \partial_{r}f_{2}|_{(0,0)} =\alpha \left(\mathbb{E}_{Q_{\out}^{0}}vx\right)^2 && \partial_{s}f_{2}|_{(0,0)} =\frac{\alpha}{\rho_{z}^2}\left(\mathbb{E}_{Q_{\out}^{0}}x^2-\rho_{z}\right)^2
\end{align}

%%%%%%%%%%%%%%%%%%%%%%%%%%%%%%%%%%
\paragraph{Expansion of $f_{3}$:}
%%%%%%%%%%%%%%%%%%%%%%%%%%%%%%%%%%
Note that $f_3$ is independent of $(r,s)$, so it can be treated separately. Expanding in $t\ll 1$ gives
\begin{align}
    f_{3}(t) = \mathbb{E}_{\xi}\left[\frac{1}{\mathcal{Z}^{t}_{z}}\left( \int\dd x~P_{z}(z) e^{-\frac{t}{2}z^2+\sqrt{t} z\xi} z\right)^2\right]\underset{t\ll 1}{=}\left(\mathbb{E}_{P_{z}}z^2\right)^2 t + O(t^{3/2})
\end{align}
\noindent where we have used the consistency condition in eq.~\eqref{eq:condition:one}. Therefore
\begin{align}
    \partial_{t}f_{3}|_{t=0} = \left(\mathbb{E}_{P_{z}}z^2\right)^2
\end{align}

%%%%%%%%%%%%%%%%%%%%%%%%%%%%%%%%%%%%%%%%%%
\subsection*{Bringing the overlaps back}
%%%%%%%%%%%%%%%%%%%%%%%%%%%%%%%%%%%%%%%%%%
In our problem, we have
\begin{align}
    r = \frac{q_{u}}{\Delta} && t = \hat{q}_{z} && s = q_{z} 
\end{align}
\noindent and therefore the partial derivatives have to be re-scaled,
\begin{align}
    \partial_{r} = \Delta\partial_{q_{u}} && \partial_{t} = \partial_{\hat{q}_{z}} &&  \partial_{s} = \partial_{q_{z}}
\end{align}
And therefore the Jacobian of the problem is
\begin{align}
    \dd \textbf{f}(0,0,0) =
    \begin{pmatrix}
        \frac{1}{\Delta}\left(\mathbb{E}_{Q_{\out}^{0}}v^2\right)^2 & 0 & \frac{1}{\rho_{z}^{2}}\left(\mathbb{E}_{Q_{\out}^{0}}vx\right)^2 \\
        \frac{\alpha}{\Delta}\left(\mathbb{E}_{Q_{\out}^{0}}vx\right)^2 & 0 & \frac{\alpha}{\rho_{z}^2}\left(\mathbb{E}_{Q_{\out}^{0}}x^2-\rho_{z}\right)^2 \\ 
        0 & \left(\mathbb{E}_{P_{z}}z^2\right)^2 & 0
    \end{pmatrix}
\end{align}

%%%%%%%%%%%%%%%%%%%%%%%%%%%%%%%%%%%%%%%%%%%%%%%%%%%%%%%%%%
\subsection*{Jacobian for the $\bu\bv^{\intercal}$ model}
%%%%%%%%%%%%%%%%%%%%%%%%%%%%%%%%%%%%%%%%%%%%%%%%%%%%%%%%%%
The main difference in the Wishart model is that the state evolution is given in terms of four variables $(p,r,t,s) \equiv \left(\frac{q_{u}}{\Delta},\beta \frac{q_{v}}{\Delta},q_{z},\hat{q}_{z}\right)$, with the update functions given by
\begin{align}
\textbf{f}(p,r,t,s) = 
\begin{pmatrix}
f_{0}(r)\\
f_{1}(p,s)\\	
f_{2}(p,s)\\	
f_{3}(t)
\end{pmatrix}=2
\begin{pmatrix}
\partial_{r}\Psi_{u}(r)	\\
\partial_{p}\Psi_{\out}(p,s)\\
\alpha\partial_{s}\Psi_{\out}(p,s)\\
\partial_{t}\Psi_{z}(t)
\end{pmatrix}.
\end{align}
Note that $(f_{1},f_{2},f_{3})$ are exactly as before, with the only difference that $(f_1,f_2)$ are now evaluated at $p$ instead of $r$. The only new function is $f_0$, which depends only on $r$. This means that the new column in the Jacobian is orthogonal to all the other columns, with a single non-zero entry given by $\partial_{r}f_{0}|_{r=0}$. An easy expansion of $f_{0}$ to first order together with the definitions of $(p,r,t,s)$ yield
\begin{align}
    \dd \textbf{f}(0,0,0,0) =
    \begin{pmatrix}
    	0 & \frac{1}{\Delta}\left(\mathbb{E}_{P_{u}}u^2\right)^2 & 0 & 0 \\
        \frac{\beta}{\Delta}\left(\mathbb{E}_{Q_{\out}^{0}}v^2\right)^2 & 0 & 0 &\frac{1}{\rho_{z}^{2}}\left(\mathbb{E}_{Q_{\out}^{0}}vx\right)^2 \\
        \frac{\beta\alpha}{\Delta}\left(\mathbb{E}_{Q_{\out}^{0}}vx\right)^2 &  0 & 0 & \frac{\alpha}{\rho_{z}^2}\left(\mathbb{E}_{Q_{\out}^{0}}x^2-\rho_{z}\right)^2 
        \\ 0 & 0 & \left(\mathbb{E}_{P_{z}}z^2\right)^2 & 0
    \end{pmatrix}.
\end{align}

%%%%%%%%%%%%%%%%%%%%%%%%%%%%%%%%%%%%%%%%%%%%%%%%%%%%%%%%
\subsection*{Transition points for specific activations}
\label{appendix:stability_threshold:wishart}
%%%%%%%%%%%%%%%%%%%%%%%%%%%%%%%%%%%%%%%%%%%%%%%%%%%%%%%%
The transition point $\Delta_{c}$ is defined as the point in which the uninformative point goes from being stable to unstable. The stability is determined in terms of the eigenvalues of the Jacobian: a fixed point is stable when the eigenvalues are smaller than one, and is unstable when the leading eigenvalue becomes greater than one.

It is instructive to look at $\Delta_{c}$ in specific cases. We let $P_{u} = P_{z} = \mathcal{N}(0,1)$ together with $P_{\out}(v|x) = \delta\left(v-\varphi(x)\right)$ and look at different (odd) activation functions $\varphi$.

\paragraph{Linear activation:} Let $\varphi(x)=x$. In this case the transition is $\Delta_{c} = \alpha+1$ in the Wigner model ($\bv\bv^{\intercal}$) and $\Delta_{c} = \sqrt{\beta(\alpha+1)}$ in the Wishart model ($\bu\bv^{\intercal}$)

\paragraph{Sign activation:} Let $\varphi(x)=\sgn(x)$. In this case the transition is $\Delta_{c} = 1+\frac{4}{\pi^2}\alpha$ in the Wigner model ($\bv\bv^{\intercal}$) and $\Delta_{c} = \sqrt{\beta\left(1+\frac{4}{\pi^2}\alpha\right)}$ in the Wishart model ($\bu\bv^{\intercal}$).

%% file: files/supplementary/transition_rmt.tex
In this section, we describe how we can derive the value $\Delta_c$ at which
a transition appears in the recovery for a linear activation function, for both the symmetric $\bv \bv^\intercal$ and non-symmetric $\bu \bv^\intercal$ case, purely from a
random matrix theory analysis.
This transition is in essence similar to the celebrated Baik-Ben Arous-P\'ech\'e (BBP)
transition of the largest eigenvalue of a spiked Wishart (or Wigner) matrix \cite{baik2005phase}.

\subsection{A reminder on the Stieltjes transform}

Let $\bbC_+ = \{z \in \bbC, \, \mathrm{Im}\, z > 0\}$. For any probability measure $\nu$ on $\bbR$, and any $z \in \bbC \backslash \mathrm{supp}\, \nu$,
we can define the Stieltjes transform of $\nu$ as:
\begin{align*}
g_\nu(z) \equiv \EE_{\nu} \frac{1}{X-z}.
\end{align*}
Note that $g_{\nu}(z)$ is a one-to-one mapping of $\bbC_+$ on itself. 
The Stieltjes transform has proven to be a very useful tool
from random matrix theory. One of its important features, that we will use to compute the bulk density (see Fig.~(\ref{main:bbp_lamp_amp_se}) of the main material)
is the Stieltjes-Perron inversion formula, that we state here (see Theorem~X.6.1 of \cite{dunford1967linear}):
\begin{theorem}[Stieltjes-Perron]\label{thm_app:stieltjes}
	Assume that $\nu$ has a continuous density on $\bbR$ with respect to the Lebesgue measure.
	Then:
	\begin{align*}
		\forall x \in \bbR, \quad \frac{\mathrm{d}\nu}{\mathrm{d}x} = \lim_{\epsilon \to 0^+} \frac{1}{\pi} \mathrm{Im} \, g_{\nu}(x + i\epsilon).
	\end{align*}
\end{theorem}
Informally, one has to think that the knowledge of the Stieltjes transform above the real line uniquely determines the measure $\nu$.
The Stieltjes transform is particulaly useful in random matrix theory. Consider a (random) symmetric matrix $M$ of size $n$, with real eigenvalues $\{\lambda_i\}$.
Then the empirical spectral measure of $M$ is defined as:
\begin{align}\label{eq_app:empirical_spectral_measure}
	\nu_n &\equiv \frac{1}{n} \sum_{i=1}^n \delta_{\lambda_i}.
\end{align}
For some random matrix ensembles, the (random) probability measure $\nu_n$ will converge almost surely and in the weak sense to a deterministic probability measure 
$\nu$ as $n \to \infty$. In this case, we will call $\nu$ the \emph{asymptotic spectral measure} of $M$.

%%%%%%%%%%%%%%%%%%%%%%%%%%%%%%%%%%%%%%%%%%%%%%%%%%%%%%%%%%%%%%%
\subsection{RMT analysis of the LAMP operator}
%%%%%%%%%%%%%%%%%%%%%%%%%%%%%%%%%%%%%%%%%%%%%%%%%%%%%%%%%%%%%%%

%%%%%%%%%%%%%%%%%%%%%%%%%%%%%%%%%%%%%%%%%%%%%%%%%%%%%%%%%%%%%%%
\subsubsection{The symmetric $\bv \bv^\intercal$ linear case}\label{subsubsec_app:rmt_vv}
%%%%%%%%%%%%%%%%%%%%%%%%%%%%%%%%%%%%%%%%%%%%%%%%%%%%%%%%%%%%%%%
In this setting, the stationary AMP equations can be reduced on the vector $\hat{\bv}$ as:
\begin{align}
	\hat{\bv} = \left[\frac{1}{k} W W^\intercal \right] \, \left[\frac{1}{\sqrt{\Delta p}} \xi + \frac{1}{\Delta} \frac{\bv \bv^\intercal}{p} - \frac{1}{\Delta}\rI_p\right] \, \hat{\bv}.
\end{align}
We assume in the following that $\rho_v = 1$ to simplify the analysis (in this linear problem, it does not imply any loss of generality).
Here $\xi / \sqrt{p}$ is a matrix from the Gaussian Orthogonal Ensemble, i.e. $\xi$ is a real symmetric matrix with entries drawn independently from a Gaussian distribution with zero mean and variance
$\EE \, \xi_{ij}^2 = (1+\delta_{ij})$. We denote:
\begin{align}\label{eq_app:gammap_vv}
	\Gamma^{vv}_p &\equiv \left[\frac{1}{k} W W^\intercal \right] \, \left[\frac{1}{\sqrt{\Delta p}} \xi + \frac{1}{\Delta} \frac{\bv \bv^\intercal}{p} - \frac{1}{\Delta}\rI_p\right].
\end{align}

From the state evolution analysis we expect that the eigenvector of $\Gamma_p^{vv}$ associated to its largest eigenvalue has a non-zero overlap with $\bv$ in the large $p$ limit
as soon as $\Delta < \Delta_c(\alpha) \equiv 1+ \alpha$.
In this section, we show this fact using only random matrix theory.

Informally, we first demonstrate that the supremum of the support of the asymptotic spectral measure of $\Gamma^{vv}_p$
touches $1$ exactly for $\Delta = \Delta_c(\alpha)$. 
Then, for $\Delta \leq \Delta_c(\alpha)$, the largest eigenvalue of $\Gamma_p^{vv}$ will converge to $1$, 
which is separated from the bulk of the asymptotic spectral density.
The corresponding eigenvector is also positively correlated with $\bv$. 
This gives more detail to the mechanisms of the transition.
We show first the following characterization of the asymptotic spectral density of $\Gamma^{vv}_p$:
\begin{theorem}\label{thm_app:rmt_vv} 
	For any $\alpha, \Delta > 0$, as $p \to +\infty$, the spectral measure of $\Gamma_p^{vv}$ converges almost surely and in the weak sense 
	to a well-defined and compactly supported probability measure $\mu(\alpha, \Delta)$, and we denote $\mathrm{supp}\, \mu$ its support.	
	We separate two cases:
	\begin{itemize}
		\item[$(i)$] If $\Delta \leq \frac{1}{4}$, then $\mathrm{supp}\, \mu \subseteq \bbR_-$.
		\item[$(ii)$] Assume now  $\Delta > \frac{1}{4}$ and denote $z_1(\Delta) \equiv -\Delta^{-1} + 2 \Delta^{-1/2} > 0$.
		Let $\rho_\Delta$ be the probability measure on $\bbR$ with density
		\begin{align}\label{eq_app:def_rho_Delta}
			\rho_\Delta(\mathrm{d}t) &= \frac{\sqrt{\Delta}}{2\pi} \sqrt{4 - \Delta \left(t+\frac{1}{\Delta} \right)^2} \mathds{1}\left\{\left|t+\frac{1}{\Delta}\right| \leq \frac{2}{\sqrt{\Delta}}\right\} \, \mathrm{d}t.
		\end{align}
		Note that the supremum of the support of $\rho_\Delta$ is $z_1(\Delta)$.
		The following equation admits a unique solution for $s \in (-z_1(\Delta)^{-1},0)$:
		\begin{align}\label{eq_app:se_vv}
			\alpha \int \rho_\Delta(\mathrm{d}t) \left(\frac{st}{1+ st}\right)^2 &= 1.
		\end{align}
		We denote this solution as $s_{\rm edge}(\alpha, \Delta)$ (or simply $s_{\rm edge}$).
		The supremum of the support of $\mu(\alpha,\Delta)$ is denoted $\lambda_{\rm max}(\alpha, \Delta)$ (or simply $\lambda_{\rm max}$).
		It is given by:
		\begin{align}
			\lambda_{\rm max} &=
			\begin{dcases}
				-\frac{1}{s_{\rm edge}} + \alpha \int \rho_\Delta(\mathrm{d}t) \frac{t}{1 + s_{\rm edge}t} \quad & \text{ if } \alpha \leq 1,\\
				\max\left(0,-\frac{1}{s_{\rm edge}} + \alpha \int \rho_\Delta(\mathrm{d}t) \frac{t}{1 + s_{\rm edge}t} \right)\quad & \text{ if } \alpha > 1.
			\end{dcases}
		\end{align}
	\end{itemize}
\end{theorem}

Before proving Theorem~\ref{thm_app:rmt_vv}, we state a very interesting corollary:
\begin{corollary}\label{corollary_app:lambdamax_vv}
	Let $\alpha > 0$. As a function of $\Delta$, $\lambda_{\rm max}$ (see Theorem~\ref{thm_app:rmt_vv}) has a unique global maximum, reached exactly at the point $\Delta = \Delta_c(\alpha) = 1 + \alpha$.
	Moreover, $\lambda_{\rm max}(\alpha, \Delta_c(\alpha)) = 1$.
\end{corollary}

We can then state the transition result. Its method of proof is very much inspired by \cite{benaych2011eigenvalues}
\footnote{Note that while all the calculations are justified, refinements would be needed in order to be completely rigorous.
These refinements would follow exactly some proofs of \cite{silverstein1995empirical} and \cite{benaych2011eigenvalues}, 
so we will refer to them when necessary.}.
\begin{theorem}\label{thm_app:transition_vv}
	Let $\alpha, \Delta > 0$.
	Let us denote $\lambda_1 \geq \lambda_2$ the first and second eigenvalues of $\Gamma_p^{vv}$.
	Then we have:
	\begin{itemize}
		\item If $\Delta \geq \Delta_c(\alpha)$, then as $p \to \infty$ we have $\lambda_1 \underset{a.s.}{\to} \lambda_{\rm max}$ and $\lambda_2 \underset{a.s.}{\to} \lambda_{\rm max}$.
		\item If $\Delta \leq \Delta_c(\alpha)$, then as $p \to \infty$ we have $\lambda_1 \underset{a.s.}{\to} 1$ and $\lambda_2 \underset{a.s.}{\to} \lambda_{\rm max}$.
	\end{itemize}
	Moreover, let us denote $\tilde{\bv}$ an eigenvector of $\Gamma_p^{vv}$ with eigenvalue $\lambda_1$, normalized 
	such that $\norm{\tilde{\bv}}^2 = p$. Then:
	\begin{align}
		\frac{1}{p^2}|\tilde{\bv}^\intercal \bv|^2 \underset{a.s.}{\to} \epsilon(\Delta).
	\end{align}
		The function $\epsilon(\Delta)$ satisfies the 
		following properties: $\epsilon(\Delta) = 0$ for all $\Delta \geq \Delta_c(\alpha)$,
		$\epsilon(\Delta) > 0$ for all $\Delta < \Delta_c(\alpha)$ and $\lim_{\Delta \to 0}\epsilon(\Delta) =1$.
\end{theorem}
Our method of proof for Theorem~\ref{thm_app:transition_vv} allows us to 
compute numerically the squared correlation $\epsilon(\Delta)$. 
It is given, for all $\Delta < \Delta_c(\alpha)$, as
\begin{align*}
	\epsilon(\Delta) &= \frac{1}{\alpha} \frac{\left[S^{(2)}(1)\right]^2}{S^{(1,2)}(1)}.
\end{align*}
The $S^{(1,2)}$ and $S^{(2)}$ functions are defined in Lemma~\ref{lemma_app:hierarchy_Sk}, and formulas are also given that allow to compute them numerically.
A non-trivial consistency check is to verify that $\epsilon(\Delta)$ coincides with the variable $q_v$ given by the mutual information
analysis of Theorem~\ref{theorem_uu} of the main material. We show numerically that they indeed coincide in Fig.~\ref{fig_app:overlap_SE_vs_RMT}.
\begin{figure}[hbt]
	\centering
	\includegraphics[scale=0.39]{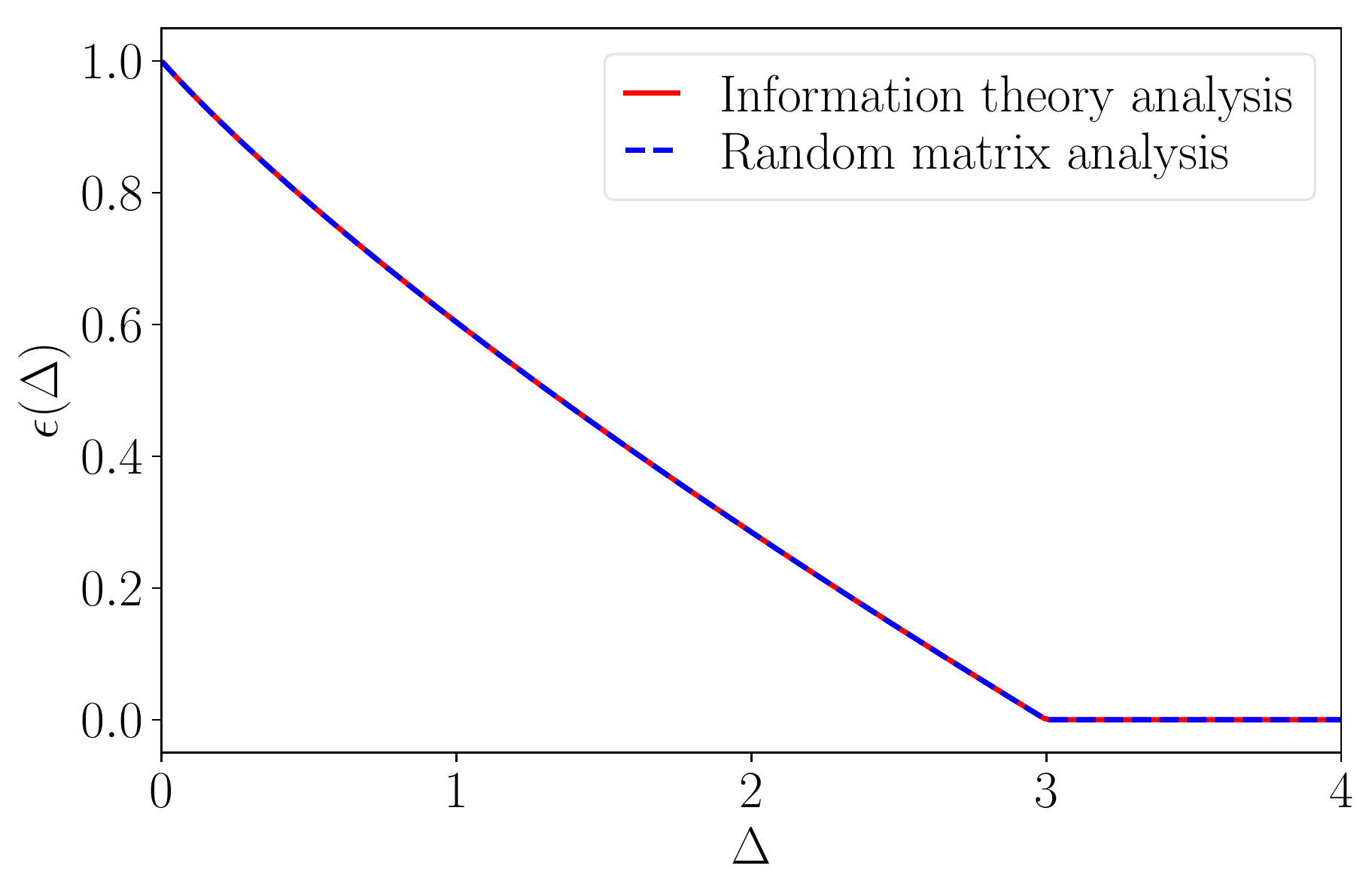}
	\caption{The function $\epsilon(\Delta)$ computed in the linear case by Theorem~\ref{theorem_uu} of the main material (information theoretic analysis) and Theorem~\ref{thm_app:transition_vv} (random matrix analysis) ($\alpha = 2$).}
	\label{fig_app:overlap_SE_vs_RMT}
\end{figure}

\begin{remark*}[The nature of the transition]
	As was already noticed in some previous works (see for instance a related remark in \cite{benaych2011eigenvalues}), the existence of a transition in the largest eigenvalue and the corresponding eigenvector
	for a large matrix of the type $M + \theta P$ (with $P$ of finite rank and $\theta > 0$) depends on the decay of the asymptotic spectral density of $M$ at the right edge of its bulk. For a power-law decay, there can 
	be either no transition, a transition in the largest eigenvalue and the corresponding eigenvector, or a transition in the largest eigenvalue but not in the corresponding eigenvector.
	The situation in our setting is somewhat more involved, as both the bulk and the spike depend on the parameter $\Delta$, and they are not independent (they are correlated via the matrix $W$).
	However, this intuition remains true: if we do not show and use it explicitely, the decay of the density of $\mu(\alpha,\Delta)$ at the right edge is of the type $(\lambda_{\rm max} - \lambda)^{1/2}$, which 
	is the hidden feature that is responsible for a transition both in the largest eigenvalue and the corresponding eigenvector, which is what we show in Theorem~\ref{thm_app:rmt_vv}.
\end{remark*}

\subsubsection{The non-symmetric $\bu \bv^\intercal$ linear case}\label{subsubsec_app:rmt_uv}

The analysis is very similar to the one of the symmetric case of Section~\ref{subsubsec_app:rmt_vv}.
The counterpart to the matrix of eq.~(\ref{eq_app:gammap_vv}) is here:
\begin{align}\label{eq_app:def_gammap_uv}
	\Gamma_p^{uv} &\equiv \frac{1}{\Delta}   \frac{\du{W} \du{W}^\intercal}{k}  \times \( \frac{1}{1+\Delta}\frac{\du{y}^\intercal\du{y}}{p}  -  \beta \,\rI_p  \) \in \bbR^{p \times p}.
\end{align}
Recall that we have here $\alpha = \frac{p}{k}$ and $\beta = \frac{n}{p}$. $W \in \bbR^{p \times k}$ is an i.i.d. standard Gaussian matrix, and the matrix $y \in \bbR^{n \times p}$ is constructed as:
\begin{align}
	y = \sqrt{\Delta} \xi + \frac{\bu \bv^\intercal}{\sqrt{p}}.
\end{align}
Here, $\xi \in \bbR^{n \times p}$ is also an i.i.d. standard Gaussian matrix, independent of $W$. 
As it will be useful for stating the theorem, we recall the Marchenko-Pastur probability measure with ratio $\beta$, denoted $\rho_{{\rm MP},\beta}$ \cite{marvcenko1967distribution}:
	\begin{subnumcases}{\label{eq:def_MP}}
	\lambda_+(\beta) = \left(1 + \frac{1}{\sqrt{\beta}}\right)^2, &\\
	\lambda_-(\beta) = \left(1 - \frac{1}{\sqrt{\beta}}\right)^2, &\\
	\frac{\mathrm{d}\rho_{{\rm MP}, \beta}}{\mathrm{d}t} \equiv (1-\beta)\, \delta(t) + \frac{\beta}{2 \pi} \frac{\sqrt{\left[\lambda_+(\beta)-t\right] \left[t-\lambda_-(\beta)\right]}}{t} \mathds{1}_{t \in (\lambda_-(\beta) , \lambda_+(\beta))}. &
	\end{subnumcases}
We can now state the couterpart to Theorem~\ref{thm_app:rmt_vv} in the $\bu \bv^\intercal$ setting:
\begin{theorem}\label{thm_app:rmt_uv} 
	For any $\alpha, \beta, \Delta > 0$, the spectral measure of $\Gamma_p^{uv}$ converges almost surely and in the weak sense 
	to a well-defined and compactly supported measure $\mu(\Delta,\alpha,\beta)$. We denote $\mathrm{supp} \, \mu$ its support.
	We introduce a function $z_1$ and a probability measure $\rho_{\beta,\Delta}$ as follows:
		\begin{align}
			z_1(\beta,\Delta) &\equiv \frac{-\beta+\Delta+2 \Delta \sqrt{\beta}}{\Delta(1+\Delta)},  \nonumber \\
			\label{eq_app:def_rho_beta_Delta}
			\frac{\mathrm{d} \rho_{\beta,\Delta}}{\mathrm{d}t} &\equiv \frac{1+\Delta}{\beta} \frac{\mathrm{d}\rho_{{\rm MP}, \beta}}{\mathrm{d}t} \left(\frac{1+\Delta}{\beta} t + \frac{1+\Delta}{\Delta}\right).
		\end{align}
		Note that $z_1(\beta,\Delta)$ is the supremum of the support of $\rho_{\beta,\Delta}$. Let finally
	\begin{align}\label{eq:Delta_pos}
		\Delta_{\rm pos}(\beta) \equiv \frac{\beta}{1 + 2\sqrt{\beta}}.	
	\end{align}
	We separate two cases:
	\begin{itemize}
		\item[$(i)$] If $\Delta \leq \Delta_{\rm pos}(\beta)$, then $z_1(\beta,\Delta) \leq 0$ and $\mathrm{supp}\, \mu \subseteq \bbR_-$.
		\item[$(ii)$] Assume now $\Delta > \Delta_{\rm pos}(\beta)$. Then $z_1(\beta,\Delta) > 0$. The following equation admits a unique solution
		for $s \in (-z_1(\beta,\Delta)^{-1},0)$:
		\begin{align}\label{eq_app:se_uv}
			\alpha \int \rho_{\beta,\Delta}(\mathrm{d}t) \left(\frac{st}{1+st}\right)^2 &= 1.
		\end{align}
		We denote this solution as $s_{\rm edge}(\alpha, \beta, \Delta)$ (or simply $s_{\rm edge}$). We denote $\lambda_{\rm max}(\alpha, \beta, \Delta)$ (or only $\lambda_{\rm max}$) the supremum of the support of $\mu(\Delta, \alpha, \beta)$.
		Then we have:
		\begin{align}
			\lambda_{\rm max} &= 
			\begin{dcases}
				-\frac{1}{s_{\rm edge}} + \alpha \int \rho_{\beta, \Delta}(\mathrm{d}t) \frac{t}{1 + s_{\rm edge}t} \quad & \text{ if } \alpha \leq 1,\\
				\max\left(0,-\frac{1}{s_{\rm edge}} + \alpha \int \rho_{\beta, \Delta}(\mathrm{d}t) \frac{t}{1 + s_{\rm edge}t} \right)\quad & \text{ if } \alpha > 1.
			\end{dcases}
		\end{align}
	\end{itemize}
\end{theorem}

We can state the corresponding corollary to this theorem:

\begin{corollary}\label{corollary_app:lambdamax_uv} 
	Let $\alpha, \beta > 0$. Seen as a function of $\Delta$, $\lambda_{\rm max}$ (see Theorem~\ref{thm_app:rmt_vv}) has a unique global maximum, attained exactly at the point $\Delta_c(\alpha,\beta) \equiv \sqrt{\beta(1 + \alpha)}$.
	Moreover, 
	\begin{align*}
		\lambda_{\rm max}(\alpha, \beta, \Delta_c(\alpha, \beta)) &= 1.
	\end{align*} 
\end{corollary}

We can then describe the complete transition.
Proving this transition would follow the same main lines as the proof of the transition in the $\bv \bv^\intercal$ case (Theorem~\ref{thm_app:transition_vv}), but would be significantly heavier.
This is left for future work, so we state the transition in this setting as a conjecture:

\begin{conjecture}\label{conjecture_app:transition_uv}
	Let $\alpha, \beta, \Delta> 0$.
	Let us denote $\lambda_1 \geq \lambda_2$ the first and second eigenvalues of $\Gamma_p^{uv}$.
	Then we have:
	\begin{itemize}
		\item If $\Delta \geq \Delta_c(\alpha, \beta)$, then as $p \to \infty$ we have $\lambda_1 \underset{a.s.}{\to} \lambda_{\rm max}$ and $\lambda_2 \underset{a.s.}{\to} \lambda_{\rm max}$.
		\item If $\Delta \leq \Delta_c(\alpha, \beta)$, then as $p \to \infty$ we have $\lambda_1 \underset{a.s.}{\to} 1$ and $\lambda_2 \underset{a.s.}{\to} \lambda_{\rm max}$.
	\end{itemize}
	Let us denote $\tilde{\bv}$ an eigenvector of $\Gamma_p^{uv}$ with eigenvalue $\lambda_1$, normalized 
	such that $\norm{\tilde{\bv}}^2 = p$. Then:
	\begin{align}
		\frac{1}{p^2}|\tilde{\bv}^\intercal \bv|^2 \underset{a.s.}{\to} \epsilon(\Delta).
	\end{align}
		It satisfies $\epsilon(\Delta) = 0$ for all $\Delta \geq \Delta_c(\alpha, \beta)$,
		$\epsilon(\Delta) > 0$ for all $\Delta < \Delta_c(\alpha, \beta)$ and $\lim_{\Delta \to 0}\epsilon(\Delta) =1$.
\end{conjecture}

%%%%%%%%%%%%%%%%%%%%%%%%%%%%%%%%%%%%%%%%%%%%%%%%%%%%%%%%
\subsection{Proofs}
%%%%%%%%%%%%%%%%%%%%%%%%%%%%%%%%%%%%%%%%%%%%%%%%%%%%%%%%

%%%%%%%%%%%%%%%%%%%%%%%%%%%%%%%%%%%%%%%%%%%%%%%%%%%%%%%%
\subsubsection{Proof of Theorem~\ref{thm_app:rmt_vv} and Corollary~\ref{corollary_app:lambdamax_vv}}
%%%%%%%%%%%%%%%%%%%%%%%%%%%%%%%%%%%%%%%%%%%%%%%%%%%%%%%%
\paragraph{Proof of Theorem~\ref{thm_app:rmt_vv}}
\begin{proof}[Proof of Theorem~\ref{thm_app:rmt_vv} $(ii)$]
	We begin by treating the more involved case $(ii)$, that is we assume $\Delta > \frac{1}{4}$.
 Note first that by basic linear algebra, the spectrum of $\Gamma_p^{vv}$
is, up to $0$ eigenvalues, the same as the spectrum of the following matrix $\Gamma_k^{vv}$:
\begin{align}
	\Gamma_k^{vv} &\equiv \frac{1}{k} W^\intercal \left[\frac{1}{\sqrt{\Delta p}} \xi + \frac{1}{\Delta} \frac{\bv \bv^\intercal}{p} - \frac{1}{\Delta}\rI_p\right] W \in \bbR^{k \times k},
\end{align}
More precisely, if $p \geq k$ (so $\alpha \geq 1$) we have $\mathrm{Sp}\, (\Gamma_p^{vv}) = \mathrm{Sp}\, (\Gamma_k^{vv}) \cup\{0\}^{p-k}$,
and conversely if $k > p$. These additional zero eigenvalues in the case $\alpha > 1$ explain the $\max(0,\cdot)$ term in the conclusion of Theorem~\ref{thm_app:rmt_vv}.

For the remainder of the proof we can thus consider $\Gamma^{vv}_k$ instead of $\Gamma_p^{vv}$ given the remark above.
Moreover, for simplicity we will drop the $vv$ exponent in those matrices, and just denote them $\Gamma_k, \Gamma_p$.
The bulk of $\Gamma_k$ can be studied using standard random matrix theory results. Such matrices were first studied by Marchenko and Pastur in a seminal work \cite{marvcenko1967distribution},
which was generalized (and made rigorous) later in \cite{silverstein1995empirical}.
Note finally that by the celebrated results of Wigner \cite{wigner1993characteristic}, the spectral distribution of the matrix
$\xi/\sqrt{\Delta p} -\rI_{p}/\Delta$ converges in law (and almost surely) as $p \to \infty$ to $\rho_\Delta$, given by eq.~(\ref{eq_app:def_rho_Delta}).
We can then use Theorem~1.1 of \cite{silverstein1995empirical}, that we recall here for our setting:
\begin{theorem}[Silverstein-Bai]\label{thm_app:silverstein}
	Let $p,k \to \infty$ with $p/k \to \alpha > 0$. Let $W \in \bbR^{p \times k}$ be an i.i.d. Gaussian matrix, whose elements come from
	the standard Gaussian distribution $\mathcal{N}(0,1)$. Let $T_p \in \bbR^{p \times p}$ be a random symmetric matrix, independent of $W$, such that the empirical spectral distribution
	of $T_p$ converges (almost surely) in law to a measure $\rho_T$. Then, almost surely, the empirical spectral distribution of $B_k \equiv \frac{1}{k} W^\intercal T_p W$ converges
	in law to a (nonrandom) measure $\mu_B$, whose Stieltjes transform satisfies, for every $z \in \bbC_+$:
	\begin{align}\label{eq_app:silverstein_eq_thm}
		g_{\mu_B}(z) &= -\left[z - \alpha \int \nu_T(\mathrm{d}t) \frac{t}{1+ t g_{\mu_B}(z)}\right]^{-1}.
	\end{align}
	Moreover, for every $z \in \bbC_+$, there is a unique solution to eq.~(\ref{eq_app:silverstein_eq_thm}) such that  $g_{\mu_B}(z) \in \bbC_+$.
	This equation thus characterizes unambiguously the measure $\mu_B$.
\end{theorem}

Applying Theorem~\ref{thm_app:silverstein} to our setting shows that we can define
$\nu(\alpha, \Delta)$ as the limit eigenvalue distribution of $\Gamma_k$, and we denote $g_\nu(z)$ its Stieltjes transform.
From the remarks above, $\mu(\alpha, \Delta)$ and $\nu(\alpha, \Delta)$ only differ by the addition of a delta distribution. For instance, if $\alpha \geq 1$:
\begin{align}
	\mu(\alpha, \Delta) &= \alpha \nu(\alpha, \Delta) + (1-\alpha) \delta_0.	
\end{align}
The main quantity of interest to us is $z_{\rm edge}$, defined as the supremum of the support of $\nu(\alpha, \Delta)$. If $z_{\rm edge} \geq 0$, then it will also be the supremum 
of the support of $\mu(\alpha, \Delta)$, and thus equal to $\lambda_{\rm max}$.
Theorem~\ref{thm_app:silverstein} shows that for every $z \in \bbC_+ \cup (\bbR \backslash \mathrm{supp} \, \nu)$, $g_\nu(z)$ is the only solution in $\bbC_+ \cup \bbR$ to the following equation:
\begin{align}\label{eq_app:silverstein}
	g_\nu(z) &= -\left[z - \alpha \int \rho_\Delta(\mathrm{d}t) \frac{t}{1+ t g_\nu(z)}\right]^{-1}.
\end{align}
The validity of the equation for $\bbR \backslash \mathrm{supp} \, \nu$ (and not only on $\bbC_+$) follows from the continuity of $g_\nu(z)$ on $\bbC_+ \cup (\bbR \backslash \mathrm{supp} \, \nu)$, a generic property of the Stieltjes transform.
It is easy to see that $g_\nu$ induces a strictly increasing
diffeomorphism $g_\nu : (z_{\rm edge},+\infty) \to (\lim_{z_\to z_{\rm edge}^+} g_\nu(z),0)$, so that we can define its inverse $g_\nu^{-1}$ and from eq.~(\ref{eq_app:silverstein}), it satisfies
for every $s \in (\lim_{z_\to z_{\rm edge}^+} g_\mu(z),0)$:
\begin{align}\label{eq_app:silverstein_inverse}
	g^{-1}_{\nu}(s) &= -\frac{1}{s} + \alpha \int \rho_\Delta(\mathrm{d}t) \frac{t}{1 + st}.
\end{align}

\paragraph{Remark} Note that this can be written in terms of the ${\cal R}$-transform of $\nu$ (an useful tool of free probability):
\begin{align*}
	{\cal R}_\nu(s) &\equiv g_\nu^{-1}(-s) - \frac{1}{s} = \alpha \int \rho_\Delta(\mathrm{d}t) \frac{t}{1 - st}. 
\end{align*}

 In order to compute $z_{\rm edge}$ from eq.~(\ref{eq_app:silverstein}), we use a result of Section~4 of
 \cite{silverstein1995empirical}, also stated for instance in \cite{lee2016tracy}, that describes the form of the support of $\nu(\alpha, \Delta)$.
 It can be stated in the following way. Recall that since $\Delta > \frac{1}{4}$, $z_{1}(\Delta) > 0$ is the maximum of the support of $\rho_\Delta$.
 Let $s_{\rm edge}$ be the unique solution in $(-z_1(\Delta)^{-1},0)$ of
 the equation $(g_\nu^{-1})'(s) = 0$, that is by eq.~(\ref{eq_app:silverstein_inverse}):
 \begin{align}\label{eq_app:sedge_vv}
	\alpha \int \rho_\Delta(\mathrm{d}t) \left(\frac{st}{1 + st}\right)^2 &= 1.
 \end{align}
 Indeed, it is straighforward to show that the left-hand side of eq.~(\ref{eq_app:sedge_vv})
 tends to $0$ as $s \to 0^-$, tends to $+\infty$ as $s \to -z_1(\Delta)^{-1}$, and is a strictly decreasing and continuous function of $s$.
 Then (see for instance eq.~(2.13) and eq.~(2.14) of \cite{lee2016tracy}) $z_{\rm edge}$ is given by
 \begin{align}
	z_{\rm edge} &= \lim_{s \to s_{\rm edge}^+} g_{\nu}^{-1}(s), \nonumber \\
	&=  -\frac{1}{s_{\rm edge}} + \alpha \int \rho_\Delta(\mathrm{d}t) \frac{t}{1 + s_{\rm edge}t}.
 \end{align}
 This ends the proof of $(ii)$. 
\end{proof}
Let us make a final remark that will be useful in our future analysis. Note that $z_1(\Delta) > 1$ for all $\Delta > 1$.
Moreover, for all $\Delta > 1$, we have by an explicit computation:
\begin{align*}
	\alpha \int \rho_\Delta(\mathrm{d}t) \left(\frac{t}{1 - t}\right)^2 &= \frac{\alpha}{\Delta - 1}.
\end{align*}
By the argument above, this yields the following result, that we state as a lemma:
\begin{lemma}\label{lemma_app:sedge_1}
	Assume $\Delta > 1$. Then:
	\begin{itemize}
		\item[$(i)$] If $\Delta < \Delta_c(\alpha)$, then $s_{\rm edge} > -1$.
		\item[$(ii)$] If $\Delta = \Delta_c(\alpha)$, then $s_{\rm edge} = -1$.
		\item[$(iii)$] If $\Delta > \Delta_c(\alpha)$, then $s_{\rm edge} < -1$.
	\end{itemize}	
\end{lemma}
\begin{proof}[Proof of Theorem~\ref{thm_app:rmt_vv}, $(i)$]
	Assume now $\Delta \leq \frac{1}{4}$. Then the support of $\rho_\Delta$ is a subset of $\bbR_-$.
	Since $0 \in \bbR_-$, we can use again the remark we made in the proof of $(ii)$ to study $\Gamma_k$ instead of
	$\Gamma_p$. Moreover, Theorem~\ref{thm_app:silverstein} still applies here so that we have the Silverstein equation~(\ref{eq_app:silverstein_inverse})
	 for every $s \in \bbC_+$:
	 \begin{align*}
		g^{-1}_{\nu}(s) &= -\frac{1}{s} + \alpha \int \rho_\Delta(\mathrm{d}t) \frac{t}{1 + st}.
	 \end{align*}
	 By the Stieltjes-Perron inversion Theorem~\ref{thm_app:stieltjes}, it is enough to check that for every $z > 0$, there exists a unique $s < 0$ such that
	 $g_\nu^{-1}(s) = z$. Indeed, this will yield $s = g_\nu(z) \in \bbR$. In particular, $\lim_{\epsilon \to 0^+} \mathrm{Im}\, g_\nu(z + i \epsilon) = 0$ for every $z > 0$, which will imply
	  $\mathrm{supp}(\nu) \subseteq \bbR_-$ and thus $\mathrm{supp}(\mu) \subseteq \bbR_-$.

	 Therefore, let $z > 0$. From eq.~(\ref{eq_app:silverstein_inverse}) and the fact that $\mathrm{supp}(\rho_\Delta)\subseteq \bbR_-$, we easily obtain:
	 \begin{align*}
		\lim_{s \to - \infty} g_{\nu}^{-1}(s) &= 0, \\
		\lim_{s \to 0^-} g_{\nu}^{-1}(s) &= + \infty.
	 \end{align*}
	 Moreover, $g_{\nu}^{-1}(s)$ is a strictly increasing  continuous function of $s$, so that the existence and unicity of $s = g_{\nu}(z) < 0$ is
	 immediate, which ends the proof.
\end{proof}

\paragraph{Proof of Corollary~\ref{corollary_app:lambdamax_vv}}
\begin{proof}
	Let us make a few remarks:
	\begin{itemize}
		\item By Theorem~\ref{thm_app:rmt_vv}, we know that if $\Delta \leq \frac{1}{4}$, then $\lambda_{\rm max} \leq 0$.
		\item It is trivial by the form of $\Gamma_p$ that, as $\Delta \to + \infty$, $\lambda_{\rm max} \to 0$.
	\end{itemize}
	Let $z_{\rm edge} = -\frac{1}{s_{\rm edge}} + \alpha \int \rho_\Delta(\mathrm{d}t) \frac{t}{1 + s_{\rm edge}t}$.
	 Then we know that $\lambda_{\rm max} = z_{\rm edge}$ if $\alpha \leq 1$ and
	$\lambda_{\rm max} = \max(0,z_{\rm edge})$ if $\alpha > 1$. In particular, by the remark above, $z_{\rm edge} \leq 0$ for $\Delta = \frac{1}{4}$ and $z_{\rm edge} \to 0^+$ as $\Delta \to \infty$.
	It is easy to see that $z_{\rm edge}$ is a continuous and differentiable function of $\Delta$, so that if we show the two following facts for any $\Delta \geq \frac{1}{4}$:
	\begin{align}
		\label{eq_app:dzedge_delta}
		\frac{\mathrm{d} z_{\rm edge}}{\mathrm{d} \Delta} &= 0 \Leftrightarrow \Delta = \Delta_c(\alpha) = 1+\alpha, \\
		\label{eq_app:zedge}
		z_{\rm edge}(\Delta_c(\alpha)) &= 1,
	\end{align}
	this would end the proof as $z_{\rm edge}$ would necessarily have a unique global maximum, located in $\Delta = \Delta_c(\alpha)$, in which we have
	$\lambda_{\rm max} = 1$. We thus prove eq.~(\ref{eq_app:dzedge_delta}) and eq.~(\ref{eq_app:zedge}) in the following.

	\textbf{Proof of eq.~(\ref{eq_app:dzedge_delta}) \hspace{0.5cm}} By the chain rule:
	\begin{align*}
		\frac{\mathrm{d} z_{\rm edge}}{\mathrm{d} \Delta} &= \frac{\partial z_{\rm edge}}{\partial \Delta} + \frac{\partial s_{\rm edge}}{\partial \Delta} \frac{\partial z_{\rm edge}}{\partial s_{\rm edge}}, \\
			&= \frac{\partial z_{\rm edge}}{\partial \Delta},
	\end{align*}
	using the very definition of $s_{\rm edge}$, eq.~(\ref{eq_app:sedge_vv}), as $z_{\rm edge} = g_\nu^{-1}(s_{\rm edge})$.
	Given the explicit form of $\rho_\Delta$, one can compute easily:
	\begin{align*}
			\frac{\partial z_{\rm edge}}{\partial \Delta} &= - \alpha \frac{s_{\rm edge} + 2 s_{\rm edge}^2 - \Delta + \sqrt{s_{\rm edge}^2 - 2 s_{\rm edge}(1+2s_{\rm edge}) \Delta + \Delta^2}}{2 s_{\rm edge}^3 \sqrt{s_{\rm edge}^2 - 2 s_{\rm edge}(1+2s_{\rm edge}) \Delta + \Delta^2}}.
	\end{align*}
	It is then simple analysis to see that since $s_{\rm edge} < 0$, $\frac{\partial z_{\rm edge}}{\partial \Delta} = 0$ is equivalent to $s_{\rm edge} = -1$ and $\Delta > 1$.
	Recall that $s_{\rm edge}$ is originally defined as a solution to eq.~(\ref{eq_app:sedge_vv}):
	\begin{align*}
		\alpha \int \rho_\Delta(\mathrm{d}t) \left(\frac{s_{\rm edge}t}{1 + s_{\rm edge}t}\right)^2 &= 1.
	\end{align*}
	Inserting $s_{\rm edge} = -1$ into this equation and using the explicit form of $\rho_{\Delta}$ given by eq.~(\ref{eq_app:def_rho_Delta}), and using moreover that $\Delta > 1$,
	this reduces to:
	\begin{align*}
		\frac{\alpha}{\Delta - 1} &= 1,
	\end{align*}
	which is equivalent to $\Delta = \Delta_c(\alpha) = 1+\alpha$.

	\textbf{Proof of eq.~(\ref{eq_app:zedge}) \hspace{0.5cm}} By Lemma~\ref{lemma_app:sedge_1}, we know that for $\Delta = \Delta_c(\alpha)$ we have
	$s_{\rm edge} = -1$. Given eq.~(\ref{eq_app:def_rho_Delta}), it is then straightforward to compute:
	\begin{align*}
		z_{\rm edge}(\Delta_c(\alpha)) &= -1 + \alpha \int \rho_{\Delta_c(\alpha)}(\mathrm{d}t) \frac{t}{1-t}, \\
		&= 1.
	\end{align*}
\end{proof}

%%%%%%%%%%%%%%%%%%%%%%%%%%%%%%%%%%%%%%%%%%%%%%%%%%%%%%%%%%%%%%%
\subsubsection{Proof of Theorem~\ref{thm_app:transition_vv}}
%%%%%%%%%%%%%%%%%%%%%%%%%%%%%%%%%%%%%%%%%%%%%%%%%%%%%%%%%%%%%%%

\paragraph{Transition of the largest eigenvalue}

This part is a detailed outline of the proof.
Some parts of the calculation are not fully rigorous, however they can be justified more precisely 
by following exactly
the lines of \cite{benaych2011eigenvalues} and \cite{silverstein1995analysis}. 
We will emphasize when such refinements have to be made.
 Recall that we have by eq.~(\ref{eq_app:gammap_vv}) the following decomposition of $\Gamma_p^{vv}$ (that we denote $\Gamma_p$ for simplicity):
\begin{align}
	\Gamma_p &= \underbrace{\left[\frac{1}{k} W W^\intercal \right] \, \left[\frac{1}{\sqrt{\Delta p}} \xi - \frac{1}{\Delta}\rI_p\right]}_{\Gamma_p^{(0)}} + \underbrace{\frac{1}{\Delta} \frac{WW^\intercal}{k} \frac{\bv \bv^\intercal}{p}}_{\textrm{rank }1 \textrm{ perturbation}}.
\end{align}
Theorem~\ref{thm_app:rmt_vv} and Corollary.~\ref{corollary_app:lambdamax_vv}, along with their respective proofs, already describe in great detail the limit eigenvalue distribution of $\Gamma_p^{(0)}$.
We first note that for any $\lambda \in \bbR$ that is not an eigenvalue of $\Gamma_p^{(0)}$ one can write:
\begin{align*}
	\det \left(\lambda \rI_p - \Gamma_p\right) &= \det \left(\lambda \rI_p - \Gamma_p^{(0)}\right) \det \left(\rI_p  - \left(\lambda \rI_p - \Gamma_p^{(0)}\right)^{-1} \frac{1}{\Delta} \frac{WW^\intercal}{k} \frac{\bv \bv^\intercal}{p}\right).
\end{align*}
In particular, this implies immediately that $\lambda$ is an eigenvalue of $\Gamma_p$ and not 
an eigenvalue of $\Gamma_p^{(0)}$ if and only if $1$ is an eigenvalue of $\left(\lambda \rI_p - \Gamma_p^{(0)}\right)^{-1} \frac{1}{\Delta} \frac{WW^\intercal}{k} \frac{\bv \bv^\intercal}{p}$.
Since this is a rank-one matrix, its only non-zero eigenvalue is equal to its trace, so it is equivalent to:
\begin{align}\label{eq_app:isolated_lambda_vv}
	1 &= \mathrm{Tr} \, \left[\left(\lambda \rI_p - \Gamma_p^{(0)}\right)^{-1} \frac{1}{\Delta} \frac{WW^\intercal}{k} \frac{\bv \bv^\intercal}{p}\right].
\end{align}
Recall that by definition, $\bv$ is constructed as $\bv = W \bz / \sqrt{k}$, with $\bz$ a standard Gaussian i.i.d. vector in $\bbR^k$, independent of $W$.
For any matrix $A$, we have the classical concentration $\frac{1}{k} \bz^\intercal A \bz = \frac{1}{k}\mathrm{Tr} A$ with high probability as $k \to \infty$.
In eq.~(\ref{eq_app:isolated_lambda_vv}), this yields at leading order as $p\to \infty$:
\begin{align}
	\Delta &= \frac{1}{p} \mathrm{Tr} \, \left[\left(\lambda \rI_p - \Gamma_p^{(0)}\right)^{-1} \left(\frac{WW^\intercal}{k}\right)^2 \right].
\end{align}
We will prefer to use $k \times k$ matrices.
We use the simple linear algebra identity, for any $p \times p$ symmetric matrix $A$, and any integer $q \geq 1$:
\begin{align*}
	\mathrm{Tr}\, \left[\left(\lambda \rI_p - \frac{WW^\intercal}{k} A\right)^{-1} \left(\frac{WW^\intercal}{k}\right)^q\right] &= \mathrm{Tr}\, \left[\left(\lambda \rI_k - \frac{1}{k}W^\intercal A W\right)^{-1} \left(\frac{W^\intercal W}{k}\right)^q\right].
\end{align*} 
This can be derived for instance by expanding both sides in powers of $\lambda^{-1}$ and using the cyclicity of the trace.
Finally, we can state that the eigenvalues of $\Gamma_p$ that are outside of the spectrum of $\Gamma_p^{(0)}$ must satisfy, as $k \to \infty$:
\begin{align}\label{eq_app:isolated_lambda_vv_2}
	\alpha \Delta &= \frac{1}{k} \mathrm{Tr} \, \left[\left(\lambda \rI_k - \Gamma_k^{(0)}\right)^{-1} \left(\frac{W^\intercal W}{k}\right)^2 \right],
\end{align}
with 
\begin{align*}
	\Gamma_k^{(0)} &\equiv \frac{1}{k} W^\intercal \left[\frac{1}{\sqrt{\Delta p}} \xi - \frac{1}{\Delta}\rI_p\right] W.
\end{align*}

We will now make use of two important lemmas, at the core of our analysis.
They will also prove to be useful in the eigenvector correlation analysis.
\begin{lemma}\label{lemma_app:hierarchy_Sk}
	Recall that $\nu$ is the limit eigenvalue distribution of $\Gamma_k^{(0)}$, that the supremum of its support
	is $\lambda_{\rm max}$, and its Stieltjes transform is $g_\nu$.
	For every integer $r \geq 0$, we define:
	\begin{align*}
		S^{(r)}_k(\lambda) &\equiv \frac{1}{k} \mathrm{Tr} \, \left[\left(\Gamma_k^{(0)} - \lambda \rI_k\right)^{-1} \left(\frac{W^\intercal W}{k}\right)^r \right].
	\end{align*}
	For $r \in \{0,1,2,3\}$\footnote{The almost sure convergence could probably be extended to all $r \in \bbN^\star$ but we will only use these values of $r$ in the following.} 
	and every $\lambda > \lambda_{\rm max}$, as $k \to \infty$ $S^{(r)}_k(\lambda)$ converges almost surely to a well defined limit $S^{(r)}(\lambda)$. 
	This limit is given by:
	\begin{align}
		\begin{cases}
			S^{(0)}(\lambda) &= g_\nu(\lambda),\\
			S^{(1)}(\lambda) &= g_\nu(\lambda) \left[\alpha - (1+\lambda g_\nu(\lambda))\right], \\
			S^{(2)}(\lambda) &= g_\nu(\lambda) \left[\alpha (1+\alpha) - (1+2\alpha) (1+\lambda g_\nu(\lambda)) + (1+\lambda g_\nu(\lambda))^2\right], \\
			S^{(3)}(\lambda) &= g_\nu(\lambda) \left[(\alpha+3\alpha^2+\alpha^3) - (1+5\alpha + 3\alpha^2) (1+\lambda g_\nu(\lambda)) \right. \\
			&\left. \hspace{1.5cm} + (2+3\alpha)(1+\lambda g_\nu(\lambda))^2 - (1+\lambda g_\nu(\lambda))^3\right].
		\end{cases}
	\end{align}
	We define similarly for every integer $r,q \geq 0$:
	\begin{align*}
		S^{(r,q)}_k(\lambda) &\equiv \frac{1}{k} \mathrm{Tr} \, \left[\left(\Gamma_k^{(0)} - \lambda \rI_k\right)^{-1} \left(\frac{W^\intercal W}{k}\right)^r \left(\Gamma_k^{(0)} - \lambda \rI_k\right)^{-1} \left(\frac{W^\intercal W}{k}\right)^q\right].
	\end{align*}
	Note that $S^{(r,q)}_k = S^{(q,r)}_k$ and that $S^{(r,0)}_k(\lambda) = \partial_z S_k^{(r)}(\lambda)$. 
	For every $\lambda > \lambda_{\rm max}$, $S^{(1,1)}_k(\lambda)$ and $S^{(1,2)}_k(\lambda)$ converge almost surely (as $k \to \infty$) to well-defined limits, that satisfy the following equations:
	\begin{align*}
		&S^{(1,1)}(\lambda) = g_\nu(\lambda) S^{(2)}(\lambda) - \left[1 + \lambda g_\nu(\lambda)\right] \partial_\lambda S^{(1)}(\lambda) \\ 
		&+ \alpha g_\nu(\lambda) \left[g_\nu(\lambda) + S^{(1)}(\lambda)\right] \int \frac{\rho_\Delta(\mathrm{d}t) t}{\left(1+t g_\nu(\lambda)\right)^2} \left[t \, \partial_\lambda S^{(1)}(\lambda) - g_\nu(\lambda)\right], \\
		&S^{(1,2)}(\lambda) = g_\nu(\lambda) S^{(3)}(\lambda) - \left[1 + \lambda g_\nu(\lambda)\right] \left[S^{(1,1)}(\lambda) + (1+\alpha)\partial_\lambda S^{(1)}(\lambda) \right]\\ 
		&+ \alpha g_\nu(\lambda) \left[(1+\alpha) g_\nu(\lambda) + S^{(1)}(\lambda)+ S^{(2)}(\lambda)\right] \int \frac{\rho_\Delta(\mathrm{d}t) t}{\left(1+t g_\nu(\lambda)\right)^2} \left[t \, \partial_\lambda S^{(1)}(\lambda) - g_\nu(\lambda)\right].
	\end{align*}
\end{lemma}
\begin{lemma}\label{lemma_app:evalue_transition_vv}
	Let $\alpha, \Delta > 0$.
	We focus mainly on $S^{(2)}(\lambda)$. We have:
	\begin{itemize}
		\item[$(i)$] For every $r$, $S^{(r)}(\lambda)$ is a strictly increasing function of $\lambda$, and $\lim_{\lambda \to \infty} S^{(r)}(\lambda) = 0$.
		\item[$(ii)$] For every $\lambda > \lambda_{\rm max}$, $S^{(2)}(\lambda) = - \alpha \Delta$ if and only if $\Delta \leq \Delta_c(\alpha)$ and
	$\lambda = 1$.
		\item[$(iii)$] For every $\Delta > \Delta_c(\alpha)$, $\lim_{\lambda \to \lambda_{\rm max}} S^{(2)}(\lambda) \in (-\alpha \Delta, 0)$ (it is well defined by monotonicity of $S^{(2)}(\lambda)$).
	\end{itemize}
\end{lemma}

Let us see how item $(ii)$ of Lemma~\ref{lemma_app:evalue_transition_vv} and eq.~(\ref{eq_app:isolated_lambda_vv_2}) end the proof of the 
eigenvalue transition. First, note that by the celebrated Weyl's interlacing inequalities \cite{weyl1949inequalities}, we have:
\begin{align*}
	\liminf_{p \to \infty} \lambda_1 &\geq \lambda_{\rm max}, \\ 	
	\limsup_{p \to \infty} \lambda_2 &\leq \lambda_{\rm max}.
\end{align*}
This implies that because the perturbation of the matrix is of rank one, \emph{at most one} outlier eigenvalue will exist in the limit $p \to \infty$.
By eq.~(\ref{eq_app:isolated_lambda_vv_2}), this outlier $\lambda_1$ exists if and only if it satisfies, in the large $p \to \infty$ limit, the equation 
$S^{(2)}(\lambda_1) = - \alpha \Delta$. By item $(ii)$ of Lemma~\ref{lemma_app:evalue_transition_vv}, this is the case only for $\lambda_1 = 1$ and $\Delta \leq \Delta_{c}(\alpha)$, 
which ends the proof.
A completely rigorous treatement of these arguments requires to state more precisely concentration results. 
Such a treatment has been made in \cite{benaych2011eigenvalues} in a very close case (from which all the arguments transpose), and we refer to it for more details.
We finally describe the proofs of the lemmas in the following.
\begin{proof}[Proof of Lemma.~\ref{lemma_app:hierarchy_Sk}]
	The essence of the computation originates from the derivation of Theorem~\ref{thm_app:silverstein} in \cite{silverstein1995empirical}.
	Note that $S_k^{(0)}(\lambda)$ converges a.s. to the Stieljtes transform $g_\nu(\lambda)$ as $k \to \infty$ by Theorem~\ref{thm_app:silverstein}.
	For every $1 \leq i \leq p$, $w_i$ denotes the $i$-th row of $W$. We denote $y = \frac{1}{\sqrt{\Delta p}} \xi - \frac{1}{\Delta} \rI_p$.
	Since $W$ is independent of $y$, we can denote $y_1,\cdots,y_p$ the eigenvalues of $y$, and their empirical distribution converges a.s. to $\rho_\Delta$ as we know.
	We have in distribution:
	\begin{align*}
		\Gamma_k^{(0)} &= \frac{1}{k} W^\intercal \, y\, W \overset{d}{=} \frac{\alpha}{p} \sum_{i=1}^p y_i \, w_i \, w_i^\intercal. 
	\end{align*}
	For every $i$, we denote:
	\begin{align*}
		\Gamma_{k,i}^{(0)} &\equiv = \frac{\alpha}{p} \sum_{j(\neq i)}^p y_j \, w_j \, w_j^\intercal. 
	\end{align*}
	Note that $\Gamma_{k,i}^{(0)}$ is independent of $w_i$.
	We start from the (trivial) decomposition, for every $\lambda$:
	\begin{align}\label{eq_app:starting_point_cavity}
		-\frac{1}{\lambda} &= \left(\Gamma_k^{(0)} - \lambda \rI_k\right)^{-1}	- \frac{1}{\lambda} \frac{W^\intercal \, y \, W}{k} \left(\Gamma_k^{(0)} - \lambda \rI_k\right)^{-1}.
	\end{align}
	We will make use of the Sherman-Morrison formula that gives the inverse of a matrix perturbed by a rank-one change:
	\begin{align}
		\label{eq_app:Sherman-Morrison_full}
		\left(B + \tau \omega \omega^\intercal\right)^{-1}	&= B^{-1} - \frac{1}{1 + \tau \omega^\intercal B^{-1} \omega} B^{-1} \omega \omega^\intercal B^{-1}, \\
		\label{eq_app:Sherman-Morrison}
		\omega^\intercal \left(B + \tau \omega \omega^\intercal\right)^{-1}	&= \frac{1}{1 + \tau \omega^\intercal B^{-1} \omega} \omega^\intercal B^{-1}.
	\end{align}
	Using it in eq.~(\ref{eq_app:starting_point_cavity}) yields:
	\begin{align}\label{eq_app:starting_point_cavity_2}
		-\frac{1}{\lambda} &= \left(\Gamma_k^{(0)} - \lambda \rI_k\right)^{-1}	- \frac{\alpha}{\lambda} \frac{1}{p} \sum_{i=1}^p y_i \frac{w_i}{1 + \frac{y_i}{k} w_i^\intercal  (\Gamma^{(0)}_{k,i} - \lambda \rI_k)^{-1} w_i} w_i^\intercal \left(\Gamma_{k,i}^{(0)} - \lambda \rI_k\right)^{-1}.
	\end{align}
	 Taking the trace of eq.~(\ref{eq_app:starting_point_cavity_2}), 
	using the independence of $w_i$ and $\Gamma_{k,i}^{(0)}$, 
	and the concentration $\frac{1}{k} w_i^\intercal A w_i = \frac{1}{k} \mathrm{Tr} A$ with high probability for large $k$, we obtain the following equation:
	\begin{align}\label{eq_app:eq_gnu}
		-\frac{1}{\lambda} = g_\nu(\lambda)  - g_{\nu}(\lambda) \frac{\alpha}{\lambda} \int \rho_\Delta(\mathrm{d}t) \frac{t}{1+ t g_\nu(\lambda)}.
	\end{align}
	This is exactly the identity in Theorem~\ref{thm_app:silverstein} ! In the following, we will use very similar identities.
	A completely rigorous derivation of these would, however, require many technicalities to ensure in particular the concentration of all the involved quantities.
	It would exactly follow the proof of \cite{silverstein1995empirical}, and thus we do not repeat all the technicalities here.
	We can multiply eq.~(\ref{eq_app:starting_point_cavity_2}) by $\frac{W^\intercal W}{k}$, and take the trace:
	\begin{align*}
		- \frac{1}{\lambda} \frac{1}{k} \mathrm{Tr} \,\left[ \frac{WW^\intercal}{k}	\right] = S^{(1)}_k(\lambda) - \frac{\alpha}{\lambda} \frac{1}{p} \sum_i y_i \frac{ \frac{w_i^\intercal}{\sqrt{k}} \left(\Gamma_{k,i}^{(0)} - \lambda \rI_k\right)^{-1} \left(\frac{1}{k} \sum_{j(\neq i)} w_j w_j^\intercal + \frac{1}{k} w_i w_i^\intercal\right) \frac{w_i}{\sqrt{k}}}{1 + \frac{y_i}{k} w_i^\intercal  (\Gamma^{(0)}_{k,i} - \lambda \rI_k)^{-1} w_i}.
	\end{align*}
	In the large $p,k$ limit, this implies that $S_k^{(1)}(\lambda)$ converges to a well-defined limit $S^{(1)}(\lambda)$, and this limit satisfies:
	\begin{align*}
		- \frac{\alpha}{\lambda} = S^{(1)}(\lambda) - \frac{\alpha}{\lambda} \left[\int \rho_\Delta(\mathrm{d}t) \frac{t}{1 +t g_\nu(\lambda)}\right] \left(g_\nu(\lambda) + S^{(1)}(\lambda)\right)
	\end{align*}
	Using finally eq.~(\ref{eq_app:eq_gnu}), it is equivalent to:
	\begin{align*}
		S^{(1)}(\lambda) &= g_\nu(\lambda) \left[\alpha - (1 + \lambda g_\nu(\lambda))\right].	
	\end{align*}
	Multiplying eq.~(\ref{eq_app:starting_point_cavity_2}) by $\left(\frac{W^\intercal W}{k}\right)^2$ or $\left(\frac{W^\intercal W}{k}\right)^3$ yields, by the same analysis:
	\begin{align*}
		S^{(2)}(\lambda) &= g_\nu(\lambda) \left[\alpha (1+\alpha) - (1+2\alpha) (1+\lambda g_\nu(\lambda)) + (1+\lambda g_\nu(\lambda))^2\right], \\
		S^{(3)}(\lambda) &= g_\nu(\lambda) \left[(\alpha+3\alpha^2+\alpha^3) - (1+5\alpha + 3\alpha^2) (1+\lambda g_\nu(\lambda)) \right. \\
		&\left. \hspace{1.5cm} + (2+3\alpha)(1+\lambda g_\nu(\lambda))^2 - (1+\lambda g_\nu(\lambda))^3\right].
	\end{align*}
	The convergence of $S_k^{(1,1)}(\lambda)$ and $S_k^{(1,2)}(\lambda)$ follows from the same analysis, as well as the equations they satisfy.
	We detail the derivation of the equation on $S^{(1,1)}(\lambda)$ and leave the derivation of the second equation for the reader.
	We multiply eq.~(\ref{eq_app:starting_point_cavity_2}) by $\frac{W^\intercal W}{k}$. To simplify the calculations, we make use of concentrations, and denote $F_{i} \equiv \frac{W^\intercal W}{k} - \frac{1}{k} w_i w_i^\intercal$, which is independent of $w_i$.
	We obtain at leading order as $p \to \infty$:
	\begin{align*}
		-\frac{W^\intercal W}{k \lambda} &= \left(\Gamma_k^{(0)} - \lambda \rI_k\right)^{-1} \frac{W^\intercal W}{k}	- \frac{\alpha}{\lambda} \frac{1}{p} \sum_{i=1}^p \frac{y_i}{1 + y_i g_\nu(\lambda)} w_i w_i^\intercal \left(\Gamma_{k,i}^{(0)} - \lambda \rI_k\right)^{-1} F_i \\
		& \hspace{1cm}- \frac{\alpha}{\lambda} \frac{1}{p} \sum_{i=1}^p \frac{y_i g_\nu(\lambda)}{1 + y_i g_\nu(\lambda)} w_i w_i^\intercal.
	\end{align*}
	We multiply this equation by $(\Gamma_k^{(0)} - \lambda \rI_k)^{-1}$ and we use Sherman-Morrison formula eq.~(\ref{eq_app:Sherman-Morrison_full}):
	\begin{align*}
		\left(\Gamma_k^{(0)} - \lambda \rI_k\right)^{-1} &=\left(\Gamma_{k,i}^{(0)} - \lambda \rI_k\right)^{-1} - \left(\Gamma_{k,i}^{(0)} - \lambda \rI_k\right)^{-1} \frac{y_i w_i w_i^\intercal}{1 + y_i g_\nu(\lambda)} \left(\Gamma_{k,i}^{(0)} - \lambda \rI_k\right)^{-1}.
	\end{align*}
	Using again the concentration of $\frac{1}{k} w^\intercal A w$ on $\frac{1}{k} \mathrm{Tr}[A]$, this yields the cumbersome expression:
	\begin{align}
		\label{eq_app:S11_last}
		-\frac{W^\intercal W}{k \lambda} \left(\Gamma_k^{(0)} - \lambda \rI_k\right)^{-1} &= \left(\Gamma_k^{(0)} - \lambda \rI_k\right)^{-1} \frac{W^\intercal W}{k}	\left(\Gamma_k^{(0)} - \lambda \rI_k\right)^{-1} \\
		&- \frac{\alpha}{\lambda} \frac{1}{p} \sum_{i=1}^p \frac{y_i}{1 + y_i g_\nu(\lambda)} w_i w_i^\intercal \left(\Gamma_{k,i}^{(0)} - \lambda \rI_k\right)^{-1} F_i \left(\Gamma_{k,i}^{(0)} - \lambda \rI_k\right)^{-1} \nonumber\\
		& + \frac{\partial_\lambda S^{(1)}(\lambda)}{\lambda} \, \frac{\alpha}{p} \sum_{i=1}^p \frac{y_i^2}{(1 + y_i g_\nu(\lambda))^2} w_i w_i^\intercal \left(\Gamma_{k,i}^{(0)} - \lambda \rI_k\right)^{-1} \nonumber\\
		& - \frac{\alpha}{\lambda} \frac{1}{p} \sum_{i=1}^p \frac{y_i g_\nu(\lambda)}{(1 + y_i g_\nu(\lambda))^2} w_i w_i^\intercal \left(\Gamma_{k,i}^{(0)} - \lambda \rI_k\right)^{-1}.\nonumber
	\end{align}
	We finally multiply this equation by $\frac{W^\intercal W}{k}$ and take its trace. Using again the concentrations, we reach:
	\begin{align*}
		-\frac{S^{(2)}(\lambda)}{\lambda} &= S^{(11)}(\lambda) - \frac{\alpha}{\lambda p} \sum_{i=1}^p \frac{y_i}{1+y_i g_\nu(\lambda)} \left[S^{(11)}(\lambda) + \partial_\lambda S^{(1)}(\lambda)\right] \\
		& + \frac{\partial_\lambda S^{(1)}(\lambda)}{\lambda} \, \frac{\alpha}{p} \sum_{i=1}^p \frac{y_i^2}{(1 + y_i g_\nu(\lambda))^2} \left[g_\nu(\lambda) + S^{(1)}(\lambda)\right]\\
		& - \frac{\alpha}{\lambda} \frac{1}{p} \sum_{i=1}^p \frac{y_i g_\nu(\lambda)}{(1 + y_i g_\nu(\lambda))^2} \left[g_\nu(\lambda) + S^{(1)}(\lambda)\right].
	\end{align*}
	We now take the limit $p \to \infty$ in the sum over $i$ and use Theorem~\ref{thm_app:silverstein} in the form:
	\begin{align*}
		\frac{\alpha}{\lambda} \int \rho_\Delta(\mathrm{d}t) \frac{t}{1+tg_\nu(\lambda)} &= 1 + \frac{1}{\lambda g_\nu(\lambda)}.	
	\end{align*}
	Inserting this into eq.~(\ref{eq_app:S11_last}) along with some trivial algebra yields:
	\begin{align*}
		S^{(1,1)}(\lambda) &= g_\nu(\lambda) S^{(2)}(\lambda) - \left[1 + \lambda g_\nu(\lambda)\right] \partial_\lambda S^{(1)}(\lambda) \\ 
		&+ \alpha g_\nu(\lambda) \left[g_\nu(\lambda) + S^{(1)}(\lambda)\right] \int \frac{\rho_\Delta(\mathrm{d}t) t}{\left(1+t g_\nu(\lambda)\right)^2} \left[t \, \partial_\lambda S^{(1)}(\lambda) - g_\nu(\lambda)\right],
	\end{align*}
	which is what we aimed to show.
	Performing the same analysis for $S^{(1,2)}(\lambda)$ ends the proof.
\end{proof}

\begin{proof}[Proof of Lemma~\ref{lemma_app:evalue_transition_vv}]
	Point $(i)$ is trivial by definition of $S_k^{(r)}(\lambda)$ and the almost sure convergence proven in Lemma~\ref{lemma_app:hierarchy_Sk}.
	We turn to points $(ii)$ and $(iii)$.
	Let us denote the following function:
	\begin{align*}
		T^{(2)}(s) &\equiv s \left[\alpha(1+\alpha) - (1+2\alpha)\left(1+s g_\nu^{-1}(s)\right) + \left(1+s g_\nu^{-1}(s)\right)^2\right].
	\end{align*}
	By Lemma~\ref{lemma_app:hierarchy_Sk}, we have $T^{(2)}(s) = S^{(2)}(g_\nu^{-1}(s))$ so $T^{(2)}(s) < 0$ for $s \in (s_{\rm edge}, 0)$ by negativity of $S^{(2)}(\lambda)$ (as the trace of a negative matrix).
	Therefore, point $(ii)$ is equivalent to:
	\begin{align}\label{eq_app:lemma_eigenvalue_s}
		\forall s \in (s_{\rm edge}, 0), \quad T^{(2)}(s)&= - \alpha \Delta \Leftrightarrow s = g_\nu(1) \textrm{ and } \Delta \leq \Delta_c(\alpha),
	\end{align}
	while point $(iii)$ means that for every $\Delta > \Delta_c(\alpha)$,
	\begin{align}
		\forall s \in (s_{\rm edge}, 0), \quad T^{(2)}(s)&> - \alpha \Delta.
	\end{align}
	The condition $s > s_{\rm edge}$ arises naturally as the counterpart of $z \geq \lambda_{\rm max}$. Recall that by
	Corollary~\ref{corollary_app:lambdamax_vv}, we have $\lambda_{\rm max} \leq 1$ for all $\Delta$.
	As $g_\nu^{-1}(s)$ is here completely explicit by eq.~(\ref{eq_app:silverstein_inverse}), and recalling the form of $\rho_\Delta$ in eq.~(\ref{eq_app:def_rho_Delta}), 
	it is easy to show by an explicit computation the following identity:
	\begin{align*}
		\forall s \neq -1, \quad T^{(2)}(s) &= - \alpha \Delta	+ \alpha \left[g_\nu^{-1}(s) - 1\right] \frac{s - \Delta - 2s \Delta + \sqrt{s^2 - 2 s(1+s) \Delta + \Delta^2}}{2 (1+s)}, \\
		T^{(2)}(-1) &= 
		\begin{cases}
			- \alpha (1+\alpha) \quad & \textrm{ if } \Delta \geq 1, \\		
			- \alpha \Delta (1+\alpha \Delta) \quad & \textrm{ if } \Delta \leq 1.
		\end{cases}
	\end{align*}
	It is then easy to see that the only possible solution to $T(s) = -\alpha\Delta$ with $s \in (s_{\rm edge},0)$ is $s = g_\nu(1)$, if $g_\nu(1) \neq -1$.
	However, by Lemma~\ref{lemma_app:sedge_1}, for any
	$\Delta > \Delta_c(\alpha)$ we have 
	$s_{\rm edge} < -1$. Moreover, in this case, one computes very easily (all expressions are explicit) $g_\nu^{-1}(-1) = 1$. Given the identity above, there is therefore no solution to $T^{(2)}(s) = - \alpha \Delta$ in $(s_{\rm edge}, 0)$. By continuity of $T^{(2)}(s)$, and since 
	$\lim_{s \to 0} T^{(2)}(s) = 0$, this implies $T^{(2)}(s) > - \alpha \Delta$ for $s \in (s_{\rm edge},0)$, which proves point $(iii)$.

	Assume now $\Delta \leq \Delta_c(\alpha)$. Note that the case $\Delta = \Delta_c(\alpha)$ is easy, as $s_{\rm edge} = -1$ is the unique solution to $T^{(2)}(s) = -\alpha (1+\alpha)$.
	For $\Delta < \Delta_c(\alpha)$, by Lemma~\ref{lemma_app:sedge_1} we obtain $-1 < s_{\rm edge}$. In particular, $g_\nu(1) > s_{\rm edge} > -1$, and we thus have that $s = g_\nu(1)$
	is a solution (and the only one) to $T^{(2)}(s) = - \alpha \Delta$  by the identity shown above.
	This shows $(ii)$ and ends the proof of Lemma~\ref{lemma_app:evalue_transition_vv}.
\end{proof}

\paragraph{Correlation of the leading eigenvector}

We now turn to the study of the leading eigenvector.
Let $\tilde{\bv}$ be an eigenvector associated with the largest eigenvalue $\lambda_1$, normalized
such that $\norm{\tilde{\bv}}^2 = p$.
Then we have:
\begin{align}
	(\lambda_1 \rI_p - \Gamma_p^{(0)}) \tilde{\bv} &= \frac{1}{\Delta} \frac{WW^\intercal}{k} \frac{\bv^\intercal \tilde{\bv}}{p} \bv.
\end{align}
By normalization of $\tilde{\bv}$, we obtain:
\begin{align*}
	\tilde{\bv} &= \sqrt{p}\frac{\left(\lambda_1 \rI_p - \Gamma_p^{(0)}\right)^{-1} \frac{WW^\intercal}{k} \bv}{\sqrt{\bv^\intercal \frac{WW^\intercal}{k}\left(\lambda_1 \rI_p -\left(\Gamma_p^{(0)}\right)^\intercal\right)^{-1} \left(\lambda_1 \rI_p -\Gamma_p^{(0)}\right)^{-1}\frac{WW^\intercal}{k} \bv}},
\end{align*}
and therefore:
\begin{align}
	\frac{1}{p^2}\left|\tilde{\bv}^T \bv\right|^2 &= \frac{1}{p} \frac{\left[\bv^\intercal \left(\lambda_1 \rI_p - \Gamma_p^{(0)}\right)^{-1} \frac{WW^\intercal}{k} \bv\right]^2}{\bv^\intercal \frac{WW^\intercal}{k}  \left(\lambda_1 \rI_p - \left(\Gamma_p^{(0)}\right)^\intercal\right)^{-1} \left(\lambda_1 \rI_p - \Gamma_p^{(0)}\right)^{-1}\frac{WW^\intercal}{k} \bv}.
\end{align}
Using $\bv = \frac{W}{\sqrt{k}} \bz$ and the concentration of $\frac{1}{k} \bz^\intercal A \bz$ on $\frac{1}{k}\mathrm{Tr} \, A$, we reach that as $p,k \to \infty$, we have:
\begin{align}
	\frac{1}{p^2}\left|\tilde{\bv}^T \bv\right|^2 &\sim \frac{\left[\frac{1}{p}\mathrm{Tr}\, \left\{\left(\lambda_1 \rI_p - \Gamma_p^{(0)}\right)^{-1} \left(\frac{WW^\intercal}{k}\right)^2 \right\}\right]^2}{\frac{1}{p} \mathrm{Tr} \, \left\{\left(\lambda_1 \rI_p -\left(\Gamma_p^{(0)}\right)^\intercal\right)^{-1}\left(\lambda_1 \rI_p - \Gamma_p^{(0)}\right)^{-1} \left(\frac{WW^\intercal}{k}\right)^3\right\}}.
\end{align}

The numerator is equal to $[\alpha^{-1} S_k^{(2)}(\lambda_1)]^2$, using the $S^{(r)}$ functions that we introduced in Lemma~\ref{lemma_app:hierarchy_Sk}.
Let us compute the denominator. Recall that we can write $\Gamma_p^{(0)} = W W^\intercal M / k$, with a symmetric matrix $M$ that is independent of $W$.
For any $z$ large enough, we can expand:
\begin{align*}
	 &\mathrm{Tr} \, \left\{\left(z \rI_p -\left(\Gamma_p^{(0)}\right)^\intercal\right)^{-1}\left(z \rI_p - \Gamma_p^{(0)}\right)^{-1} \left(\frac{WW^\intercal}{k}\right)^3\right\},  \\
	 &= \sum_{a=0}^\infty \sum_{b=0}^\infty z^{-a-b-2} \, \mathrm{Tr} \, \left\{\left(M \frac{W W^\intercal}{k}\right)^a \left(\frac{W W^\intercal}{k} M)\right)^b \left(\frac{W W^\intercal}{k}\right)^3\right\}, \\
	 &\overset{(a)}{=} \sum_{a=0}^\infty \sum_{b=0}^\infty z^{-a-b-2} \, \mathrm{Tr} \, \left\{\left(\frac{W^\intercal M W}{k}\right)^a \frac{W^\intercal W}{k} \left(\frac{W^\intercal M W}{k}\right)^b \left(\frac{W^\intercal W}{k}\right)^2\right\}, \\
	 &= \mathrm{Tr} \, \left\{\left(z \rI_k -\Gamma_k^{(0)}\right)^{-1}\frac{W^\intercal W}{k}\left(z \rI_k - \Gamma_k^{(0)}\right)^{-1} \left(\frac{W^\intercal W}{k}\right)^2\right\}, \\
	 &= k S^{(1,2)}_k(z),
\end{align*}
where in $(a)$ we used the cyclicity of the trace. Given Corollary~\ref{corollary_app:lambdamax_vv}, we know $\liminf_{p \to \infty} \lambda_1 \geq \lambda_{\rm max}$, so we can use the above calculation 
to write:
\begin{align}\label{eq_app:epsilon_Delta_vv}
	\epsilon(\Delta) &= \lim_{\lambda \to \lambda_1} \lim_{k \to \infty} \frac{1}{\alpha} \frac{\left[S_k^{(2)}(\lambda)\right]^2}{S_k^{(1,2)}(\lambda)}.
\end{align}
As in the eigenvalue transition proof, to make this fully rigorous one would need to use more precisely the concentration results, 
and follow exactly the lines of \cite{benaych2011eigenvalues}.
We now use the transition of the leading eigenvalue (Corollary~\ref{corollary_app:lambdamax_vv}), that gives us the value of $\lambda_1$. 
\begin{itemize}
	\item For $\Delta < \Delta_c(\alpha)$, we know that $\lambda_1$ converges almost surely to $1$. Consequently, we have in this case:
	\begin{align*}
		\epsilon(\Delta) &= \frac{1}{\alpha} \frac{\left[S^{(2)}(1)\right]^2}{S^{(1,2)}(1)}.
	\end{align*}
	By Lemma~\ref{lemma_app:evalue_transition_vv}, we know that $S^{(2)}(1) = - \alpha \Delta$. Moreover, by Corollary~\ref{corollary_app:lambdamax_vv} 
	$\lambda_{\rm max} < 1$. This implies that $S^{(1,2)}(1) \in (0, +\infty)$. Indeed, $1$ is out of the bulk of $\nu(\alpha, \Delta)$, so $g_\nu(1) \in (-\infty,0)$ and 
	by the relations shown in Lemma~\ref{lemma_app:hierarchy_Sk}, all the transforms $S^{(r)}(1)$ and $S^{(r,q)}(1)$ will be finite. Note that $S^{(1,2)}(1) > 0$ by positivity of the matrices involved.
	This implies that for every $\Delta < \Delta_c(\alpha)$, $\epsilon(\Delta) > 0$.
	\item For $\Delta = \Delta_c(\alpha)$, we have $\lambda_{\rm max} = 1$ and $\lim_{\lambda \to 1}S^{(2)}(\lambda) = -\alpha \Delta$ as we have shown.
	For every $r,q$, let us define the functions $T^{(r)}$ and $T^{(r,q)}$ by $S^{(r)}(\lambda) = T^{(r)}[g_\nu(\lambda)]$ and $S^{(r,q)}(\lambda) = T^{(r,q)}[g_\nu(\lambda)]$.
	By Lemma~\ref{lemma_app:hierarchy_Sk} and the chain rule, we have:
	\begin{align}\label{eq_app:T12}
	&\forall s \in (s_{\rm edge}, 0), \\
	& T^{(1,2)}(s) = s T^{(3)}(s) - \left[1 + s g^{-1}_\nu(s)\right] \left[T^{(1,1)}(s) + (1+\alpha)\frac{\partial_s T^{(1)}(s)}{\partial_s g^{-1}_\nu(s)} \right] \nonumber \\ 
	&+ \alpha s \left[(1+\alpha)s + T^{(1)}(s)+ T^{(2)}(s)\right] \int \frac{\rho_\Delta(\mathrm{d}t) t}{\left(1+t s\right)^2} \left[t \, \frac{\partial_s T^{(1)}(s)}{\partial_s g^{-1}_\nu(s)} - s\right]. \nonumber
	\end{align}
	Recall that $g_\nu^{-1}(s)$ is explicit by eq.~(\ref{eq_app:silverstein_inverse}) and $s_{\rm edge} = \lim_{\lambda \to \lambda_{\rm max}} g_\nu(\lambda)$. It moreover satisfies (cf Theorem~\ref{thm_app:rmt_vv})
	$\partial_s g_\nu^{-1}(s_{\rm edge}) = 0$.
	For $\Delta = \Delta_c(\alpha)$, by Lemma~\ref{lemma_app:sedge_1}
	we have $g_\nu(1) = -1 = s_{\rm edge}$. 
	It is then only trivial algebra to verify from eq.~(\ref{eq_app:T12}) and the remaining relations of Lemma~\ref{lemma_app:hierarchy_Sk} that $T^{(1,2)}(-1) = + \infty$, 
	which implies $\epsilon(\Delta_c(\alpha)) = 0$. 
	\item We investigate here the $\Delta \to 0$ limit. In this limit, we know from eq.~(\ref{eq_app:epsilon_Delta_vv}) and the analysis in the case $\Delta < \Delta_c(\alpha)$ above
	that 
	\begin{align*}
		\lim_{\Delta \to 0} \epsilon(\Delta) &= \lim_{\Delta \to 0} \frac{\alpha \Delta^2}{S^{(1,2)}(1)}.
	\end{align*}
	It is again heavy but straightforward algebra to verify from eq.~(\ref{eq_app:T12}) and the remaining relations of Lemma~\ref{lemma_app:hierarchy_Sk} that as $\Delta \to 0$ and for any $s \in (s_{\rm edge}, 0)$:
	\begin{align*}
		T^{(1,2)}(s) &= \alpha \Delta^2 + \mathcal{O}(\Delta^{3}).	
	\end{align*}
	This yields $\lim_{\Delta \to 0} \epsilon(\Delta) = 1$.
	\item Finally, we consider $\Delta > \Delta_c(\alpha)$.
	By eq.~(\ref{eq_app:epsilon_Delta_vv}) and item $(iii)$ of Lemma~\ref{lemma_app:evalue_transition_vv},
	 to obtain $\epsilon(\Delta) = 0$ we only need to prove that 
	$\lim_{\lambda \to \lambda_{\rm max}} S^{(1,2)}(\lambda) = +\infty$.
	Equivalently, we must show $\lim_{s \to s_{\rm edge}} T^{(1,2)}(s) = +\infty$. Recall that $\partial_s g_\nu^{-1}(s_{\rm edge}) = 0$
	and that since $s_{\rm edge}$ is finite, all $T^{(r)}(s_{\rm edge})$ for $r = 0,1,2,3$ are finite as well by Lemma~\ref{lemma_app:hierarchy_Sk}.
	It thus only remains to check that $\lim_{s \to s_{\rm edge}} T^{(1,2)}(s)\partial_s \, g_\nu^{-1}(s) > 0$. This would imply that $\lim_{s \to s_{\rm edge}} T^{(1,2)}(s) = + \infty$.
	We put this statement as a lemma, actually stronger than what we need:
	\begin{lemma}\label{lemma_app:bound_T12}
		For every $\alpha > 0$ and $\Delta > 1$, we have
		\begin{align*}
			\liminf_{s \to s_{\rm edge}} T^{(1,2)}(s) \partial_s g_\nu^{-1}(s) > 0.	
		\end{align*}
	\end{lemma}
	We prove this for every $\Delta > 1$, while only the case $\Delta > 1+\alpha$ is needed in our analysis.
	As already argued, this lemma ends the proof.
	\begin{proof}[Proof of Lemma~\ref{lemma_app:bound_T12}]
		The idea is to lower bound $S^{(1,2)}(\lambda)$ by $\partial_\lambda g_\nu(\lambda)$, for every $\lambda > \lambda_{\rm max}$.
		We separate three cases:
		%[leftmargin=5.5mm]
		\begin{itemize}
			\item First, assume $\alpha > 1$. Then $W^\intercal W/k$ is full rank. In particular, by the classical results of \cite{marvcenko1967distribution}, 
			its lowest eigenvalue, denoted $\zeta_{\rm \min}$ converges almost surely to $(1-\alpha^{-1/2})^2$, the left edge of the Marchenko-Pastur distribution.
			Moreover, for any two symmetric positive square matrices $A$ and $B$, we know that $\mathrm{Tr}\,[AB] \geq 0$. Indeed, there exists a positive square root  
			of $A$, and $\mathrm{Tr}\,[AB] = \mathrm{Tr}[A^{1/2} B A^{1/2}] \geq 0$. This implies immediately that if $a_0$ is the smallest eigenvalue of $A$, then 
			$\mathrm{Tr}\,[AB] \geq a_0 \mathrm{Tr}\,[B]$, as $A-a_0 \rI$ is positive. We can use this to write, for any $\lambda > \lambda_{\rm max}$:
			\begin{align*}
			S^{(1,2)}_k(\lambda) &= \frac{1}{k} \mathrm{Tr} \, \left[\left(\Gamma_k^{(0)} - \lambda \rI_k\right)^{-1} \left(\frac{W^\intercal W}{k}\right) \left(\Gamma_k^{(0)} - \lambda \rI_k\right)^{-1} \left(\frac{W^\intercal W}{k}\right)^2\right], \\
			&\geq \zeta_{\rm min}^2  \frac{1}{k} \mathrm{Tr} \, \left[\left(\Gamma_k^{(0)} - \lambda \rI_k\right)^{-1} \left(\frac{W^\intercal W}{k}\right) \left(\Gamma_k^{(0)} - \lambda \rI_k\right)^{-1} \right], \\
			&\geq \zeta_{\rm min}^3  \frac{1}{k} \mathrm{Tr} \, \left[\left(\Gamma_k^{(0)} - \lambda \rI_k\right)^{-2} \right].
			\end{align*}
			Taking the limit $k \to \infty$ in this last inequality, we obtain:
			\begin{align}
				S^{(1,2)}(\lambda) \geq \left(1-\alpha^{-1/2}\right)^6 \partial_\lambda g_\nu(\lambda).
			\end{align}
			Taking the limit $\lambda \to \lambda_{\rm max}$ (or equivalently $s \to s_{\rm edge}$) yields
			\begin{align}
				\liminf_{s \to s_{\rm edge}} T^{(1,2)}(s) \partial_s g_\nu^{-1}(s) \geq   \left(1-\alpha^{-1/2}\right)^6 > 0.	
			\end{align}
			\item Now assume $\alpha < 1$. We do the same reasoning, as $W W^\intercal / k$ is now full rank, and it smallest eigenvalue, also denoted $\zeta_{\rm min}$, converges 
			a.s. as $k \to \infty$ to $(1-\sqrt{\alpha})^2$. We know (see the beginning of the current proof of the eigenvector correlation) that we can rewrite $S_k^{(1,2)}(\lambda)$ as the trace of a $p \times p$ matrix:
			\begin{align*}
			S^{(1,2)}_k(\lambda) &= \frac{1}{k} \mathrm{Tr} \, \left[\left(\left(\Gamma_k^{(0)}\right)^\intercal - \lambda \rI_k\right)^{-1} \left(\Gamma_k^{(0)} - \lambda \rI_k\right)^{-1} \left(\frac{W W^\intercal}{k}\right)^3\right], \\
			&\geq \zeta_{\rm min}^3  \frac{1}{k} \mathrm{Tr} \, \left[\left(\left(\Gamma_k^{(0)}\right)^\intercal - \lambda \rI_k\right)^{-1}\left(\Gamma_k^{(0)} - \lambda \rI_k\right)^{-1} \right], \\
			&\geq \zeta_{\rm min}^3  \frac{1}{k} \mathrm{Tr} \, \left[\left(\Gamma_k^{(0)} - \lambda \rI_k\right)^{-2} \right],
			\end{align*}
			in which the last inequality comes from $\mathrm{Tr}\, [A A^\intercal] \geq \mathrm{Tr}\,[A^2]$ for any positive square matrix $A$.
			Once again, taking the limit $k \to \infty$, and then the limit $\lambda \to \lambda_{\rm max}$, this yields
			\begin{align}
				\liminf_{s \to s_{\rm edge}} T^{(1,2)}(s) \partial_s g_\nu^{-1}(s) \geq   \left(1-\alpha^{1/2}\right)^6 > 0.	
			\end{align}
			\item Finally, we treat the $\alpha = 1$ case. In this case, we can not use easy bounds as in the two previous cases as the support of the Marchenko-Pastur distribution touches $0$.
			However, recall that everything is explicit here : $\rho_\Delta$ is given by eq.~(\ref{eq_app:def_rho_Delta}), $g_\nu^{-1}(s)$ is given by eq.~(\ref{eq_app:silverstein_inverse}) and 
			Lemma~\ref{lemma_app:hierarchy_Sk} gives all the $T^{(r)}$ and $T^{(r,q)}$ in terms of $g_\nu^{-1}$ and $\rho_\Delta$.
			We can moreover use what we proved in Theorem~\ref{thm_app:rmt_vv}:
			\begin{align*}
				\partial_s g_{\nu}^{-1}(s_{\rm edge}) &= \frac{1}{s^2} - \alpha \int \rho_\Delta(\mathrm{d}t) \frac{t^2}{(1+ts_{\rm edge})^2} = 0.
			\end{align*}
			This can be used to simplify the term $\partial_s T^{(1)}(s)$ and the term $\int \rho_\Delta(\mathrm{d}t) \frac{t^2}{(1+ts)^2}$.
			Some heavy but straightforward algebra yields from these relations that the following limit is finite, and is given by:
			\begin{align*}
				\lim_{s \to s_{\rm edge}} T^{(1,2)}(s) \, \partial_s g_\nu^{-1}(s) &= h(s_{\rm edge}),
			\end{align*}
			with 
			\begin{align*}
				&h(s) = \frac{h_1(s)^2 \times h_2(s)}{4 s^6} , \\
				&h_1(s) = -\Delta +\sqrt{\Delta ^2+s^2-2 \Delta  (2 s+1) s}+s, \\
				&h_2(s) = 3 \Delta -3 \sqrt{\Delta ^2+s^2-2 \Delta  (2 s+1) s}+s (4 s-3), 
			\end{align*}
			It is then very simple algebra (solving quadratic equations and using $\Delta > 1$) to see that  there is no
			real negative solution to $h(s) = 0$, and that $h(s) > 0$ for all $s \in (-\infty,0)$. This implies that $h(s_{\rm edge}) > 0$, which ends the proof.
		\end{itemize}
	\end{proof}
\end{itemize}

All together, this ends the proof of Theorem~\ref{thm_app:transition_vv}.
%%%%%%%%%%%%%%%%%%%%%%%%%%%%%%%%%%%%%%%%%%%%%%%%%%%%%%%%%%%%%%%
\subsubsection{Proof of Theorem~\ref{thm_app:rmt_uv} and Corollary~\ref{corollary_app:lambdamax_uv}}
%%%%%%%%%%%%%%%%%%%%%%%%%%%%%%%%%%%%%%%%%%%%%%%%%%%%%%%%%%%%%%%
\paragraph{Proof of Theorem~\ref{thm_app:rmt_uv}}
\begin{proof}
	The proof is very similar to the proof of Theorem~\ref{thm_app:rmt_vv}, and we will only point out the main differences.
	The proof of $(i)$ is exactly the same as the proof of the point $(i)$ of Theorem~\ref{thm_app:rmt_vv}, once one notices
	that for $\Delta \leq \Delta_{\rm pos}(\beta)$, the support of $\rho_{\beta,\Delta}$ is a subset of $\bbR_-$.
	We thus turn to the proof of $(ii)$.
	Again, the spectrum of $\Gamma_p^{uv}$, given by eq.~(\ref{eq_app:def_gammap_uv}) is, up to $0$ eigenvalues, the same as the spectrum of $\Gamma_k^{uv}$, defined as follows:
	\begin{align}
		\Gamma_k^{uv} &\equiv \frac{1}{\Delta} \frac{1}{k} \du{W}^\intercal \( \frac{1}{1+\Delta}\frac{\du{y}^\intercal\du{y}}{p}  -  \beta \,\rI_p  \) \du{W} \in \bbR^{k \times k}.
	\end{align}
	We drop for simplicity the $uv$ exponents in these matrices.
	Once again, we can apply the Silverstein equation of Theorem~\ref{thm_app:silverstein} and the same arguments that we used in the proof of Theorem~\ref{thm_app:rmt_vv}
	completely transpose here. One notices that, by the classical Marchenko-Pastur results \cite{marvcenko1967distribution}, the spectral distribution of $\du{y}^\intercal\du{y}/(p \Delta(1+\Delta))  -  (\beta/\Delta) \,\rI_p $
	converges almost surely and in law to $\rho_{\beta,\Delta}$, before repeating the exact arguments of the proof of Theorem~\ref{thm_app:rmt_vv}.
	This ends the proof of Thm.~\ref{thm_app:rmt_uv}.
\end{proof}

\paragraph{Proof of Corollary~\ref{corollary_app:lambdamax_uv}}
\begin{proof}
		Let $\alpha, \beta > 0$. We note:
	\begin{itemize}
		\item By Theorem~\ref{thm_app:rmt_uv}, we know that if $\Delta = \Delta_{\rm pos}(\beta)$, then $\lambda_{\rm max} \leq 0$.
		\item It is trivial by the form of $\Gamma_p$, see eq.~(\ref{eq_app:def_gammap_uv}), that as $\Delta \to + \infty$, $\lambda_{\rm max} \to 0$.
	\end{itemize}
	Let $z_{\rm edge} = -\frac{1}{s_{\rm edge}} + \alpha \int \rho_{\beta, \Delta}(\mathrm{d}t) \frac{t}{1 + s_{\rm edge}t}$. Then we know that $\lambda_{\rm max} = z_{\rm edge}$ if $\alpha \leq 1$ and
	$\lambda_{\rm max} = \max(0,z_{\rm edge})$ if $\alpha > 1$. In particular, by the remark above, $z_{\rm edge} \leq 0$ for $\Delta \leq \Delta_{\rm pos}(\beta)$ and $z_{\rm edge} \to 0^+$ as $\Delta \to \infty$.
	It is easy to see that $z_{\rm edge}$ is a continuous and derivable function of $\Delta$, so that if we show the two following facts for any $\Delta \geq \Delta_{\rm pos}(\beta)$:
	\begin{align}
		\label{eq_app:dzedge_uv}
		\frac{\mathrm{d} z_{\rm edge}}{\mathrm{d} \Delta} &= 0 \Leftrightarrow \Delta = \Delta_c(\alpha,\beta) = \sqrt{\beta(1+\alpha)} \\
		\label{eq_app:zedge_uv}
		z_{\rm edge}(\Delta_c(\alpha, \beta)) &= 1,
	\end{align}
	this would end the proof as $z_{\rm edge}$ would necessarily have a unique local maximum, located in $\Delta_c(\alpha, \beta)$, in which we have
	$\lambda_{\rm max} = 1$. We thus prove eq.~(\ref{eq_app:dzedge_uv}) and eq.~(\ref{eq_app:zedge_uv}) in the following.

	\paragraph{Proof of eq.~(\ref{eq_app:dzedge_uv}):} By the chain rule,
	\begin{align*}
		\frac{\mathrm{d} z_{\rm edge}}{\mathrm{d} \Delta} &= \frac{\partial z_{\rm edge}}{\partial \Delta} + \frac{\partial s_{\rm edge}}{\partial \Delta} \frac{\partial z_{\rm edge}}{\partial s_{\rm edge}}, \\
			&= \frac{\partial z_{\rm edge}}{\partial \Delta},
	\end{align*}
	by the very definition of $s_{\rm edge}$, c.f. Theorem~\ref{thm_app:rmt_uv}, since $z_{\rm edge} = g_{\nu}^{-1}(s_{\rm edge})$.
	 Given the explicit form of $\rho_{\beta,\Delta}$, c.f. eq.~(\ref{eq_app:def_rho_beta_Delta}), one can compute $z_{\rm edge}$ as a function of $s_{\rm edge}$.
	 Its expression is cumbersome, but nevertheless explicit (we write $s$ instead of $s_{\rm edge}$ to avoid too heavy expressions):
	\begin{align*}
		z_{\rm edge} &= \frac{-\alpha  \Delta  (\Delta +1)+\alpha  \sqrt{\Delta ^2 (\Delta +1)^2+s^2 \left(\beta ^2-2 \beta  \Delta  (2 \Delta +1)+\Delta ^2\right)-2 \Delta  (\Delta +1) s (\beta -\Delta )}}{2 s^2 (\beta  s-\Delta )} \\
				&\hspace{1cm} + \frac{2 (\alpha -1) \beta  s^2+\alpha s (\beta -\Delta )+2 \Delta  s}{2 s^2 (\beta  s-\Delta )}.
	\end{align*}

	From this expression, it is simple analysis to verify that the only $s_{\rm edge} \in (-z_1(\beta,\Delta)^{-1},0)$
	that satisfies $\frac{\partial z_{\rm edge}}{\partial \Delta} = 0$ is $s_{\rm edge} = -1$, and only if $\Delta > \sqrt{\beta}$.
	Recall that $s_{\rm edge}$ is defined as the solution to:
	\begin{align*}
		\alpha \int \rho_{\beta,\Delta}(\mathrm{d}t) \left(\frac{s_{\rm edge}t}{1 + s_{\rm edge}t}\right)^2 &= 1.
	\end{align*}
	Inserting $s_{\rm edge} = -1$ into this equation and using the explicit form of $\rho_{\beta,\Delta}$ of eq.~(\ref{eq_app:def_rho_beta_Delta}) and that $\Delta > \sqrt{\beta}$,
	this reduces to:
	\begin{align*}
		\frac{\alpha \beta}{\Delta^2 - \beta} &= 1,
	\end{align*}
	which is equivalent to $\Delta = \Delta_c(\alpha,\beta) = \sqrt{\beta(1+\alpha)}$.

	\paragraph{Proof of eq.~(\ref{eq_app:zedge_uv}):} Given the computation above, we know that for $\Delta = \Delta_c(\alpha, \beta)$ we have
	$s_{\rm edge} = -1$. Given eq.~(\ref{eq_app:def_rho_beta_Delta}), it is straightforward to compute:
	\begin{align*}
		z_{\rm edge}(\Delta_c(\alpha, \beta)) &= -1 + \alpha \int \rho_{\Delta_c(\alpha, \beta)}(\mathrm{d}t) \frac{t}{1-t}, \\
		&= 1.
	\end{align*}
\end{proof}

\subsection{A note on non-linear activation functions}\label{subsec_app:nonlinear_rmt}

We consider here a non-linear activation function, in the spiked Wigner model or the spiked Wishart model.
In these models, the spectral method with a non-linear activation function consists in taking the largest eigenvalue and the corresponding eigenvector of the matrix $\Gamma_{p}^{uu}$
(for the spiked Wigner model) or $\Gamma_p^{uv}$ (for the spiked Wishart model). These matrices are given by:
\begin{align*}
	\Gamma^{uu}_p &=  \frac{1}{\Delta}  \( (a-b) \rI_p  +  b  \frac{W W^\intercal}{k} + c  \frac{\id_p \id_k^\intercal}{k} \frac{W^\intercal}{\sqrt{k} }  \) \times \( \frac{Y}{\sqrt{p}}  -  a \id_{M}  \) \,, \\
	\Gamma^{uv}_p &=  \frac{1}{\Delta}  \( (a-b) \rI_p  +  b  \frac{W W^\intercal}{k} + c  \frac{\id_p \id_k^\intercal}{k} \frac{W^\intercal}{\sqrt{k} }  \) \times \( \frac{1}{a + \frac{\Delta}{d}} \frac{Y^\intercal Y}{p}  -  d \beta \rI_p  \) 
\end{align*}
In these equations, $a,b,c$ are coefficients that depend on the non-linearity. In the linear case, $c = 0$ and $a = b = 1$.
Let us now assume for instance a non-linearity such that $a,b \neq 0$ and $c = 0$.
Both $\Gamma_p^{uv}$ and $\Gamma_p^{uu}$ can be represented as 
\begin{align}\label{eq_app:gamma_p_nonlinear}
	\Gamma_p &= \left[(a-b) \rI_p + b \frac{WW^\intercal}{k}\right]	M, 
\end{align}
in which $M$ is a symmetric (non necessarily positive or negative) matrix, independent of $W$. In order to perform the same analysis we made in the case of a linear activation function, 
we need in particular to be able to characterize the bulk of such matrices.
Although this might be doable with more refined techniques, this does not seem to come as a direct consequence of the analysis 
of Silverstein and Bai \cite{marvcenko1967distribution, silverstein1995analysis}.
Indeed, one cannot write that the eigenvalues of $\Gamma_p$ are identical, up to $0$ eigenvalues, to the ones of a matrix of the type
\begin{align*}
	\frac{1}{k} W^\intercal M' W,
\end{align*}
which are the types of matrices covered by the analysis of Bai and Silverstein. Moreover, it is not immediate to use results of free probability \cite{voiculescu1992free} in this context.
Indeed, $\Gamma_p$ in eq.~(\ref{eq_app:gamma_p_nonlinear}) is the product of two matrices that are asymptotically free, but $M$ is not positive, which prevents 
a priori the use of the classical results on the $S$-transform of a product of two asymptotically free matrices. Writing $\Gamma_p$ as the sum of $(a-b)M$ and $b (W W^\intercal) M /k$ does not yield any obvious results 
either, as these two matrices are not asymptotically free. 
For this reason, and although there might exist techniques to study the bulk of the matrix of eq.~(\ref{eq_app:gamma_p_nonlinear}) and the transition in its largest eigenvalue, 
this is left for future work.

%% file: files/supplementary/plots_vu.tex
Despite we illustrated the main part mostly with the Wigner model, in this section we present phase diagrams for the Wishart model. We show in particular a heat map of $\textrm{MMSE}_v$ as a function of the noise to signal ratio $\Delta/\rho_v^2$ for linear, sign and relu activation functions in Fig.~\ref{appendix:fig_map_mse_delta_alpha_vu}. The white dashed line marks the critical threshold $\Delta_c$, given in the Wishart model by eq.~\eqref{appendix:stability_threshold:wishart}, while the the dotted line shows the critical threshold of reconstruction for PCA.
 Besides we show also the mean squared error as a function of the noise variance for larger values of $\alpha$ in Fig.~\ref{appendix:fig_mse_u_VU}. The $\textrm{MMSE}_v$ has been obtained solving the state evolution equations eq.~\eqref{appendix:se_from_amp_bayes_uv}, that show as well an unique stable fixed point for the large range of values that we studied, initializing with either informative or random conditions.
  Finally we illustrate the LAMP algorithm for the linear activation in the Wishart model with $\alpha= \beta = 1$, and compare it to classical PCA and AMP algorithms. We show the comparison in Fig.~\ref{appendix:bbp_lamp_amp_se}  and we added their corresponding state evolutions.
 \begin{figure}[!htb]
	\centering
		\includegraphics[width=1.0\linewidth]{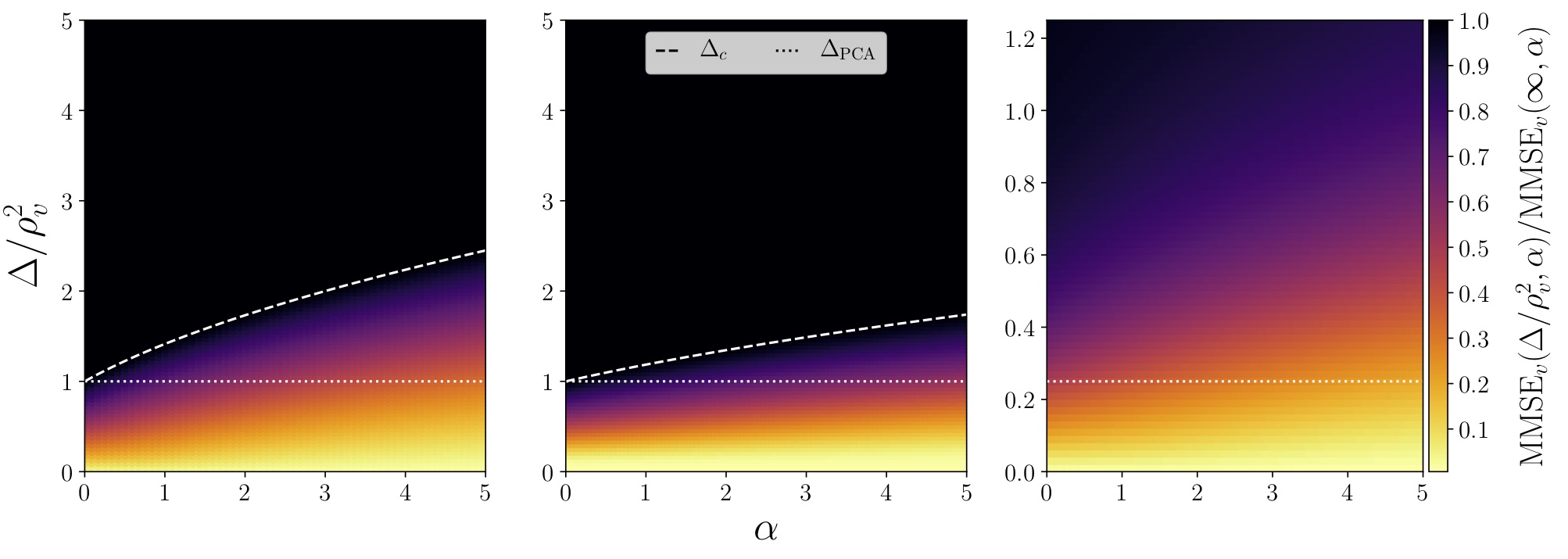}
	\caption{Spiked Wishart model: ${\rm MMSE}_v$ on the spike as a function of noise to signal ratio $\Delta/\rho_v^2$, and generative prior (\ref{gen_single}) with
          compression ratio $\alpha$ for linear  (left), sign
          (center), and relu (right) activations at $\beta=1$. Dashed white lines
          mark the phase transitions $\Delta_c$, matched by both the
          AMP and LAMP algorithms. Dotted white line
          marks the phase transition of canonical PCA.}
	\label{appendix:fig_map_mse_delta_alpha_vu}
\end{figure}
\begin{figure}[!htb]
	\centering
	\includegraphics[width=1\linewidth]{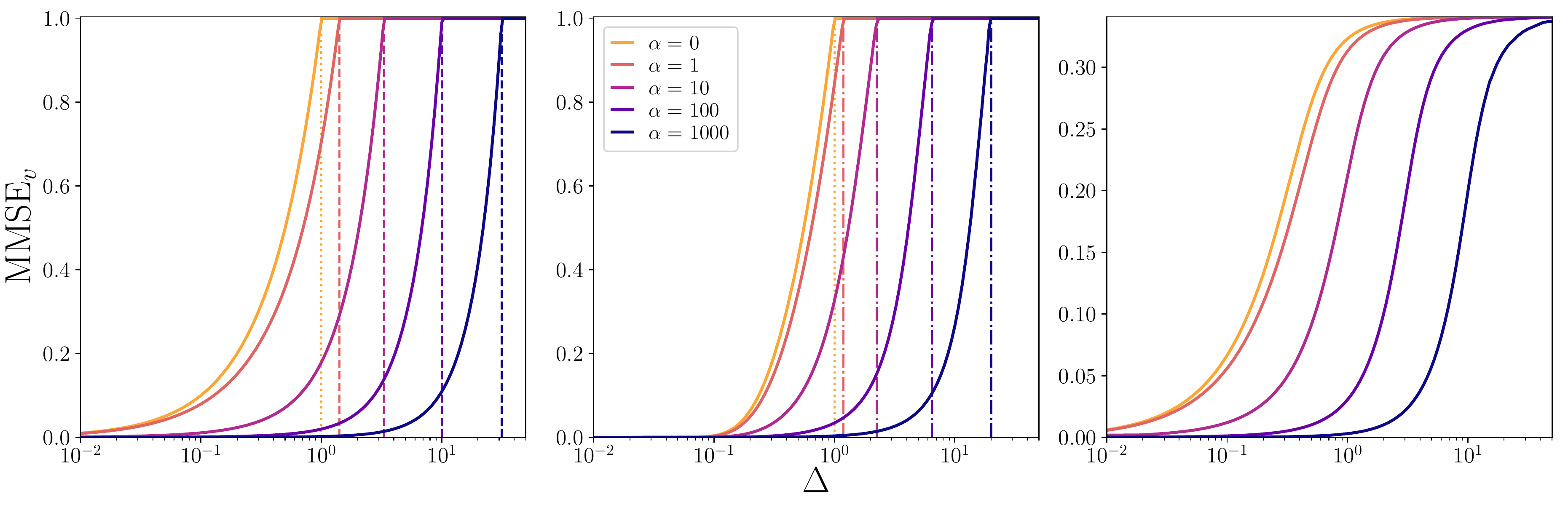}
	\caption{Spiked Wishart model: ${\rm MMSE}_v$ as a function of
          noise $\Delta$ for a wide range of compression ratios
          $\alpha=0,1,10,100,1000$, for linear (left), sign
          (center), and relu (right) activations, at $\beta=1$.}
	\label{appendix:fig_mse_u_VU}
\end{figure}
\begin{figure}[!htb]
	\centering
	\includegraphics[width=0.33\linewidth]{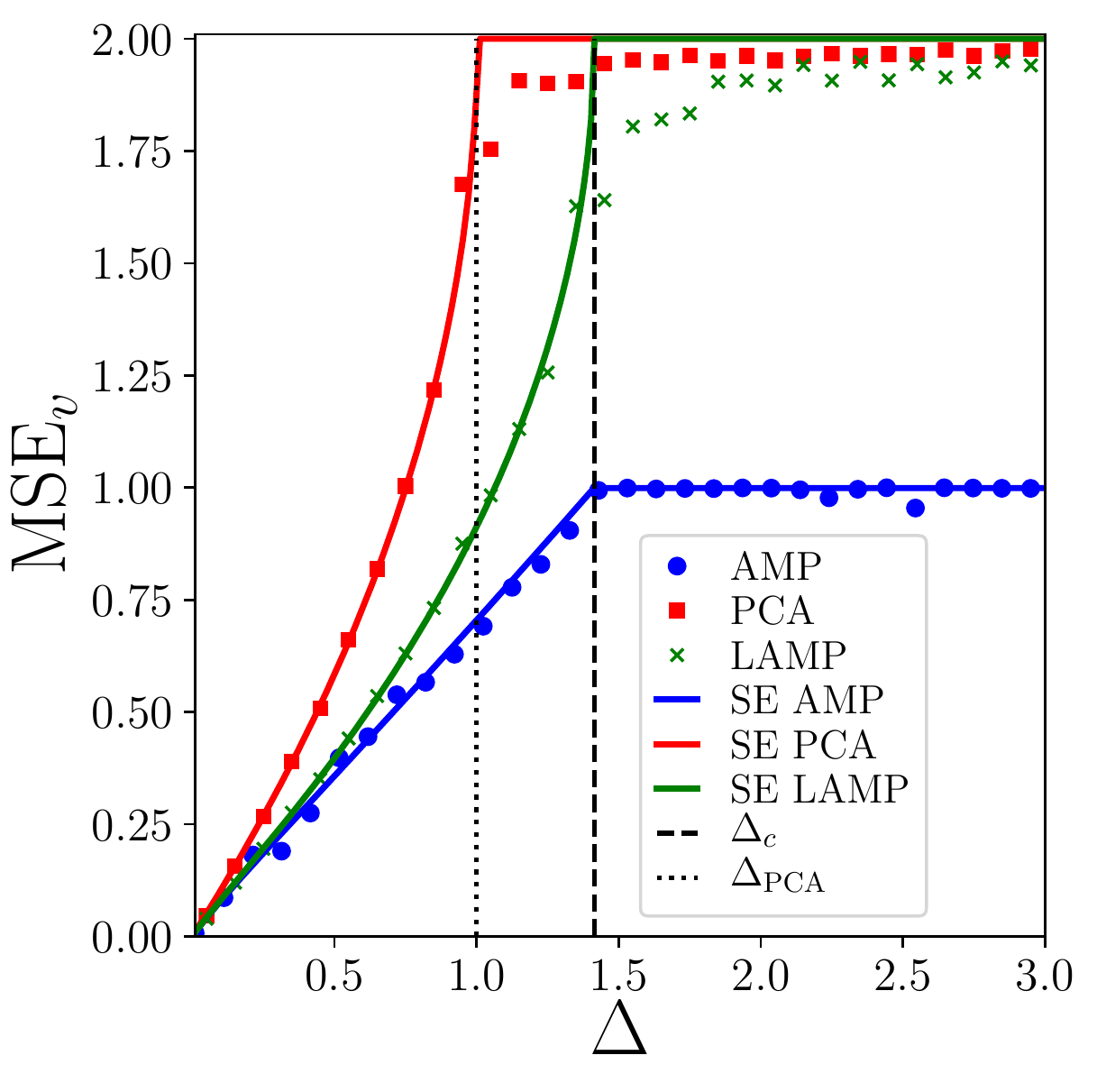}
	\caption{Spiked Wishart model: Comparison between PCA, LAMP and AMP for
          the linear activation at $\beta=1$ and compression ratio
          $\alpha=1$. Lines correspond to the theoretical asymptotic
          performance of PCA (red line), LAMP (green line) and AMP
          (blue line). Dots correspond to simulations
          of PCA (red squares), LAMP (green crosses) and AMP (blue
          points) for $k=10^4$, $\sigma^2=1$.}
    \label{appendix:bbp_lamp_amp_se}
\end{figure}